\documentclass[12pt,fleqn]{amsart}
\usepackage{latexsym}
\usepackage{amsmath}
\usepackage{amssymb}
\usepackage{epsfig}
\usepackage{amsfonts}
\usepackage{graphics}    

\usepackage[margin=2cm]{geometry}
\usepackage{tikz}
\usepackage{subcaption}
\usetikzlibrary{positioning}
\usetikzlibrary{calc}

\usepackage[utf8]{inputenc}

\usepackage[english]{babel}
\usepackage{comment}

\usepackage{amsmath, amssymb, textcomp, amsthm}
\usepackage[T1]{fontenc}

\usepackage{xcolor}

\usepackage{graphicx} 

\usepackage{hyperref}
\usepackage[capitalize]{cleveref}

\usepackage[stable]{footmisc}  

\usepackage{xargs}     

\usepackage{tikz}
\usetikzlibrary{matrix}
\usetikzlibrary{positioning}
\usetikzlibrary{cd}
\usetikzlibrary{arrows}

\newtheorem{lemma}{Lemma}[section]
\newtheorem{proposition}[lemma]{Proposition}
\newtheorem{theorem}[lemma]{Theorem}
\newtheorem{corollary}[lemma]{Corollary}
\newtheorem{remarks}[lemma]{Remarks}
\newtheorem{definitions}[lemma]{Definitions}
\newtheorem{examples}[lemma]{Examples}
\newtheorem{comments}[lemma]{Comments}
\newtheorem{observation}[lemma]{Observation}
\newtheorem{observations}[lemma]{Observations}

\makeatletter
\newcommand{\dashedrightarrow}[1][2pt]{%
  \settowidth{\@tempdima}{$\longrightarrow$}\rightarrow
  \makebox[-\@tempdima]{\color{white}\rule[0.5ex]{#1}{1pt}\,}
  \phantom{\longrightarrow}
}
\makeatother
\usepackage{fancyhdr}

\usepackage{ifthen}
\newcommand{\stacks}[2][]{\cite[\href{http://stacks.math.columbia.edu/tag/#2}{\ifthenelse{\equal{#1}{}}{Tag #2}{#1, Tag #2}}]{stacks-project}}  

\let\phi\varphi
\let\varphi\phi

\usepackage{fancyhdr}

\renewcommand{\tilde}{\widetilde}
\renewcommand{\bar}{\overline}

\usepackage{mfirstuc}  
\usepackage{listings}  

\mathchardef\mhyphen="2D

\let\subset\subseteq

\begin{document}
          \def\square{\Box}
           


\title[Whittaker groups and hyperelliptic curves]{ Whittaker groups and hyperelliptic curves \\ \mbox{ } \\
{{\it dedicated to the memory of Harm H.~Voskuil}}}
\fancyhead[EC]{Marius van der Put and Jaap Top}
\fancyhead[OC]{Whittaker groups and hyperelliptic curves}
\author{Marius van der Put and  Jaap Top}
\address{Bernoulli Institute,  Nijenborgh 9,
9747 AG~Groningen, the Netherlands.}
\email{m.van.der.put@rug.nl, j.top@rug.nl}
\date{\today}

\begin{abstract} Let $K$ be a complete, non-archimedean valued field with a residue field of characteristic different from $2$. A Whittaker group $\Gamma$ is a discontinuous subgroup of $PGL(2,K)$, freely generated by elements $s_0,\dots ,s_g$ of order two, each defined by a pair of fixed points 
$\{a_0,b_0\},\dots ,\{a_g,b_g\}$. These fixed points are called ``in good position''.

 A subgroup $W\subset \Gamma$ of index 2 is a Schottky group and produces a
hyperelliptic Mumford curve $\Omega/W\rightarrow \Omega/\Gamma \cong \mathbb{P}^1$, called {\it Whittaker curve},
 of genus $g$ and with branch locus $\mathbb{B}\subset \mathbb{P}^1(K)$.

An explicit parametrization of Whittaker curves in terms of theta functions for $W$ and $\Gamma$ and the data
of the fixed points, is developed. In particular, this allows one to express the branched points (and other data such as $p$-adic periods and $p$-adic heights)
 in terms of values of theta functions. 

 A central theme of this paper is the relation between the fixed points and the branch locus. For a given configuration
 $(P,m)$ of $g+1$ pairs of points in $\mathbb{P}^1$, one defines a rigid space  $Fix_{P,m}$ of fixed points in good position with that configuration and  a rigid space of branched points $Branch_{P,m}$ in that configuration. A main result is that the
 natural morphism $FB\colon Fix_{P,m}\rightarrow Branch_{P,m}$ is a rigid \'etale  covering with Galois group
 $\{\pm 1\}^{d-1}$ for some $d\geq 1$.  

For all cases of genus $g=2,3$ (and for some more),  an approximation of $FB$ is computed which confirms the
main result.

Classification of Whittaker groups and analytic reductions of Whittaker curves  is another important issue of this paper. 
The background material in this paper complements the work of L.~Gerritzen, G.~Van Steen, F.~Herrlich  and others.
It involves re-examination of some proofs, the derivation of  properties of  semi-stable analytic reductions and  studying
 good position of fixed points.  
\end{abstract}

\maketitle

\section{\rm Introduction and Summary}

 $\ \ \ $ E.T.~Whittaker (1899 and 1929), \cite{Wh1, Wh2}, studied the uniformization of  hyperelliptic curves over $\mathbb{C}$ with the aim of finding curves having
 an explicit uniformizing Fuchsian subgroup of $PGL_2(\mathbb{R})$ and an explicit equation as well. This topic is still active.
 
 Whittaker groups $\Gamma \subset PGL_2(K)$, the analogues over a non-archimedean valued field $K$, were introduced 
 in \cite{P, G-vdP}. They are natural examples for Schottky groups and Mumford curves. 
  There is a renewed interest in Whittaker groups because  the uniformization of hyperelliptic Mumford curves (here called Whittaker curves)
 is of importance for the explicit computations of $p$-adic periods, $p$-adic heights and a $p$-adic BSD conjecture (see also \cite{T}).   
In collaboration with one of us, the team of researchers  E.~Kaya, M.~Masdeu and  J.S.~M\"uller \cite{KMMP} aims to develop algorithms and examples for these
$p$-adic heights. This is  the direct inspiration for the current manuscript, and some of the results obtained here have been exploited in their
work.\\

In the present paper we also revisit \cite{P,G-vdP,F-vdP} in order to complete proofs, to correct mistakes  and to answer old and new questions.
Proofs of statements which can easily be found in \cite{G-vdP} and/or \cite{F-vdP} are not reproduced. 
Beyond that, the  important 
themes are: explicit formulas (Propositions~\ref{3.1} and \ref{3.2} and \ref{3.3}) for the uniformization of Whittaker curves; a result
(Theorem~\ref{new5.8}) explaining  the relation between fixed points and 
branch points and  finally (Sections~\ref{explicit1} and \ref{explicit2}) explicit examples for genus 2 and 3. We now describe in  more detail the contents of each following section.\\
 
 The main characters are introduced in \textit{Section~\ref{Section2}}, namely:  the field $K$, good position for fixed points, Whittaker groups $W\subset \Gamma$, 
 Schottky groups, ordinary points $\Omega\subset \mathbb{P}^1_K$, (potential) Mumford curves. It continues with an overview of 
analytic reductions of some rigid spaces, e.g., $\mathbb{P}^1_K$ and its open subspaces, and curves over $K$. 

 For a pair $(X,S)$ with $X/K$ a curve and $S\subset X(K)$ a finite set, one shows that (after a finite separable extension of $K$, see
Corollary~\ref{1.2}) 
 there  exists a {\it minimal semi-stable reduction, that separates the points $S$}. If genus$(X)+\# S\geq 3$, then this reduction is unique 
 and will be denoted by $\overline{(X,S)}$. 
 
  A {\it criterion \textrm{(Theorem~\ref{1.4})} for a curve $X$ over $K$, to be  a Mumford curve}, is given a ``complete proof''. This complements the incomplete   proof in \cite{G-vdP, F-vdP} of the following  statement (which is part of the proof):
\begin{quote}
``A one-dimensional space $\Omega$, having a semi-stable reduction $\overline{\Omega}$ which is a tree of projective lines
over the residue field (and satisfies a natural  condition on the nodes), has an embedding into $\mathbb{P}^1_K$ as complement of a
compact subset of $\mathbb{P}^1(K)$. This embedding  is unique up to the action of $PGL_2$ on $\mathbb{P}^1$.''
\end{quote}
We note that the above statement is a variation on \cite[\S 4]{Mu}, where a certain formal scheme over a complete 
local ring is ``embedded'' in $\mathbb{P}^1$. 
More criteria concerning Whittaker curves are given. For example (combining Theorems~\ref{1.4}, \ref{1.5} and \ref{1.7}):  
\begin{quote}
``Let $X/K$ be a hyperelliptic curve and $S$ its set of Weierstrass points (this coincides with the fixed points of the hyperelliptic involution).
Then $X$ is a Whittaker curve if and only if the semi-stable reduction for  the pair $(X,S)$ exists over $K$ and all its irreducible
components are projective lines over the residue field.''
\end{quote} 

In \textit{Section~\ref{Section3}}, one is  given a hyperelliptic curve $X\rightarrow \mathbb{P}^1_K$ with ramification locus  $S\subset X(K)$ and branch locus $\mathbb{B}\subset \mathbb{P}^1(K)$. Over a suitable extension of the field $K$, the obvious ramified covering of 
 $\overline{(\mathbb{P}^1,\mathbb{B})}$ of degree two, produces explicitly, the  semi-stable reduction $\overline{(X,S)}$. 
 It is shown in Corollary~\ref{twist} that a Whittaker curve over a field $K$ is determined by its branch locus in $\mathbb{P}^1(K)$. This extends to arbitrary genus a result of
 Teitelbaum \cite{T}, who concluded this for genus~$2$.

    In \textit{Section~\ref{Section4}}, starting from the morphism $\overline{(X,S)}\rightarrow \overline{(\mathbb{P}^1,\mathbb{B})}$ and the corresponding map of 
   the dual  graphs  $\overline{(X,S)}^d\rightarrow \overline{(\mathbb{P}^1,\mathbb{B})}^d$,  {\it fundamental domains}
   for the Whittaker group $\Gamma$ are constructed. On the dual graphs this is done by omitting for each pair of conjugated edges and each pair of 
   conjugated vertices one of the items. The result is a (special) maximal subtree of $\overline{(X,S)}^d$.
   Let $F_0$ be a connected component of the  preimage of this maximal subtree in the universal covering $\overline{\Omega}^d$ of $\overline{(X,S) }^d$. 
 Now $F_0$ and the corresponding spaces $F$ and  $\overline{F}$  in $\Omega \subset \mathbb{P}^1$ and $\overline{\Omega}$ 
 are called {\it fundamental domains}.
 
The fundamental domain $F$ is a connected affinoid subspace of $\Omega \subset \mathbb{P}^1$. The intersection of $F$
 with the set of all fixed points for $\Gamma$  is the union of $g+1$ pairs of points $\{a_i,b_i\}_{i=0}^g$. Let $s_i\in PGL_2$
 be the order 2 element with fixed points $a_i,b_i$. It is shown  (Theorem \ref{2.2}) that {\it  $\Gamma =<s_0>*\dots *<s_g>$} is
the Whittaker group corresponding to $X$.
 The action of the generators $s_i$ of $\Gamma$ on the universal covering $\overline{\Omega}^d$ of $\overline{(X,S)}^d$ is 
 made explicit in three  typical cases.  

In  \textit{Section~\ref{Section5}}  these examples are presented, and the notion of {\it configuration} is introduced.  
 Let $\mathbb{B}=\{a_0,b_0,\dots ,a_g,b_g\}\subset \mathbb{P}^1$ be the branch locus of a Whittaker curve. Its dual graph
 $P:=\overline{(\mathbb{P}^1, \mathbb{B})}^d$ is a tree  with a map $m$ from the set of symbols $\{a_0,b_0,\dots ,a_g,b_g\}$ to the vertices
of $P$, satisfying 
 certain properties (see Definitions~\ref{2.4} for details).\\
  We will call a finite tree $P$ and a map $m$ from symbols $\{a_0,b_0,\dots ,a_g,b_g\}$ to its vertices satisfying these properties,
 {\it a configuration $(P,m)$ of genus $g$.} 
 Note, in passing, that this notion is close to what is called ``clusters'', see \cite{D-D-M-M}.\\
 
 To $(P,m)$ we associate a space $Branch_{P,m}$ consisting of all branch loci $\mathbb{B}\subset \mathbb{P}^1$
 (considered up to conjugation by $PGL_2$) such that $\overline{(\mathbb{P}^1, \mathbb{B})}^d$ is a configuration of type $(P,m)$.  
 This is a rigid analytic space in a natural way.\\

  Fixed points $\{a_0,b_0,\dots , a_g,b_g\}\subset \mathbb{P}^1$ for involutions  produce a point of $Branch_{P,m}$ for a 
  certain configuration. There  is a notion ``restricted'' for fixed points (see Definitions~\ref{2.4}) and this leads to 
  a rigid analytic subspace $Fix_{P,m}$ of $Branch_{P,m}$.
  It is shown in Section~\ref{Section7}, that ``restriction'' for  fixed points implies ``good position''. The map $FB\colon Fix_{P,m}\rightarrow Branch_{P,m}$ which attaches to a set of
  restricted fixed points  the branch points of the Whittaker curve is a well defined rigid analytic map.  This leads to a main result  (Theorem~\ref{new5.8}):
 \begin{quote} 
{\it    $FB\colon Fix_{P,m}\rightarrow Branch_{P,m}$ is an unramified, rigid analytic, Galois covering with Galois group $\{ \pm 1\}^{d-1}$. 
   Here $d\geq 1$ is defined by: $2^d$ is the number of fundamental domains for the configuration $(P,m)$.}
 \end{quote}

  One considers  fixed points $\{a_i,b_i\}_{i=0}^g$ in good position, the Whittaker group $W\subset \Gamma$ and the resulting
   hyperelliptic curve $X$  with affine equation 
 \[y^2=c(x-w_1)\cdots (x-w_d) \mbox{ with $d=2g+2$ or $d=2g+1$}.\]
 Here $x$ stands for the theta function $\Theta:=\prod _{\gamma \in \Gamma}\frac{z-\gamma (a)}{z-\gamma (b)}$  (for suitable $a,b$)
 for the group $\Gamma$. The $\{w_i\}$ are the images of the $\{a_i,b_i\}$ under $\Theta$. This produces an explicit formula for
 the map $FB\colon Fix_{M,p}\rightarrow Branch_{P,m}$.  The function $y$ is seen as $W$-invariant function on $\Omega$. 
 In {\it Section} 6, a delicate expression for $y$ as product of theta functions for the action of $W$ on $\Omega$  is given for both cases
 $d=2g+1$ and $d=2g+2$. \\

 In \textit{Section~\ref{Section7}} we present the information on ``good position for sets of fixed points  $\mathbb{F}=\{a_i,b_i\}_{i=0}^g$'' 
 that we are aware of. First of all there is a {\it necessary condition} that can be stated as:  $\mathbb{F}$ satisfies the properties for a
  set of branch points of a Whittaker curve (i.e., $\mathbb{F}\in Branch_{P,m}$ for some $(P,m)$).
  A {\it sufficient condition} is the ``closed disks condition'' which says that the $g+1$ pairs are separated by 
  closed disks (this translates into a property of the configuration $(P,m)$).
  
   A ``closed formula'' for good position seems impossible, even for $g=2$.
  However one can work out an algorithm deciding ``good position''. 
  It is shown (Proposition~\ref{5.2}), that a restricted set of fixed points is in good position. This implies that the map 
  $Fix_{P,m}\rightarrow Branch_{P,m}$ is well defined.

    The problem of characterizing good position is a special case of the embedding problem of finite trees 
 of finite groups into the (generalized) Bruhat-Tits tree of $PGL_{2,K}$.  We give an overview of the related papers \cite{S,He,vdP-V}.\\

 In \textit{Section~\ref{explicit1}},  formulas for $FB\colon Fix_{P,m}\rightarrow Branch_{P,m}$ are made explicit for the three cases
 with $g=2$. A suitable approximation of the formula of $FB$ shows that $FB$ is an unramified covering of
 degree 1 or 2 for the three cases. This provides an independent proof  of Theorem~\ref{new5.8}  for genus 2.
 
\textit{Section~\ref{explicit2}} lists all configurations $(P,m)$ of genus 3. For all cases (and some more, see Proposition~\ref{9.1}), the statements of 
 Theorem~\ref{new5.8} obtain in this way a proof by computation.

 \section{\rm Whittaker groups, analytic reductions}\label{Section2}

\subsection{\rm Whittaker groups and Schottky groups}\label{subsec2.1} 
$K$ denotes a complete, non-archimedean valued field. Let $K^{oo}$ denote the maximal ideal of the valuation ring $K^o$ of $K$. In the case of hyperelliptic curves and Whittaker groups, we assume that
the characteristic of the residue field $k=K^o/K^{oo}$ is different from 2. Some results discussing the remaining
case, namely residue characteristic $2$, one finds in G.~Van Steen's work \cite{S}, \cite{S2}.
If needed or convenient, we will  suppose that the valuation is discrete and write $\pi$ for a generator $K^{oo}$.\\
 
Let $S=\{a_0,b_0,\dots , a_g,b_g\}\subset \mathbb{P}^1(K)$ be a set of $2g+2$ points, considered as the union of the 
$g+1$ pairs $\{a_i,b_i\}, \ i=0,\dots ,g$.   Let $s_i\in PGL(2,K)$
be the element of order two with fixed points $\{a_i,b_i\}$.   The set $S$ is said to be in {\it good position} if the group $\Gamma$ 
generated by $\{s_0,\dots s_g \}$ is discontinuous and is isomorphic to the free product $\langle s_0\rangle *\cdots *\langle s_g\rangle$.
In the sequel we will omit the well-known case $g=1$ that leads to Tate's elliptic curve, and suppose $g\geq 2$.\\

{\it Let $S$ in good position be given}. Write $w_i:=s_is_0$ for $i=1,\dots ,g$. By assumption, there are no relations between the elements
$\{w_i\}_{i=1}^g$. The group $W$, freely generated by $\{w_i\}_{i=1}^g$, is also discontinuous.\\

 We recall that a {\it Schottky group $G$ over $K$} is a finitely generated subgroup of $PGL(2,K)$ which is discontinuous and has no elements, $\neq 1$,  of finite order. It is known that $G$ is a free group on $g\geq 1$ generators. 
 
 Here we consider the case $g\geq 2$. The
 set $\Omega \subset \mathbb{P}^1_K$ of ordinary points for $G$ is a rigid open subspace.  The quotient $\Omega/G$ is the analytification of an algebraic curve over $K$ of genus $g$ (see \cite{G-vdP,F-vdP} for details). This curve is called the  {\it Mumford curve} over $K$, attached to the Schottky group $G$.  
An algebraic curve $X$ over $K$ is called a {\it potential Mumford curve}  if there is a finite field extension $L\supset K$ such that $X\times _KL$
 is a Mumford curve.\\

 Thus, by definition, the above group $W=<w_1,\dots ,w_g>$ is a Schottky group. In this special situation, we will call $W$ (or, maybe better, the pair $W \subset \Gamma$) a {\it Whittaker group}. 
We call $S$ a set of {\it fixed points}.\\

Since $W$ has index 2 in $\Gamma$, the two subgroups of $PGL(2,K)$ have the same set of limit points 
$\mathcal{L}$. This set is compact and has no isolated points. Since $\mathbb{P}^1(K)$ is invariant under $\Gamma$, one 
has $\mathcal{L}\subset \mathbb{P}^1(K)$. The complement $\Omega:=\mathbb{P}^1_K\setminus \mathcal{L}$ has the structure of a rigid analytic space.

   {\it Example}. For the case $K=\mathbb{Q}_p$, one can construct Whittaker groups  where
$\mathcal{L}=\mathbb{P}^1(\mathbb{Q}_p)$ and $\Omega$ is the well-known '' $p$-adic upper half plane''.

\subsection{\rm  Analytic reductions and subspaces of \texorpdfstring{$\mathbb{P}^1_K$}{P1K}}\label{subsec2.2}

{\it  Here we review  some sections  of \cite{F-vdP}, especially \S 4.8-4.9. } 

The standard affinoid algebra $T_n:=K\langle z_1,\dots z_n\rangle$ consists of the
power series, in $n$ variables, $f=\sum _{m_1,\dots ,m_n\geq 0} a_{m_1,\dots , m_n}z_1^{m_1}\cdots z_n^{m_n}$ with $\lim a_{m_1,\dots ,m_n}=0$. The norm 
$\|\ f \|$ of $f$  is $\max (|a_{m_1,\dots , m_n}|)$. An affinoid algebra $A$ over $K$ is a finite extension of some $T_n$. If the affinoid algebra is reduced, then
it has a unique norm, called the spectral norm $\|\  \|_{sp}$, which satisfies $\| F^n\|_{sp}=\|F\|_{sp}^n$ for all $F\in A, \ n\geq 1$.\\ 

Suppose that $A$ is reduced and that the spectral norm $\| \ \|_{sp}$ takes its values in $|K|$.  Put $A^o=\{a\in A\ |\ \|a\|_{sp}\leq 1\}$, $A^{oo}=\{a\in A\ |\ \|a\|_{sp}<1\}$,
then $\overline{A}=A^{o}/A^{oo}$ is a finitely generated, reduced algebra over $k=K^o/K^{oo}$.

 The affinoid space $Sp(A)$, associated to $A$, is the set of the maximal ideals
of $A$. The canonical reduction $Red\colon Sp(A)\rightarrow Max(\overline{A})$ from $Sp(A)$ to the maximal ideal space
of $\overline{A}$ is given by $x\in Sp(A)$ is mapped to the image of $x\cap A^o$ in $\overline{A}$ (which is a maximal ideal). For 
the special case $Sp(T_n)=\{ (z_1,\dots, z_n)\ | \ \mbox{all } |z_i|\leq 1\}$ one has $\overline{T_n}=k[z_1,\dots ,z_n]$ and $Red$ sends $(a_1,\dots ,a_n)\in K^n$ to 
$(\bar{a}_1,\dots , \bar{a}_n)\in k^n$. Note that $Red$ is not a morphism of schemes!  \\
  
The affine formal scheme $Spf(A^o)$, corresponding to the affinoid space $Sp(A)$ is associated 
 to $A^o$ as projective limit $\underset{\leftarrow}{\lim}\  A^o/\pi ^mA^o$.  Here, we suppose for convenience that $K$ has discrete valuation. 
The underlying topological space of $Spf(A^o)$ is $Spec(\overline{A})$.\\

A rigid space $X$ can be defined as a locally finite union of affinoid spaces. Algebraic varieties over $K$ have 
the (essentially unique) structure of a rigid space.  If $X$ is given a pure affinoid covering
$\{ Sp(A_i) \}$, then this induces a reduction map $Red\colon X=\cup _i Sp(A_i)\rightarrow \cup _i Max(\overline{A}_i)$. 

The term ``pure''  denotes the condition that the  $\{Max (\overline{A}_i)\}$ glue in an obvious way to a variety over $k$. 
A pure affinoid covering induces the formal structure  obtained by gluing the affine formal spaces $Spf(A_i^o)$.
 Its underlying (locally finite) $k$-variety is $\cup \ Spec(\overline{A_i})$. 
On the other hand, a structure on $X$ as formal scheme over $K^o$ induces a pure affinoid covering of $X$. \\

Suppose that for a variety $X/K$ has a model $\mathcal{X}$ over $K^o$ is given (assume that  the valuation of $K$
is  discrete). Then the limit of the spaces $\mathcal{X}\otimes K^o/(\pi ^m)$ defines a formal structure on $X$. \\

The reduction $Red\colon X\rightarrow Z$ of a 1-dimensional rigid space $X$ is called {\it semi-stable} if the only 
singularities of $Z$ are ordinary double points (nodes). The reduction is called {\it stable} if it is semi-stable
and there is no irreducible component, isomorphic to a projective line over $k$, and meeting at most
two other irreducible components. The stable reduction of an algebraic curve (of genus $\geq 2$) is unique.  Any semi-stable
 reduction as obtained by (successive) blowing up point(s) of the stable reduction.\\

\noindent {\it Examples of reductions}.\\
(a). The pure covering $\{ z|\ |z|\leq |\pi |\}, \{z|\ |\pi|\leq |z|\leq 1\}, \{ z| \ |\frac{1}{z}|\leq 1\}$ of $\mathbb{P}^1_K$ defines a reduction
$Red\colon \mathbb{P}^1_K\rightarrow Z$, where $Z$ is the union of two projective lines $L_1,L_2$ over $k$, intersection normally in one point $p$.
Now $Red^{-1}(p)=\{z |\ |\pi| <|z| < 1\}$ and for other points of $Z$, the preimage is an open disk. A general result is:  Every reduction $Z$ of $\mathbb{P}^1_K$ is a finite tree of projective lines over $k$ meeting in ordinary double points (nodes).\\

\noindent (b). Let $A\subset \mathbb{P}^1(K)$ be a finite set with $\# A\geq 3$. There is a unique minimal reduction 
$Red\colon \mathbb{P}^1_K\rightarrow \overline{(\mathbb{P}^1_K   ,A)}$ that separates $A$. This means that $Red$ maps the points
of $A$ to distinct points and also distinct of the nodes.   The minimality means that every irreducible component of the reduction has at least 3
special points, namely the images of $A$ and the nodes. We note that a reduction can be made ``minimal'' by contracting, one by one, irreducible components which have less than three special points. \\

\noindent (c). \label{geval(c)} Let $\mathcal{L}\subset \mathbb{P}^1(K)$ be a compact, perfect set, i.e., $\mathcal{L}$
has no isolated points. Then $\Omega=\mathbb{P}^1_K$ is a rigid space over $K$. There is a unique 
minimal reduction $Red\colon \Omega \rightarrow \overline{\Omega}$. The $k$-variety is a tree of projective lines over $k$
and lines intersect in $k$-rational nodes $d$ with complete local ring $k[[x,y]]/(xy)$. The local ring at $d$ of the 
associated formal scheme is $K^o[[x,y]]/(xy-a)$ for some  $a\in K$ with $0<|a|<1$.  We will call $|a|$ the {\it size} of the double point.\\ 

One can construct this reduction as follows.  For any three  distinct points $p=(p_0,p_1,p_\infty)$ in $\mathcal{L}$, there
is a unique reduction $Red_p\colon \mathbb{P}^1_K\rightarrow \mathbb{P}^1_k$ such that 
$Red_p$ sends $p_0,p_1,p_\infty$ to $0,1,\infty \in \mathbb{P}^1(k)$. Any finite number of these
reductions can be combined into a reduction. The number of reductions is countable and by taking a limit,
taken over all finite subsets, one obtains the canonical reduction $\overline{\Omega}$.  We note that every
component of this reduction is a projective line over $k$ and meets the other components in at least three points.\\

\noindent (d).  Let $S\subset \mathbb{P}^1(K)$ be a countable set such that the topological closure $\overline{S}$
of $S$ is compact. Assume that the points of $S$ are isolated in $\overline{S}$ and that 
$\mathcal{L}= \overline{S}\setminus S$ is perfect. Then the rigid space $\Omega=\mathbb{P}^1_K \setminus \mathcal{L}$ has a unique minimal reduction $Red\colon \Omega \rightarrow \overline{(\Omega, S)}$
which separates $S$.  \\

\subsection{\rm  Semi-stable reductions of Mumford curves}

\begin{theorem}{\rm (Stable reduction theorem)}. \label{1.1} Let $X/K$ be a smooth, connected, projective curve of genus $\geq 2$. There is a finite, separable extension $K'\supset K$ such that $K'\times _KX$ has a stable reduction. 
This reduction is unique.
\end{theorem}

The existence of a  semi-stable reduction is a highly nontrivial result and has various proofs in the literature.  It can be seen that the stable reduction is unique.  Any semi-stable reduction is obtained by blowing up (successively) points on the stable reduction. 

We will review an explicit proof of the stable reduction theorem for hyperelliptic curves over $K$, 
with residue characteristic $\neq 2$. A useful extension of the stable reduction theorem is the following.

\begin{corollary}\label{1.2}   Let $X/K$ be a smooth, connected, projective curve of genus $\geq 2$ and 
$S\subset X(K)$ a finite set.  There is a finite, separable extension $K'\supset K$ such that 
$X':=K'\times _KX$ has a minimal semi-stable reduction $Red\colon X'\rightarrow \overline{(X',S)}$ which separates
the points of $S$. This reduction is unique.
\end{corollary}
 As before, {\it separation} means that the images under $Red$, of the points in $S$, are distinct and 
 are distinct from the nodes of $\overline{(X',S)}$.  Minimality means that there is no irreducible component, isomorphic to the projective line over $k$ which has $\leq 2$ special points. Here, special point means a node or the image of a point
 of $S$.  The proof is obtained by ``blowing up'' points on the given stable reduction in order to ``separate'' $S$.
 In case the valuation of $K$ is discrete, the stable reduction can be seen as a model of the curve over the valuation ring
 $K^o$. This is  a surface and ``blowing up'' is indeed blowing up points on the surface. In the general case
 ``blowing up'' can be realized by refining the pure affinoid covering associated to the given reduction.  \\
 
 For a Mumford curve, the uniformization by a Schottky group provides its stable reduction. 
\begin{proposition}\label{1.3}
 Let $\Gamma \subset PGL_2(K)$ be a Schottky group 
with set of limit points $\mathcal{L}$ and space of ordinary points $\Omega =\mathbb{P}^1_K\setminus \mathcal{L}$.
Then $\Omega/\Gamma$ is the analytification of an algebraic curve $X$ over $K$ (i.e., the Mumford curve). 

Moreover, $\Gamma$ acts on the reduction $Red\colon \Omega \rightarrow \overline{\Omega}$ defined by $\mathcal{L}$.
 The induced map $Red\colon \Omega/\Gamma \rightarrow \overline{\Omega}/\Gamma$  is the stable reduction of $X/K$.
The normalization of every irreducible component of $\overline{\Omega}/\Gamma$ is a projective line over $k$.
\end{proposition}
\begin{proof} The stabilizer of a component or a node of $\overline{\Omega}$ is a bounded subgroup of $PGL_2$. 
Since $\Gamma$ is discontinuous and has no elements of finite order ($\neq 1$), an element $\gamma \in \Gamma,\ \gamma \neq 1$
does not stabilize any node or component of $\overline{\Omega}$. 

 If $\Gamma$ has the property that $\gamma (L)\cap L=\emptyset $ for every irreducible component $L$ of $\overline{\Omega}$
 and  any $\gamma \in \Gamma, \gamma \neq 1$, then one can form the quotient $Z:= \overline{\Omega}/\Gamma$ with respect to the
 Zariski topology. This is done by choosing a fundamental domain $H$ for the action of $\Gamma$ on $\overline{\Omega}$ and gluing 
 this to a finite number of translates $\gamma H$. Then the components of $Z$ are projective lines over the residue field $k$
 and these lines intersect in ordinary nodes.  Moreover $\overline{\Omega}\rightarrow Z$ is the universal covering of $Z$ for the Zariski
 topology.\\
 
 Suppose that $\gamma_1 L_1\cap L_1=\{p\}$ occurs for some irreducible component $L_1$ and $\gamma_1 \in \Gamma$. Then, a quotient
 of $\overline{\Omega}$ by $\Gamma$ can be constructed as follows.  Replace $\Gamma$ by a normal subgroup 
 $N$ of finite index such that $\gamma L\cap L=\emptyset$ for any $L$ and $\gamma \in N, \gamma \neq 1$. Then 
 $Z':=\overline{\Omega}/N$ is well defined and $Z$ is obtained from $Z'$ by taking the quotient for the action of the finite 
 group $\Gamma /N$. A component $L$ of $Z$ can have ordinary double points and the normalization of $L$ is the 
 projective line over $k$.  For instance, the above point $p$ produces a node on the component with is the image of $L_1$ and $\gamma_1 L_1$. 
 
  The components intersect in ordinary double points.  Since $\Gamma$ acts freely, the
 map $Z'\rightarrow Z=Z'/N$ is an \'etale Galois extension. Further one can give $\overline{\Omega}\rightarrow Z$, the 
 interpretation of the universal covering for the \'etale topology. \\  
 
    Suppose that the irreducible component $L$ of $Z$ is nonsingular and meets at most two components.  Then one shows that 
the same holds for an irreducible component of $\overline{\Omega}$ which maps to $L$. The latter is in contradiction with
(c) of \S\ref{subsec2.2}, p.~\pageref{geval(c)}.  Hence the semi-stable reduction $Z$ is in fact stable. \end{proof}

\begin{theorem}\label{1.4}  Let $X/K$ be a smooth, connected, projective curve of genus $\geq 2$. Then $X$ is a Mumford
curve if and only if $X$ has a (semi-) stable reduction $\overline{X}$ over $K$ and the normalization of every
component of $\overline{X}$ is a projective line over the residue field $k$.  
\end{theorem}
{ The ``only if '' part is Proposition~\ref{1.3}.  The proof of the other direction, as presented in \cite{G-vdP,F-vdP}, is incomplete. We indicate the steps in the long proof of the ``if'' part.}
\begin{proof}  Let a semi-stable reduction $\overline{X}$ be given such that the normalization of every  component is a projective line
over $k$. If a component happens to have a double point $p$, then we refine the pure covering of $X$ in order to replace $p$ by a projective line.
This is possible over $K$ except in the case that the formal local ring at $p$ has the form $K^o[[x,y]]/(xy-\pi)$, where the valuation of $K$
is discrete and $\pi$ is a generator of the maximal ideal of $K^o$. In this special case we replace $K$ by $K(\sqrt{\pi})$ and prove the theorem
for $K(\sqrt{\pi})\times _KX$. After that, one descends to $K$. Now we assume that the components of the semi-stable reduction
$\overline{X}$ are projective lines over $k$.\\

The dual graph $\overline{X}^d$ of the reduction $\overline{X}$ has a universal covering 
$T\rightarrow \overline{X}^d$, which is a tree,  and a covering group $\Gamma$.
This lifts uniquely to the universal covering of $Z\rightarrow \overline{X}$ for the Zariski topology and to the universal covering
$\Omega \rightarrow X$ for the rigid topology. By construction $\overline{\Omega}=Z$. Further $\Gamma$ is the covering group of
$\Omega \rightarrow X$ and $\Omega/\Gamma =X$.  The technical part of the proof is to show that:\\

{\it  The space $\Omega$ has an embedding into $\mathbb{P}^1$ such that its complement $\mathcal{L}$ is a compact perfect subset
 of $\mathbb{P}^1(K)$. Moreover, this embedding is unique up to an automorphism of $\mathbb{P}^1$.
 Further, the action of $\Gamma$ on $\Omega$ extends to an action on $\mathbb{P}^1$ and  $\Gamma$
 identifies as a subgroup of  $PGL_{2}(K)$ preserving $\Omega$. Moreover, $\mathcal{L}$ is the set of limit points
 of $\Gamma$. \\  }

\noindent{\it First step}. Fix a point $t_0\in T$ and consider, for any $n\geq 1$, the subgraph 
$\{t\in T\  |\mbox{ distance}(t,t_0)\leq n\}$.
The corresponding subset $\Omega _n$ of $\Omega$ is affinoid and its reduction $\overline{\Omega_n}$ is a semi-stable finite tree
of lines (projective or affine) over the residue field of $K$. The reduction can be completed to $\overline{\Omega_n}^+$ 
and $\Omega_n$ can be enlarged to a rigid  space $\Omega _n^+$ with reduction $\overline{\Omega_n}^+$.  
This implies that $\Omega _n^+$ is a projective curve and its genus equals the genus of $\overline{\Omega_n}^+$, which is zero.
{\it We conclude that $\Omega_n$ has an embedding $f_n$ into $\mathbb{P}^1$ such that the complement of
$f_n(\Omega_n)$ is a finite (disjunct) union of open disks}. \\

\noindent{\it Second step}. We have to change given embeddings $f_n\colon \Omega_n\rightarrow \mathbb{P}^1$ such that these will glue     
to an embedding of $\cup _n\Omega_n =\Omega \rightarrow \mathbb{P}^1$. \\

Consider a connected affinoid subset $A\subset \mathbb{P}^1$, given as the complement of a finite set $h(A)$ of   
disjoint open disk, called the {\it holes} of $A$. There is a short exact sequence, see \cite[2.7.11]{F-vdP},
\[ 1\rightarrow K^*\times (1+O(A)^{oo})\rightarrow O(A)^*\overset{h}{\rightarrow} M_0(h(A))\rightarrow 1,\]
where $O(A)^{oo}=\{a\in O(A) \ |\ ||a||<1\}$,  $O(A)^*$ are the invertible functions and $M_0(h(A))$ is the group of the maps $m\colon h(A)\rightarrow \mathbb{Z}$ with
$\sum _{p\in h(A)}m(p)=0$.  One can verify the following statement: \\

{\it An $f\in O(A)^*$ provides an embedding of $A$ into $\mathbb{P}^1\setminus \{0,\infty\}$ if and only if
$f$ has image $h(f)=\delta_p-\delta _q$ for certain $p,q\in h(A),\ p\neq q$.}\\
 (i.e., $h(f)$ takes the value $1$ and the value $-1$ for $p$ and $q$ and is otherwise $0$).\\
   
Now we return to the admissible covering $\cup_n \Omega _n$ of $\Omega$. The inclusion $\Omega_n\subset \Omega_{n+1}$ induces a natural map $h(\Omega_{n+1})\rightarrow h(\Omega_n)$ sending a hole of $\Omega_{n+1}$, say $|z-1|<|a|$, to the hole of
 $\Omega _n$ of the form $|z-1|<|b|$ with $|a|<|b|<1$. By construction, this map is surjective. 
 
  We observe that for every node  of $\overline{\Omega}$
the local equation has the form $xy-\alpha$ with $0<|\alpha|<1$ and that for some $\pi$ with $|\pi| <1$ one has
$|\alpha |\leq |\pi|$ for all $\alpha$ (because there are finitely many nodes modulo the action of $\Gamma$).  By using the reductions one concludes that
$|a|\leq |\pi|\cdot |b|$. 

Now consider distinct points $p,q$ in $\underset{\leftarrow}{\lim} \ h(\Omega_n)$, given by sequences $p_1,p_2,p_3,\ldots$
and $q_1,q_2,q_3,\ldots$ and a point $a\in \Omega_1$. Choose an embedding $f_n\in O(\Omega_n)^*$ with image
$\delta_{p_n}-\delta_{q_n}$ and $f_n(a)=1$. 
\[ \mbox{This leads to }\frac{f_n}{f_{n+1}}\in 1+O^{oo}(\Omega_n) \mbox{ for all } n.\]   

\noindent{\it Third step}.  A straightforward computation shows  the restriction map $r_n\colon O^{oo}(\Omega_{n+1})\rightarrow O^{oo}(\Omega_n)$
to have norm $\| r_n\|\leq |\pi |$.  This implies that the map 
\[ d\colon \prod _n O^{oo}(\Omega_n) \rightarrow \prod _n O^{oo}(\Omega_n), \quad (a_1,a_2,a_3,\dots )\mapsto (a_1-a_2, a_2-a_3,\dots )\]
is surjective. Indeed, for any sequence $(b_1,b_2,b_3,\dots )\in \prod _n O^{oo}(\Omega_n)$, the associated sequence
$(\sum_{n\geq1} b_n,\sum_{n\geq 2} b_n, \sum_{n\geq3} b_n, \ldots)$ is well defined and is mapped by $d$ to 
$(b_1,b_2,b_3,\dots)$. The same holds if the additive case $O^{oo}$ is replaced by the multiplicative case $1+O^{oo}$.
This produces elements $g_n\in 1+O^{oo}(\Omega_n)$ such that the $g_n\cdot f_n$ glue to an element 
$f\in O(\Omega)^*$ such that, for all $n$, the map $f$ restricted to $\Omega_n$ is an embedding.
{\it  Thus $f\colon \Omega \rightarrow \mathbb{P}^1$ is an embedding. }\\

The complement of $f(\Omega_n)$ consists of a number of holes, containing the holes of the complement of $f(\Omega _{n+1})$. The latter have radii multiplied  by something $\leq |\pi|$. It follows that
$\mathcal{L}:=\cap \ \{\mbox{complement of } f(\Omega_n)\}$ is a compact set and $\Omega$ is isomorphic to
$\mathbb{P}^1\setminus \mathcal{L}$.\\

Suppose that $g\colon \Omega =\mathbb{P}^1\setminus \mathcal{L} \rightarrow \mathbb{P}^1\setminus \{0,\infty\}$ is another embedding. The restriction of $g$ to $\Omega _n$ maps to $\delta_{a_n}-\delta_{b_n}$ in the set of holes $h(\Omega_n)$. The holes $a_n$ and $b_n$
intersect in the limit to points $a,b \in \mathbb{P}^1$. Then $g$ is a rational function on $\mathbb{P}^1$ with divisor $[a]-[b]$.
Thus the embedding $g$ equals $f$ up to an element in $PGL_2(K)$. 

In particular, for any $\gamma \in \Gamma$, the map $\gamma \colon \Omega \rightarrow \Omega \subset \mathbb{P}^1$ extends to a
rational function and in fact to an element in the automorphism group of $\mathbb{P}^1$.  This identifies $\Gamma$ with a Schottky
subgroup of $PGL_2(K)$. Indeed, $\Gamma$ is finitely generated, free and discontinuous since $\Omega$ consists of ordinary points. 
The set of  limit points of the group $\Gamma $ is contained in $\mathcal{L}$
and in fact equal to $\mathcal{L}$ since the quotient $\Omega/\Gamma$ is a complete curve. 
\end{proof}

\subsection{\rm Semi-stable reductions of Whittaker curves}

Let $S=\{ \{a_i,b_i\}_{i=0}^g\}\subset K \cup \{\infty \}$ denote a set of fixed points in good position.  As in \S\ref{subsec2.1}, this defines
the groups   $W\subset \Gamma =\langle s_0,\dots , s_g\rangle \subset PGL(2,K)$, the set of limit points $\mathcal{L}$ and the subspace $\Omega =\mathbb{P}^1_K\setminus \mathcal{L}$ of $\mathbb{P}^1_K$. We call this  special Mumford curve $X:=\Omega/W$ a {\it Whittaker curve}.

The quotient 
$\Omega/\Gamma$ is also a smooth, connected, complete curve over $K$.  The morphism
$X\rightarrow \Omega/\Gamma$ has degree two and Galois group $\Gamma /W=\{1,\sigma\}$.

\begin{theorem}\label{1.5}  With notations as above,
\begin{itemize}
\item[\textrm{\rm (1)}] The set $S$  belongs to $\Omega$.   
 \item[\textrm{\rm (2)}]  $\Omega/\Gamma\cong \mathbb{P}^1_K$. 
\item[\textrm{\rm (3)}] $X:=\Omega/W$ is a hyperelliptic Mumford curve of genus $g$.
\item[\textrm{\rm (4)}] The sets of  ramification points $Ram_X \subset X(K)$ and  of branch points  $Branch_X\subset \mathbb{P}^1(K)$
are the images of $S$ under $\Omega \rightarrow X$ and $\Omega \rightarrow \mathbb{P}^1_K$.\end{itemize}
\end{theorem}
\begin{proof} 
Proof of (1).  Let $x\in \Omega/W$ be a ramification point. Choose a point $y\in \Omega$ with image $x$. Let $\sigma$ denote the hyperelliptic involution of $\Omega/W$. Since $\Omega \rightarrow \Omega/W$ is the universal covering, the map $\sigma$ lifts uniquely to an automorphism
 $\tau$ of $\Omega$ such that $\tau (y)=y$. Now $\tau^2$ is the unique lift of $\sigma^2=id$ satisfying $\tau^2(y)=y$. Thus $\tau^2=id$. Since $\Gamma$ is the free product $<s_0>*\cdots *<s_g>$, 
 $\tau$ is a conjugate $\gamma s_i\gamma ^{-1}$ for some $i$ and some $\gamma \in \Gamma$.   Therefore $\gamma^{-1}y$ is equal to $a_i$ or $b_i$ and one of these points belongs to $\Omega$.  Applying this method to all $2g+2$ ramification points in $\Omega/W$ one obtains that $S\subset \Omega$.\\

\noindent  Proof of  (2). 
 Using cohomology of some rigid sheaves, one can show that $\Omega/\Gamma$ has genus zero.
The points $S\subset \Omega \cap K$ produce $K$-rational points in $\Omega/\Gamma$ and therefore
$\Omega/\Gamma \cong \mathbb{P}^1_K$.  We prefer a more explicit proof.\\

 We will produce a meromorphic function $F(z)$ on $\Omega$ which is $\Gamma$-invariant and has modulo the action of $\Gamma$ one simple zero and one simple pole on $\Omega$. Moreover, we require that $F(z)$
is defined over $K$ and, as a consequence, maps $\Omega \cap \mathbb{P}^1(K)$ to $\mathbb{P}^1(K)$.
Then $F(z)$ induces a meromorphic function $\tilde{F}\colon \Omega/\Gamma \rightarrow \mathbb{P}^1_K$ with one simple zero and one simple pole. Thus $\tilde{F}$ is an isomorphism (defined over $K$).\\

 For points $a,b\in \Omega \cap K$, such that $\Gamma$-orbits of $a,b,a_0,b_0,\dots a_g,b_g$ and $\infty$ are all disjoint, 
 one considers the {\it theta function} $F(z)=\prod _{\gamma \in \Gamma} \frac{z-\gamma a}{z-\gamma b}$ for the group $\Gamma$. This product converges since, for any $\epsilon >0$, the number of $\gamma \in \Gamma$ with $|\gamma (a)-\gamma(b)|\geq \epsilon $ is finite. (see \cite{G-vdP} or \cite{P}). Moreover  there is, for every $\delta \in \Gamma$, a constant $c(\delta)\in K^*$ with $F(\delta z)c(\delta)=F(z)$. Clearly  $\delta \mapsto c(\delta)$ is a homomorphism and therefore
 $c(s_0),\dots ,c(s_g)\in \{\pm 1\}$. 
 
 In general, the ``automorphy factor'' $c(\delta)$ depends analytically on $a,b$. For $a$ and $b$ close enough, one has
 $|c(\delta)-1|<1$. Indeed, for fixed $a\in \Omega$ and fixed $\delta \in \Gamma$, the factor $c(\delta)$ is an analytic function of $b$. This function has value 1 for $b=a$ and thus the inequality holds for $b$ close to $a$.\\

 For $\delta =s_i$, this implies that $c(s_i)=1$ holds  for $(a,b)$ in a connected component 
 $\Omega^* \times \Omega^*$, where $\Omega^*$ is the complement in $\Omega$ of the orbits of the points $a_0,b_0,\dots ,a_g,b_g, \infty$. Since this rigid space is connected, one has $c(s_i)=1$ for any choice
 $(a,b)\in \Omega^* \times \Omega^* \setminus \{\mbox{diagonal}\}$. This shows the $\Gamma$-invariance of $F(z)$. 
  The orbits $\Gamma a$ and $\Gamma b$ are the simple zero's and the simple poles of the theta function $F(z)$.
  Finally, $F(z)$ is defined over $K$ since $a,b \in K$ and $\Gamma \subset PGL(2,K)$.\\
 
 \noindent  Statements (3) and (4)  follow easily from (2).        \end{proof}

\begin{theorem} \label{1.6} Let $X/K$ be a hyperelliptic curve with set of branch points $Branch_X\subset \mathbb{P}^1(K)$.
For each irreducible component $L$ of the reduction  $\overline{(\mathbb{P}^1_K, Branch_X)}$ one considers the map
\[\phi_L\colon Branch_X\rightarrow  \overline{(\mathbb{P}^1_K, Branch_X)} \overset{projection}{\rightarrow} L\]
with image $\{p_{1},\dots p_{m}\}$ ($m\geq 3$ and the points $\{p_i\}$  are depending on $L$).

Then $X$ is a potential Mumford curve if and only if for every $L$, at most two of the numbers
$m_i:=\# \phi_L^{-1}(p_{i})$ are odd.  \end{theorem}

\begin{proof}  By Theorem \ref{1.4}, the curve $X$ is a potential Mumford curve if and only if the normalisation of every 
irreducible component of a semi-stable reduction has genus zero. Here, we consider the semi-stable reduction of $X$ induced by
the reduction $ \overline{(\mathbb{P}^1_K, Branch_X)}$. The irreducible components $M$ of this reduction are nonsingular curves which
meet at ordinary double points. See Section~\ref{Section3} for a detailed computation. Moreover each $M$ is a degree two covering of  
some $L$ (as above) ramified in $\sum _i (m_i\! \! \mod \! (2))[p_i]$. The genus of $M$ is zero if and only if there are at most two odd numbers $m_i$.\end{proof}
 
The rather complicated proof of the following theorem is a variation on a chapter of G.~Van Steen's thesis \cite{S}, see also \cite{S2}.  

\begin{theorem}\label{1.7}
Let $X/K$ be a hyperelliptic curve and $Ram_X$ its set of Weierstrass points. Suppose that 
$Ram_X\subset X(K)$, that the semi-stable reduction $\overline{(X,Ram_X)}$ exists over $K$ and that 
the normalization of every irreducible component of this reduction is a projective line over $k$. 

Then $X/K$ is a Whittaker curve.\end{theorem}
\begin{proof} $X/K$ is a Mumford curve according to Theorem \ref{1.4}.

 Write $\phi\colon \Omega \rightarrow X=\Omega /W$ for the uniformization, where $W\subset PGL_2(K)$ is a Schottky group of rank $g\geq 2$.
  Let $\sigma$ denote the hyperelliptic involution of $X$. Let 
 $q\in \Omega$ map to $p=\phi (q)$ which is a Weierstrass point (i.e., fixed point of $\sigma$). Then $\sigma$ lifts to an automorphism $s=s_0$ of $\Omega$ with $s(q)=q$. Then $s^2=1$ since $s^2$  is a lifting of $\sigma^2=id$ and $s^2(q)=q$.
The fixed points of $s$ are in $\Omega \cap \mathbb{P}^1(K)$ and $s\in PGL_2(K)$.
Write $\Gamma =<W,s>\subset PGL_2(K)$. 
 
  Now $W\subset \Gamma$ is a subgroup of index 2. We want to produce a free basis
 $\{ w_1,\dots ,w_g\}$ of $W$ such that $sw_is^{-1}=w_i^{-1}$ for all $i$. Indeed, this will imply that
 $\Gamma$ is freely generated by the elements $s_0=s, s_1=sw_1,\dots ,s_g=sw_g$ of order two.  \\
 
 Conjugation by $s$ induces an automorphism $\alpha$ of order 2 of $W$. According to \cite{D-S}, $W$ has a free basis
 $\{ a_i \}_{i\in I} \cup \{b_j, c_j\}_{j\in J} \cup \{x_\lambda, y_j\}_{\lambda \in \Lambda, j\in J_\lambda}$ such that
 \[\alpha (a_i)=a_i,\ \alpha (b_j)=c_j,\ \alpha (c_j)=b_j,\ \alpha(x_\lambda) =x_\lambda^{-1},\ 
 \alpha (y_j )=x_\lambda^{-1}y_{j}x_\lambda .\] 
 We have to show that for $W$ only the elements $x_\lambda$ are present. The rank $r$ of $\Gamma$ is defined as
 the dimension of the vector space $\mathbb{Q}\otimes \Gamma_{ab}$.  A general result for finitely generated discontinuous
 groups $\Delta$ is: the rank of $\Delta$ is the genus of $\Omega/\Delta$. In our case $r=0$. Finally, a computation shows
 that $r=0$ implies that only the elements $x_\lambda$ are present in the above presentation of $W$.\\
 
 Finally, we want to show that each of the $g+1$ generators $s_i, i=0,\dots , g$ (of order two) of $\Gamma$ has its fixed points
 in $\mathbb{P}^1(K)$. For every Weierstrass point $\tilde{q}\in X(K)$, one takes a preimage 
 $\tilde{p}\in \Omega  \cap \mathbb{P}^1(K)$ and considers the lift $\tilde{s}$ of $\sigma$ satisfying 
 $\tilde{s}(\tilde{p})=\tilde{p}$. Then $\tilde{s}$ has order two and its fixed points are in $\Omega \cap \mathbb{P}^1(K)$.
 Further,  $\tilde{s}\in \Gamma$ is conjugated to some $s_i$. Since $\Gamma \subset PGL_2(K)$, the fixed points of
 this $s_i$ are in $\mathbb{P}^1(K)$. By varying the Weierstrass point $\tilde{q}$, one obtains all $s_j$ have this property.
\end{proof}

\begin{theorem}\label{1.8}  With the above notations. \\
 {\rm (1)} The set $\Gamma (S)=\{ \gamma (s)|\  \gamma \in \Gamma,\ s\in S \}\subset \Omega \cap K$ is invariant under $\Gamma$ and its set of limit points is $\mathcal{L}$. The group $\Gamma$ acts on the reduction 
 $\overline{(\Omega, \Gamma (S))}$. \\
 {\rm (2)} The quotient  $\overline{X}:=\overline{(\Omega, \Gamma (S))}/W$ is the minimal 
 semi-stable reduction of $X=\Omega/W$ which separates the ramification points $Ram_X\subset X(K)$.\\
 {\rm (3)}.  The hyperelliptic involution $\sigma$ of $X$ acts on the reduction $\overline{X}$. 
 The quotient $\overline{X}/<\sigma > $ identifies with the reduction  $\overline{(\mathbb{P}^1_K,Branch_X)}$.  
 Therefore, the quotient  $\overline{(\Omega, \Gamma (S))}/\Gamma$ also coincides with
  $\overline{(\mathbb{P}^1_K,Branch_X)}$.\\
 \end{theorem}
 \begin{proof} (1) is obvious. $\overline{X}$ is semi-stable since $\overline{(\Omega,\Gamma(S))}$ and its quotient $\overline{X}$ are locally isomorphic.  The minimality of $\overline{X}$ with respect to separating $Ram_X$ follows from
 the properties:   $\Gamma (S)$ is the preimage of $Ram_X$ and $\overline{(\Omega , \Gamma(S)) }$
 is semi-stable and minimal with respect to separating $\Gamma(S)$.  This proves (2).\\
 
   The quotient  $\overline{X}/ <\sigma >$ is a semi-stable reduction of  $X / <\sigma >=\mathbb{P}^1_K$ and separates $Branch_X$, which is the image of $Ram_X$. Therefore the quotient  $\overline{X}/ <\sigma >$ maps to $\overline{(\mathbb{P}^1_K,Branch_X)}$.
This morphism is an isomorphism or a ``blow up''.  A ``blow up'' would produce an irreducible component $L\cong \mathbb{P}^1_k$
having at most two special points. Here `special point' means a node or the image of a point of $Ram_X$. Such a line $L$ does  
not exist since $\overline{X}$ does not have a similar line.\end{proof}

 \begin{comments} \label{newpairs}
 (1). Suppose that the hyperelliptic curve $X$ is a potential Whittaker curve. Theorems \ref{1.6} and \ref{1.7} imply
  that the set of ramification points has {\it  a natural division into pairs}.  These pairs are the images of the pairs of fixed points.\\ 
  
 \noindent (2).  Let the hyperelliptic $X$ over $K$ with $Branch_X\subset \mathbb{P}^1(K)$ be a potential Whittaker curve.  We note that
 $X$ is, in general,  not yet a Whittaker curve. In particular, the fixed points corresponding to the Whittaker group associated to
 $X$ are, in general, not defined over $K$. The Theorems \ref{1.7} and \ref{1.8} imply the more precise statement:\\
 
{\it  Let $K'\supset K$ be a finite separable extension. Then $X\times_KK'$ is a Whittaker curve if and only if
 the minimal semi-stable reduction for $X\times _KK'$ which separates $Branch_X$,  exists over $K'$.\\ }
\end{comments}

\section{\rm A semi-stable reduction of a hyperelliptic curve}\label{Section3}

\begin{proposition} \label{2.1}
 Let $\phi \colon X\rightarrow \mathbb{P}^1_K$ be a hyperelliptic curve with branch locus 
 $\mathbb{B}\subset \mathbb{P}^1(K)$. After replacing $K$ by a finite extension (if needed), the
 ``pull back''  of the reduction  $\overline{(\mathbb{P}^1_K,\mathbb{B})}$ under the map $\phi$ 
 is the unique minimal semi-stable reduction $\overline{(X,Ram)}$ of $X$ which separates the set of ramification points
 $Ram=\phi^{-1}\mathbb{B}$.
 Moreover, the irreducible components of $\overline{(X,Ram)}$ are nonsingular.
\end{proposition}

 \begin{remarks}\label{Rem3.2} A field extension $\tilde{K}\supset K$ is sufficient for the existence of  $\overline{(X,Ram)}$ if some 
(computable) elements of the value group $|K^*|$ are squares in the value group $|\tilde{K}^*|$ and moreover
the residue field of $\tilde{K}$ contains square roots of certain (computable) elements in $k^*$.
 
More precisely, for normalizing the affine equation of $X$ and for each {\it ramified } double point of $\overline{(X,Ram)}$, one needs, 
in general, a degree two extension of the value group $|K^*|$. 
Each {\it unramified} double point of $\overline{(X,Ram)}$ is in general defined over a
quadratic extension of the residue field $k$.  
Moreover, the residue field of $\tilde{K}$ should contain these quadratic extensions. \hfill $\square$\\
 
 \end{remarks}

\noindent {\bf A proof of \ref{2.1} by a computation refining }\cite[5.5.4]{F-vdP}.\\
Let  $\phi\colon X\rightarrow \mathbb{P}^1_K$ denote a hyperelliptic curve of genus $g\geq 2$ with its canonical morphism $\phi$ to $\mathbb{P}^1_K$. Let $\mathbb{B}=\{e_1,\dots , e_{2g+2}\}\subset K$ be its branch locus  and 
 $y^2=c(x-e_1)\cdots (x-e_{2g+2})$  with $c\in K^*$ its affine equation.

 {\it The dual graph} of the reduction 
 $Red\colon \mathbb{P}^1_K\rightarrow \overline{(\mathbb{P}^1_K,\mathbb{B})}$ consists of a set of vertices $V$ 
 and  a set $E$ of edges. For notational convenience we suppose that the reduction has more than one component.
  Each $v\in V$ denotes a component $L_v\cong \mathbb{P}^1_k$ of the reduction.
 Let $L_v^*$ denote $L_v$ minus its double points and write $U(v)$ for the affinoid space $Red^{-1}(L_v^*)$.
 
  An edge $e\in E$ between $v_1,v_2\in V$ corresponds to a point $\{p\}=L_{v_1}\cap L_{v_2}$ of the reduction.
  Let $p^*$ denote $L_{v_1}\cup L_{v_2}\setminus \{ \mbox{ all nodes and the images of } \mathbb{B}\}$. 
  
  Put $U(e):=Red^{-1}(p^*)$.  Then $\{ U(v) \}_{v\in V} \cup \{U(e)\}_{e\in E}$ is a pure affinoid
covering of $\mathbb{P}^1_K$ and the corresponding reduction is again $Red$.  
     
 Write $X(v):=\phi^{-1}U(v),\ X(e):=\phi^{-1}U(e)$, then  $\{ X(v) \}_{v\in V} \cup \{X(e)\}_{e\in E}$ is an affinoid
 covering of $X$. We will show that, after possibly a finite, separable extension of $K$, this is a pure covering and
 produces the unique minimal semi-stable reduction of $X$ that separates the ramification points $\phi^{-1}\mathbb{B}$.
 For this purpose we have to compute equations for the affinoid spaces $X(v)$ and $X(e)$.
 
  {\it  It suffices to show} that 
 the spectral norms of these affinoid spaces  take their values in $|K|$, that the irreducible components of their reductions
 have no singularities and that they meet in ordinary double points.\\

{\bf Formulas for $X(v)$}. For a suitable choice of the parameter $z$ for $\mathbb{P}^1_K$ the data are as follows.
There are elements $b_1,\dots , b_t, d_1,\dots ,d_u\in K^o$ with distinct images in $k=K^o/K^{oo}$. 
Then $L_v$ with parameter $\bar{z}$ has double points $\bar{b}_1,\dots ,\bar{b}_t,\infty$ and the images of $\mathbb{B}$
(under $Red$ lying on $L_v$) are $\bar{d}_1,\dots ,\bar{d}_u$.  Note that $t$ or $u$ can be ``zero''.
The affinoid space  $U(v)=Red^{-1}L_v^*$ has the equations $ |z|\leq 1, |z-b_1|=\cdots =|z-b_t|=1$. The relation with $\mathbb{B}$ is: $m_i:=\# (\mathbb{B}\cap \{ |z-b_i|<1\})$ and $m_\infty:=\# (\mathbb{B}\cap \{ |z|>1\})$ are $\geq 2$, 
$\# (\mathbb{B}\cap \{ |z-d_i|<1\})=1$ and $u+m_\infty +\sum m_i =2g+2$.\\

The original equation $y^2=c(x-e_1)\cdots (x-e_{2g+2})$ is transformed into one of the form
\[y^2=\tilde{c}(z-b_1)^{m_1}\cdots (z-b_t)^{m_t}(z-d_1)\cdots (z-d_u)(1+\epsilon)\] with
$\epsilon \in O(U(v)),\ \|\epsilon \| <1,\ \tilde{c}\in K^*$.  Note that $(1+\epsilon)$ is a square.
After removing square factors this reduces to 
\[y^2=f:=\tilde{c}(z-b_1)^*\cdots (z-b_t)^*(z-d_1)\cdots (z-d_u)\]
with $*\in \{0,1\}, *\equiv m_i \bmod (2)$. \\

A reduction of the affinoid algebra $O(X(v))$ produces a reduced space only if one uses the 
spectral norm $\| \ \|_{sp}$.The reduction with respect to the spectral norm has good properties 
only if $\|O(X(V))\|_{sp}\subset |K|$. Now $\|y^2\|_{sp}=|\tilde{c}|$. 
{\it In case $|\tilde{c}|\in |K^*|^2$}, we may suppose
that $|\tilde{c}|=1$. One makes the guess that $O(X(v))^o=O(U(v))^o[y]/ (y^2-f)$.
The guess is correct if the residue algebra $\overline{O(X(V))}$ is reduced. 
 The latter  is equal to $k[z,\frac{1}{(z-\bar{b}_1)\cdots (z-\bar{b}_t)}][y]/ (y^2-\bar{f})$ where $\bar{f}$ is the
 separable polynomial $\bar{\tilde{c}}(z-\bar{b}_1)^*\cdots (z-\bar{b}_t)^*(z-\bar{d}_1)\cdots (z-\bar{d}_u)$.
Thus $\overline{O(X(V))}$ is reduced and has no singularities.\\

{\it In the case $|\tilde{c}|\not \in |K^*|^2$}, an extension $K'=K(a)$ with $|a^2|=|\tilde{c}|$ is needed to 
obtain a semi-stable reduction. Then one continues as above.

 One observes that the genus of the component in the reduction of $X$, covering $L_v$, is zero if and only if
$u+\# \{*\mbox{ with }m_*\equiv 1 \mod 2\}$ is $\leq 2$. \\

{\bf Formulas for $X(e)$}. Let $z$ be a parameter for $\mathbb{P}^1_K$. Consider elements
$a, b_1,\dots, b_t, c_1,\dots , c_u \in K^*$ such that $0<|a|<1$, $b_1,\dots ,b_t, c_1,\dots ,c_u\in K^o$ and
the images of $b_1,\dots ,b_t$ in $k$ are distinct and the images of $c_1,\dots ,c_u$ in $k$ are distinct. 
 The affinoid  space $Red^{-1} p^*$ is given by the equations $|a|\leq |z| \leq 1$ and 
$|z-b_i|\geq 1$ for all $i$ and $|\frac{a}{z}-c_j|\geq 1$ for all $j$. The elements of $\mathbb{B}$ are mapped
to the ``holes'' of the affinoid space, i.e., $|z-b_i|<1$ for all $i$, $|z|>1$ and $|\frac{a}{z}-c_j|<1$ for all $j$. \\
 
 The original equation $y^2=c(x-e_1)\cdots (x-e_{2g+2})$ transforms into an equation
 $y^2=f$ with $f=\tilde{c}z^m(z-b_1)^*\dots (z-b_t)^*(\frac{a}{z}-c_1)^*\cdots (\frac{a}{z}-c_u)^*(1+\epsilon)$
 where $\| \epsilon\|<1$.  As before, we may suppose that $m$ and all $*$ belong to $\in \{0,1\}$ and we omit
 the term $(1+\epsilon)$.  One observes that $\|y^2\|_{sp}=|\tilde{c}|$. As in the former case one may need
 an extension of $K$ in order to normalize to $|\tilde{c}|=1$. After this we have to consider two subcases.\\
 
 {\it Subcase $m=0$}.  $O(X(e))^o=O(U(e))^o[y]$ with equation $y^2-f$. The reduction $R$ of $O(U(e))^o$ is
 \[k[X_1,X_2, \frac{1}{(X_1-\bar{b}_1)\cdots (X_1-\bar{b}_t)}, \frac{1}{(X_2-\bar{c}_1)\cdots (X_2-\bar{c}_u)}]/(X_1X_2).\]
The reduction of $O(X(e))^o=O(U(e))^o[y]$ is $R[y]/(y^2-\bar{f})$. This is an \'etale extension since 
$\bar{f}=\bar{\tilde{c}}(X_1-\bar{b}_1)^*\cdots (X_1-\bar{b}_t)^*(X_2-\bar{c}_1)^*\cdots (X_2-\bar{c}_u)^*$ is a unit. \\

{\it Observations}:\\
(1). Omitting a double point $p$ from the reduction $\overline{(\mathbb{P}^1_K,\mathbb{B})}$
produces two connected components and splits $\mathbb{B}$ into two subsets. Now $m\equiv 0 \mod (2)$
occurs if and only if these two subsets have even cardinality.\\
(2).  In the case $m\equiv 0 \mod (2)$, there are above the double point $p$ of $\overline{(\mathbb{P}^1_K,\mathbb{B})}$, 
two points of $\overline{(X, Ram_X)}\times_kk^{alg}$. Here $k^{alg}$ denotes the algebraic closure of the residue field $k$
of $K$. These two points may produce a quadratic extension of $k$.\\

{\it Subcase $m=1$}. We start with the equation $y^2=zw$ and where\\
 $w:=\tilde{c} (z-b_1)^*\dots (z-b_t)^*(\frac{a}{z}-c_1)^*\cdots (\frac{a}{z}-c_u)^*$ with $*\in \{0,1\}$.
 Then $\| y \|^2_{sp}=|\tilde{c}|$ and $\|y^{-1}\|_{sp}^2=|\frac{1}{\tilde{c}\cdot  a}|$. For a good reduction it is needed
 that both $|\tilde{c}|$ and $|a|$ belong to $|K^*|^2$. After enlarging $K$, if needed, 
 we suppose that this is the case. As before, we normalize to $|\tilde{c}|=1$ and choose an element $\alpha \in K^*$ with
 $|\alpha |^2=|a|$.\\
 
 A computation leads to the guess $O(X(e))^o=O(U(e))^o[Y_1,Y_2]/J$ where the ideal $J$ is generated by the elements
 $Y_1^2-zw,\ Y_2^2-\frac{\alpha ^2}{a}\cdot \frac{a}{z}\cdot w^{-1}, Y_1Y_2-\alpha$.  
 
 The reduction  $\overline{O(X(e))}$ is the algebra $k[X_1,X_2,Y_1,Y_2,\frac{1}{F}]/ I $ where
 \[F=(X_1-\bar{b}_1)\cdots (X_1-\bar{b}_t)\cdot (X_2-\bar{c}_1)\cdots (X_2-\bar{c}_u)\] 
 and $I$ is the ideal generated by
 $ X_1X_2,Y_1Y_2, Y_1^2-X_1\overline{w},Y_2^2-X_2 \overline{\frac{\alpha^2}{a}} \overline{w}^{-1}$.
 One verifies that this space is reduced. This implies that the chosen norm is the spectral norm. 
 The reduction consists of two  smooth, irreducible affine curves over $k$, meeting in a single double point $q$
 which lies ramified above the double point point $p$ in $\overline{(\mathbb{P}^1_K,\mathbb{B})}$.  The formal local rings of $p$ and $q$ are $K^o[[X_1,X_2]]/(X_1X_2-a)$ and $K^o[[ Y_1,Y_2]]/(Y_1Y_2-\alpha)$ with
 $|a|=|\alpha |^2$.   This ends the verification of  Proposition \ref{2.1} \hfill    $\square$\\

 \begin{corollary}\label{twist} 
 Let $\phi_i\colon X_i\rightarrow \mathbb{P}^1, i=1,2$ denote Whittaker curves over $K$ with the same branch locus 
 $\mathbb{B}\subset \mathbb{P}^1(K)$.  Then there is an isomorphism $\alpha \colon X_1\rightarrow X_2$ with $\phi_2\circ \alpha =\phi _1$.  
 \end{corollary}  
 \begin{proof}  Let  $y^2=c_i(x-e_1)\cdots (x-e_{2g+2}), i=1,2$ denote affine equations for $X_i$. In the proof of Proposition \ref{2.1}
 we considered the corresponding equations  \[y^2=\tilde{c}_i(z-b_1)^{m_1}\cdots (z-b_t)^{m_t}(z-d_1)\cdots (z-d_u)(1+\epsilon)\]
 for  $X_i(v)$ for some vertex $v$ (or a similar equation for $X_i(e)$ and some edge $e$). By construction 
 $\frac{c_1}{c_2}=\frac{\tilde{c}_1}{\tilde{c}_2}$. Sinds the $X_i$ are Whittaker curves, one has $|\tilde{c_i}|\in |K^*|^2$. 
 Therefore $|\frac{c_1}{c_2}|\in |K^*|^2$.  Thus we may start by supposing $|c_1|=|c_2|$. 
 
    Since the $X_i$ are supposed to be Whittaker curves, the reductions $X_i(v)$ for suitable vertices $v$ are both affine lines over the residue field $k$, i.e., an open subset of $\mathbb{P}^1_k$. In this step the $\tilde{c}_i$ are normed to $|\tilde{c}_i|=1$ and
now the image of $\frac{\tilde{c}_1}{\tilde{c}_2}$ in $k^*$ is a square. Indeed, otherwise one reduction is a nontrivial twist of
the other.  Again  $\frac{c_1}{c_2}=\frac{\tilde{c}_1}{\tilde{c}_2}$ and the residue of the latter ratio in $k^*$ is a square.
Therefore $\frac{c_1}{c_2}$ is a square.   \end{proof}
  
  \noindent {\it Remark}. Corollary \ref{twist} extends a result in Teitelbaum's thesis \cite[Lemma 34]{T} 
  concerning  curves of genus 2, to arbitrary genus. \\

   \noindent
  {\it  Example}. Consider $X$ with equation
   \[y^2=cx(x-\epsilon_0)(x-1)(x-1-\epsilon_1)(x-2)(x-2-\epsilon_2)\] with all $0<|\epsilon_i|<1$ and 
  $|c|=1$. The reduction of $X$ obtained by pull back of $\overline{(\mathbb{P}^1, \mathbb{B})}$ depends on 
  $\overline{c}\in k^*$. In case $\overline{c}$ is a square, we may take $c=1$ and the reduction $\overline{X}$ consists of two
  projective lines over the residue field $k$ intersecting in three ordinary double points defined over $k$ and with tangents defined over $k$.
  Then $X$ is a Whittaker curve over $K$, defined by fixed points in $\mathbb{P}^1(K)$.
  
  In the opposite case, the reduction $\overline{X}$ is irreducible and $X/K$ is not a Whittaker curve. However 
  $X\times_KK(\sqrt{c})$ is a Whittaker curve.\\
  
  \noindent A natural question is: the explicit parametrization of a Whittaker curve (Prop \ref{3.2}) produces an equation 
  $y^2=c(x- \dots)\cdots (x-\dots)$. Is $c$ always a square? In Observation~\ref{g2geval} we indicate why indeed 
this holds for $g=2$ and for various configurations occurring in genus~$3$.

   \color{black}
\section{\rm  A fundamental domain for a Whittaker curve}\label{Section4}

  $\ \ \ $ {\it Notation}. Let $\phi \colon X\rightarrow \mathbb{P}^1_K$ denote a Whittaker curve over $K$ with branch locus $\mathbb{B}\subset \mathbb{P}^1(K)$,
    with ramification locus $Ram \subset X(K)$ and hyperelliptic involution $\sigma$. Then $\phi$ induces  a morphism
    $\bar{\phi}\colon \overline{X}:=\overline{(X,Ram)}\rightarrow \overline{(\mathbb{P}^1, \mathbb{B})}$ between the reductions and a morphism
    $\bar{\phi}^d\colon \overline{X}^d\rightarrow  \overline{(\mathbb{P}^1, \mathbb{B})}^d$ between the dual graphs of these spaces.  We note that $\sigma$ acts on
    $\overline{X}$ and $\overline{X}^d$.\\
    
     Let $un\colon \Omega \rightarrow X$ denote the universal rigid analytic covering with covering group $W$ and Whittaker groups $\Gamma \supset W$. 
     Let $\overline{\Omega}$ denote the reduction with respect to $un^{-1}(Ram)$. We note that the latter equals $\Gamma (S)$ where $S$ the set   
     of fixed  points of  generators of order two  for $\Gamma$. 
     
     Then $un$ induces the universal covering (with respect to the Zariski topology)
     $\bar{un}\colon \overline{\Omega}\rightarrow \overline{X}$ and the universal covering (for graphs with their usual topology)  of the dual graph 
     $\bar{un}^d\colon \overline{\Omega}^d\rightarrow \overline{X}^d$. \\
     
      The aim of this section is to {\it define} a `fundamental domain' $F\subset \Omega$ for the action of a given Whittaker group $\Gamma$ on $\Omega$
      by a construction. {\it The first step is the construction of a fundamental domain $F_0\subset \overline{\Omega}^d$ 
      for the action of $\Gamma$ on $\overline{\Omega}^d$}. 
     
     Recall that $\bar{un}^d\colon \overline{\Omega}^d\rightarrow \overline{X}^d$ has covering group $W$, the fundamental group 
     of $\overline{X}^d$.  Further $\Gamma=<W, s>$ where $s$ is any lifting of $\sigma$ satisfying $s^2={\bf 1}$.
     
     One chooses a connected subgraph $R$ of $\bar{X}^d$, which is mapped isomorphically by $\bar{\phi}^d$ to 
     $\overline{(\mathbb{P}^1,\mathbb{B})}^d$. This is done by   
     deleting from $\bar{X}^d$ for each pair of $\sigma$-conjugated edges one of them and similarly for pairs of $\sigma$-conjugated vertices.  
     Then $\bar{\phi}^d\colon R\rightarrow \overline{(\mathbb{P}^1,\mathbb{B})}^d$ is clearly an isomorphism. Since $R$ is simply connected, 
     the preimage $(\bar{un}^d)^{-1}(R)$ of $R$ is the distinct
     union of connected components isomorphic to $R$. The {\it fundamental domain $F_0\subset \overline{\Omega}^d$} is 
     by definition one of these components and $(\bar{un}^d)^{-1}R$ is the disjoint union $\cup _{w\in W} wF_0$.  
     The construction yields canonical isomorphisms $F_0\rightarrow R\rightarrow \overline{(\mathbb{P}^1,\mathbb{B})}^d$. \\

     The subtree $F_0$ of $\overline{\Omega}^d$ lifts to an algebraic subspace $F_1$ of $\overline{\Omega}$ which is the union
     of a finite number of lines over $k$ (projective or affine). \\
     The {\it fundamental domain} $F\subset \Omega$ is defined as the preimage $Red^{-1}(F_1)$ of $F_1$.
     
     \begin{remarks}\label{new4.1}{\it  Some properties of fundamental domains.}\\ 
      (1).$\ F$ is a connected affinoid space with a reduction $\overline{F}$, induced by the reduction of $\Omega$.
       By construction $\overline{F}=F_1$ and the dual graph $\overline{F}^d$ is $F_0$. The latter identifies with 
       $\overline{(\mathbb{P}^1,\mathbb{B})}^d$. There are finitely many fundamental domains $F$, corresponding to the
       finitely many choices for $R\subset \overline{X}^d$.  The opposite $F^{opp}$ of $F$ is defined by the choice 
       $\sigma R \subset \overline{X}^d$.  From $R\cup \sigma R=\overline{X}^d$ one concludes that 
       $\Omega=\cup _{\gamma \in \Gamma} \gamma (F)$.\\
       \noindent      (2). The ramification points of $X$ have a natural division into pairs (see \ref{newpairs}) and are mapped to 
       vertices of $\overline{X}^d$ which are ramified over $\overline{(\mathbb{P}^1, \mathbb{B})}^d$. This and the choice of 
       $R$ and $F_0$ imply that
      the intersection of $F$ with the set of all fixed points in $\Omega \subset \mathbb{P}^1$ consists of 
       the pairs of points $\{A_i,B_i\}_{i=0}^g$ and are mapped to the pairs of branch points $\mathbb{B}\subset \mathbb{P}^1$.
         Let $s_i$, for $i=0,\dots ,g$,  denote the order two element in $PGL_2(K)$ with fixed points
        $A_i,B_i$. By construction, $s_i\in \Gamma$. We will show that  $\Gamma =<s_0>* \dots * <s_g>$.
        This follows from Theorem \ref{2.2}, where the action of $\Gamma$ on $\Omega$ is replaced by the `equivariant' action on
        $\overline{\Omega}^d$.  \end{remarks} 
       
       \begin{lemma} \label{new4.2} Consider (with the above notation) the action of $\Gamma$ on $\overline{\Omega}^d$.\\
        If $\gamma\in\Gamma$ satisfies $\gamma F_0\cap F_0\neq \emptyset$, then  $\gamma \in \{1, s_0,\dots ,s_g\}$.
        \end{lemma}
     \begin{proof}
     Let $\gamma F_0\cap F_0\neq \emptyset$.  If $\gamma \in W$, then $\gamma =1$ since the sets $\{wF_0\}_{w\in W}$ are disjoint. 
Now consider $\gamma \not \in W$. The set $\gamma F_0\cap F_0$ is connected, since $\overline{\Omega}^d$ is a tree.
The image of $\gamma F_0\cap F_0$ (under $\overline{un}^d\colon \overline{\Omega}^d\rightarrow \overline{X}^d$) lies in a connected component of $R\cap \sigma R$. 

 Let
$\{\bar{a}_i,\bar{b}_i \} \subset \overline{X}^d$ denote the images of the ramification points of $X$. Let $C_i$, for $i=0,\dots ,g$, be the connected component of $R\cap \sigma R$ containing $\bar{a}_i,\bar{b}_i$.  
 One can  verify it follows from the construction of $R$ that $R\cap \sigma R$ is the (disjoint) union of these  
 connected $C_i,\ i=0,\dots ,g$. 
 
\noindent  Indeed, consider for convenience the case that every vertex of $\overline{X}^d$ is odd. Let $e_1,\tilde{e}_1,\dots ,e_g,\tilde{e}_g$ be the
pairs of conjugated edges. Then $R=\overline{X}^d\setminus \{e_1,\dots ,e_g\}$ and
 $\sigma R= \overline{X}^d \setminus \{\tilde{e}_1,\dots ,\tilde{e}_g\}$. Further $R\cap \sigma R$ is  obtained from the tree $R$ by deleting
 $g$ edges. It has $g+1$ connected components and it can be seen  that  each component contains some images of the  branch points.
 
  Let $C_i$ be the image of $\gamma F_0\cap F_0$. Then the images of $A_i,B_i$ are in  $\gamma F_0\cap F_0$. These points also belong to  $s_i\gamma F_0\cap s_iF_0$. Therefore $F_0 \cap s_i\gamma F_0 \neq \emptyset$.  Since $s_i\gamma \in W$ 
 one has  $s_i\gamma =1$ and $\gamma =s_i$.   \end{proof}
     
  \begin{theorem} \label{2.2} A Whittaker curve $X$ over $K$ is given and a fundamental domain $F_0\subset \overline{\Omega}^d$
     is chosen. Let $\Gamma, W, F_0, A_i,B_i$ etc. be defined as above.\\
      Let $s_i$, for $i=0,\dots, g$, be the automorphism of 
     $\overline{\Omega}^d$ which lifts the automorphism $\sigma$ of $\overline{X}^d$ and fixes the vertices (or vertex) of $F_0$
     which contains the images of $A_i,B_i$.
      Then  $s_i^2=1$ for all $i=0,\dots ,g$ and $\Gamma = <s_0>*\dots *<s_g>$.  \end{theorem}
     \begin{proof}  The element $s_i^2$ lifts the identity and has a fixed point. Therefore $s_i^2=1$. 
     Write $G$ for the subgroup of $\Gamma$ generated by $s_0,\dots , s_g$. Consider the subgraph $T=\cup _{g\in G}gF_0$
     of $\overline{\Omega}^d$. From the position of the fixed points of  the $s_i$
      one deduces that $F_0\cup s_0F_0\cup \dots \cup s_gF_0$ is   
     connected. The same argument shows that $T$ is a connected subtree of $\overline{\Omega}^d$. As before,
     $\overline{un}^d\colon \overline{\Omega}^d\rightarrow \overline{X}^d$ denotes this universal covering. We note that 
     $\overline{\Omega}^d/W\rightarrow \overline{X}^d$ is an isomorphism.
     
     The restriction 
     $\tilde{un}\colon T\rightarrow \overline{X}^d$ of $\overline{un}^d$, is surjective since the image contains $R$ and $\sigma R$ and moreover
     $\overline{X}^d=R\cup \sigma R$.\\
       
 Consider the group $G_0:=G\cap W$, which has index 2 in $G$. It is  a free group and acts without fixed points on $T$.

 Lemma \ref{new4.2}  implies that the map $T/G_0\rightarrow \overline{X}^d$ is an isomorphism. Indeed, it suffices to prove that this map is injective. This
is equivalent to $wt_1=t_2$ with $t_1,t_2\in T,\ w\in W$ implies $w\in G_0$. Write $t_i=g_i\tilde{t}_i$ with $g_i\in G$ and $\tilde{t}_i\in F_0$ for $i=1,2$.  Then $g_2^{-1}wg_1\tilde{t}_1=\tilde{t}_2$.  By Lemma \ref{new4.2} one has $w\in G\cap W=G_0$.
 We found that  $\tilde{un}\colon T\rightarrow \overline{X}^d$ is also a universal covering.

The universal property of this  covering implies that $T=\overline{\Omega}^d$ and that $G_0=W$. The subgroup  $W$ of $\Gamma$ is generated by the elements
$\{s_0s_i\}_{i=1}^g$ and is known to be a free group of rank $g$. Thus the elements $\{s_0s_i\}_{i=1}^g$ have no relations and 
$\Gamma=<s_0>* \cdots *<s_g>$.
\end{proof}

\begin{remarks}\label{new4.4} \mbox{ }\\
(1). The proof of  Lemma \ref{new4.2}  is not given in all detail for all cases.  In three examples (\S\ref{2.3}), we make
  the universal coverings and \ref{new4.1}, \ref{new4.2} more explicit. The general case is a combination of these three examples.\\
  
  \noindent (2).    Lemma \ref{new4.2} can be formulated as follows: Translates $\gamma_1F_0\neq  \gamma_2F_0$ of the fundamental domain $F_0$  are   called {\it neighbours} if $\gamma_1F_0\cap \gamma_2F_0\neq \emptyset$. Thus the neighbours of $\gamma F_0$ are
 $\gamma s_iF_0$ for $i\in \{0,\dots , g\}$. One concludes:
   
{\it  For reduced words $t$ in $s_0,\dots ,s_g$, the distance of  the vertices in $tF_0$ and in $F_0$ in the tree  $\overline{\Omega}^d$ is greater than or equal to $-1+ {\rm length}(t)$.} \\

 \noindent (3).  {\it Notation}. Let $m\colon \mathbb{B}\rightarrow V$ denote the map from the branch points to the set of vertices of the tree 
  $\overline{(\mathbb{P}^1,\mathbb{B})}^d$.  An edge $e$ of this tree  is called {\it even} if the two connected components of
 $\overline{(\mathbb{P}^1,\mathbb{B})}^d \setminus \{e\}$ contain an even number of images of $\mathbb{B}$.  
Otherwise, the edge $e$ is called {\it odd}.  {\it The criterion \ref{1.6} for $\overline{(\mathbb{P}^1, \mathbb{B})}$ to produce a 
Whittaker curve and its minimality translate into:     }

        \indent 
         The (only)  possibilities for  the vertices $v\in V$ are \\
       (a).   $\# m^{-1}(\{v\})=0$, $v$ has $\geq 3$ edges and at most two of the edges are odd.\\
       (b).   $\# m^{-1}(\{v\})=1$, $v$ has $\geq  2$ edges and one edge is odd.\\  
       (c).  $\# m^{-1}(\{v\})=2$,  $v$ has  at least one edge and all its edges are even.\\
       
        A vertex $v$ is called {\it even } if $\# m^{-1}(\{v\})=0$ and every edge of $v$ is even. Otherwise $v$ is called odd.
        The degree two covering 
        $\overline{\phi}^d\colon \overline{X}^d\rightarrow \overline{(\mathbb{P}^1,\mathbb{B})}^d$ can be described by:   
        {\it    The even edges and the even vertices are doubled}. \\
       
   \noindent (4).    The fixed points of the Whittaker group $\Gamma=<S_0,\dots ,S_g>$ for its action on $\Omega$ map to the set of branch points
       $\mathbb{B}$. This produces a division of $\mathbb{B}$ into $g+1$ pairs  denoted by  $\{A_i,B_i\},\ i=0,\dots g$ (i.e., the images of the 
       fixed points of $S_i$). This does not depend on the choice of the 
       elements $S_0,\dots , S_g$, since any element $s\in \Gamma$ of order two is $\Gamma$-conjugated to a unique element $S_i$.

       These pairs can also be directly deduced from (a), (b) and (c) as follows:\\
       $b_1,b_2\in \mathbb{B}, b_1\neq b_2$ is a pair if (i) $m(b_1)=m(b_2)$ is a vertex of type (c), or \\
       (ii) $m(b_1), m(b_2)$ are vertices of type (b) and on the shortest path from $m(b_1)$ to $m(b_2)$ there are only vertices of type (a). \\
       
    \noindent(5). A slightly different approach to the relation between fixed points and branch points is given in the
manuscripts \cite{Yel1, Yel2}. In fact here the more general case of superelliptic curves is considered. The results are
formulated in terms of `clusters' and Berkovich spaces.   
       
       \end{remarks}
       
\section{ \rm Configurations, branch points and fixed points}\label{Section5}

 \subsection{\rm Configurations and standard examples}
 
 \begin{definitions} \label{2.4} $\ $ {\it  Configurations, $Branch_{P,m}$ and $Fix_{P,m}$} \\
  \noindent  (1).  A {\it configuration} is  a pair $(P,m)$ where $P$ is a finite tree and $m$ is a map from
     the set of pairs $\{a_i,b_i\}_{i=0}^g$ to the vertices of $P$ such that for each vertex $v$, one of the properties (a), (b) and (c) 
     holds, where \\ 
      (a).   $\# m^{-1}(\{v\})=0$, $v$ has $\geq 3$ edges and at most two of the edges are odd.\\
       (b).   $\# m^{-1}(\{v\})=1$, $v$ has $\geq  2$ edges and one edge is odd.\\  
       (c).  $\# m^{-1}(\{v\})=2$,  $v$ has  at least one edge and all its edges are even.\\

\noindent (2).  An injective map $M\colon \mathbb{B}\rightarrow \mathbb{P}^1$ is called of type $(P,m)$ if the induced
map $\mathbb{B}\rightarrow \overline{(\mathbb{P}^1,M(\mathbb{B}))}^d$ coincides with $(P,m)$. We recall that such $M$, seen as 
branch locus, produces a Whittaker curve.  Let $Branch_{P,m}$ denote the set
of all maps of type $(P,m)$. 
Two elements $M_i\colon \mathbb{B}\rightarrow \mathbb{P}^1, \ i=1,2$ are identified if there is an element
$g\in {\rm PGL}_2$ with $M_2=g\circ M_1$.  We note that $Branch_{P,m}$ has a natural structure as rigid analytic space over $K$.
This will be made explicit in Examples \ref{2.3}.\\

\noindent (3). A injective map $M\colon \mathbb{B}\rightarrow \mathbb{P}^1$ of type $(P,m)$ induces a map, also called $M$, from 
$\mathbb{B}$ to the tree $\overline{(\mathbb{P}^1,M(\mathbb{B})}$ of projective lines over the residue field $k$. An irreducible 
component $L$ of this tree is called even, resp. odd, if the corresponding vertex of $\overline{(\mathbb{P}^1,M(\mathbb{B})}^d$
is even resp. odd.  

 Suppose that the component $L$ is odd. One considers the projection $pr_L\colon \overline{(\mathbb{P}^1,M(\mathbb{B})}  \rightarrow L$.
There are precisely two points $p$ on $L$ where the number of points  $pr_L^{-1}(\{p\})\cap M(\mathbb{B})$ is odd. Let   $p_1,p_2$
denote these points and define $s_L$ to be automorphism of order two of $L$ with fixed points $p_1.p_2$. Let 
$nodes(L)$ denote the nodes on $L$ different from $p_1,p_2$.  Now we can define:\\

\noindent
an (injective) map $M\colon \mathbb{B}\rightarrow \mathbb{P}^1$ of type $(P,m)$ is {\it restricted} if for very odd line $L$ the intersection
$nodes(L)\cap s_L(nodes(L))$ is empty.\\

Furthermore $Fix_{P,m}$ denotes the space of all maps $M\colon \mathbb{B}\rightarrow \mathbb{P}^1$ of type $(P,m)$ which are restricted.
Two elements $M_1,M_2$ are identified if there exists $g\in {\rm PGL}_2$ with $M_2=g\circ M_1$.  
It can be seen that $Fix_{P,m}$ has a natural structure as rigid analytic space.\\

  We note that the elements $M\in Fix_{P,m}$ are used as positions for the fixed points of Whittaker groups.
 \end{definitions}

  \begin{examples}  {\it Three  standard configurations for Theorem \ref{2.2}.}  \label{2.3}   
  
   A configuration can be seen as a combination of the following standard examples. In verifying (combinatorial) statements on
   configurations one may restrict to these examples.
  
  \label{(2.3.1)}\noindent  {\bf (5.2.1) First example}. \label{example5.2.1} {\it The tree $\overline{(\mathbb{P}^1,\mathbb{B})}^d$ is defined by}:\\
  The vertices are $p,v_0,v_1,v_2,\dots ,v_g$;  the edges  are $e_i=(p,v_i)$ for $i=0,\dots ,g$. 
  The branch points $a_i,b_i$ are mapped to $v_i$ for $i=0,\dots ,g$. 
  The graph $\overline{X}^d$ is obtained by doubling the vertex $p$, this produces vertices $p, \tilde{p}$. The edges  are doubled and this produces edges $e_i, \tilde{e}_i$.  A fundamental domain $F_0$, lying in the universal covering $UV$ of $\overline{X}^d$,  is associated to the tree $T\subset \overline{X}^d$ which is obtained by deleting $\tilde{p}$ and $\tilde{e}_{0}, \tilde{e}_{1},\dots ,\tilde{e}_{g}$. 
  
  We choose $v_0$ as base point for the homotopy. The points of the universal covering 
  \[pr\colon UV=\overline{\Omega}^d\rightarrow \overline{X}^d\] are
  the non-backtracking paths in $\overline{X}^d$, starting with $v_0$. The fibre $pr^{-1}(v_0)$ is the fundamental group $W$ which acts
  as automorphism group on $UV$. The action of the hyperelliptic involution $\sigma$ on $\overline{X}^d$ is denoted by
  $\ \tilde{}\ $ and on $UV$ is denoted by $s$.
  
 Let $a,b\in T$. Then $[a,b]$ denotes the path in $T$ from $a$ to $b$. Further $\tilde{[a,b]}$ denotes the image of this
  path under $s$. Put $\gamma_i=[v_0,v_i]\tilde{[v_i,v_0]}$ and  $s_i=s\gamma_i$ for $i=1,\dots , g$. Write $s_0=s$. 
  Define $\Gamma :=<W,s>$.  One can verify:\\
  (a).   $\{\gamma_i\}_{i=1}^g$ are free generators for $W$.\\
  (b).   $s\gamma_i=\gamma_i^{-1}s$ and $s_i^2=1$.\\
  (c).   $[\Gamma :W]=2$ and $\Gamma=<s_0>*\cdots *<s_g>$.\\
  (d).   Let $t\in <s_0>*\cdots *<s_g>$ be a reduced word. The distance between the vertices of 
  $F_0$ and $tF_0$ is $\geq -1+{\rm length} (t)$.\\
   There is only one other choice of fundamental domain, namely the one associated to the maximal subtree 
  $\sigma(T)\subset \overline{X}^d$. This corresponds to simultaneous conjugation by $s$ of the above formulas.\\
  
    Let $M\in Branch_{P,m}$, then $\overline{(\mathbb{P}^1,M(\mathbb{B}))}$ is a tree of lines $L_p,L_0,\dots ,L_g$ corresponding
    to the vertices $p,v_0,\dots ,v_g$ and with double points $L_p\cap L_0,\dots ,L_p\cap L_g$. The line $L_p$ is even.
    The lines $L_i, i=0,\dots , g$ are odd and $node(L_i)$ consists of one point $L_p\cap L_i$. We conclude that $M$ is
    restricted and that $Fix_{P,m}$ coincides with $Branch_{P,m}$. \\

      {\it The description of $Branch_{P,m}$  as rigid analytic space} is as follows. It is the subspace of $K^{2g-1}$ consisting of the tuples
      $(a_0,b_0,a_1,\dots , a_g,b_g)$ such that $a_0=0, a_1=1, a_g=\infty$ and 
      $(b_0,b_1,a_2,\dots , b_{g-1},b_g)\in K^{2g-1}$ satisfies the equalities and inequalities
      $ |a_i|\leq 1$ for $0\leq i\leq g-1$; $|a_i-a_j|=1$ for $i\neq j$ and $0\leq i,j \leq g-1$; $ 0< |b_i-a_i|<1$ for $0\leq i \leq g-1$ and
      $1<|b_g|<\infty$.      \hfill   $\square$\\

 \label{(2.3.2)}\noindent  {\bf (5.2.2) Second example}. \label{example5.2.2}{\it The tree $T:=\overline{(\mathbb{P}^1, \mathbb{B})}^d$ has only odd  vertices}. \\ 
 Let $e_1,\dots ,e_g$ denote the even edges of $T$, then $\overline{X}^d$ is obtained by adding for each $i$
  an edge $\tilde{e}_i$ with the same end points as $e_i$. The involution $\sigma$ on $\overline{X}^d$ is the
  identity on the vertices and permutes $e_i, \tilde{e}_i$ for every $i$.

  For notational convenience we subdivide each $e_i$ and $\tilde{e}_i$ by adding vertices $z_i$ and $\tilde{z}_i$.
 Moreover we choose a base point $p_0\in \overline{X}^d$ for the monodromy with $p_0\neq z_i,\tilde{z}_i$ for all $i$.\\
\indent  {\it The universal covering} $pr\colon UV\rightarrow \overline{X}^d$ can now be described as follows:\\
 (a). the vertices of the tree $UV$ are the paths $p:=p_0p_1\dots p_m$ in $\overline{X}^d$ starting with the vertex
 $p_0$ and non-backtracking (i.e., $p_j\neq p_{j+2}$).\\
 (We note that any path $q_1q_2\dots q_n$ can be changed into a non-backtracking path by omitting suitable $q_i$).\\
 (b). The neighbours of the vertex $p=p_0p_1\dots p_m$ are obtained by deleting $p_m$ or by adding a new vertex $p_{m+1}$.\\ 
 (c). $pr\colon UV\rightarrow \overline{X}^d$ is given by  $p=p_0p_1\dots p_m \mapsto p_m$.\\
 
{\it  Further notations and properties:\\}
 \noindent (d). $W=pr^{-1}(p_0)$ is a group under composition of closed paths. It is the fundamental
  group of $\overline{X}^d$ with base point  $p_0$. \\
 (e). $W$ acts on $UV$ by: $w\in W$ applied to $p_0p_1\dots p_m$ is  $wp_0p_1\dots p_m$.\\ 
 (f). The natural action of the involution $\sigma$ on $UV$ is called $s$. It acts on any path by interchanging all 
 $z_i, \tilde{z}_i$ occurring in that path. One has $s^2=1$.
  
   Further, $s$ is the lift of the automorphism $\sigma$ that fixes the path $(p_0)\in UV$. \\
 (g).   Any maximal tree in $\overline{X}^d$ is obtained by deleting, for each $i$, one of the vertices $z_i,\tilde{z}_i$.
 The tree obtained by deleting $\tilde{z_i}$ for every $i$ is also called $T\subset \overline{X}^d$.
 For two vertices $a,b$ we denote by $[a,b]$ the path in $T$ from $a$ to $b$. Furthermore $\tilde{[a,b]}$ will denote
 the result of the map $s$ applied to $[a,b]$.\\
 (h). For $i=1,\dots ,g$, the neighbours of both vertices $z_i,\tilde{z}_i$ are called $v_i,w_i$. The notation is choosen
 such that the path $[p_0,w_i]$ in $T$ contains $z_i$.

 For each $i$ we consider the closed path $\delta _i$, defined by $[p_0,w_i]\tilde{[w_i,v_i]}[v_i,p_0]$. The elements
 $\delta_1,\dots ,\delta_g$ are the standard free generators of the fundamental group $W$.
 Define $\gamma _i=[p_0,w_i]\tilde{[w_i,p_0]}$ for $i=1,\dots ,g$. One easily verifies that 
$\gamma_1,\dots ,\gamma_g$ are also free generators for $W$.\\
(i). We regard $W$ as a group of operators on $UV$. For this action one computes that for every $i$ one has
 $s\gamma_i=\gamma _i^{-1}s$. Put $s_i:=s\gamma_i$. Then $\Gamma:=<W,s>$ is an extension of degree 2 of $W$ and is freely generated by the elements $s_0=s,s_1,\dots ,s_g$ of order two.\\
 (j). From the explicit formulas of $s_0,\gamma_1,\dots ,\gamma_g$ one can deduce:
the distance of $tF_0$ and $F_0$ is $\geq -1+ \mbox{ length}(t)$  where $F_0\subset \overline{\Omega}^d=UV$ is the above 
fundamental domain $T$ and $t$ is a reduced word in the generators $s_0,\dots , s_g$.\\   
 (k). Let $T^*$ denote the maximal tree in $\overline{X}^d$ defined by deleting $a_i\in \{z_i,\tilde{z}_i\}$ for each $i$.
 This produces new free generators $\gamma_1^*,\dots ,\gamma_g^*$ for $W$ and new free generators $s_0^*,\dots , s_g^*$
 of order two for $\Gamma$. We study the relations between the new generators and the old ones.\\
 
  {\it Consider the special case} where $\overline{(\mathbb{P}^1,\mathbb{B})}^d$ has vertices $v_0,v_1,\dots, v_g$ 
and edges $e_i:=(v_0,v_i)$ for $i=1,\dots , g$ and the images of the branch points $a_i,b_i$ are in $v_i$ for 
$i=0,\dots ,g$. Take base point $p_0=v_0$ for $\overline{X}^d$. \\
Then: $\gamma_i^*=\gamma _i$ if $a_i=\tilde{z}_i$ and $\gamma_i^*=s_0\gamma _is_0=\gamma_i^{-1}$ if $a_i=z_i$. Further \\
$s^*_0=s_0$ and for $i\neq 0$ one has $s_i^*=s_i$ if $a_i=\tilde{z}_i$ and $s_i^*=s_0s_is_0$ if $a_i=z_i$.  \\

 {\it In the general case}, the formula for $\gamma_i^*$ is obtained from the above formula  by conjugating with terms  coming from
  the possibility that the path $[p_0,w_i]$ in $T$ contains some vertices $z_j$ with $j\neq i$.\\
 ($\ell$ ). Each choice of a maximal tree provides free order two generators for the group $\Gamma$. There are
$2^g$ possibilities for maximal trees. If one considers free order two generators modulo simultaneous conjugation
then there are $2^{g-1}$ possibilities. Indeed, the order two generators defined by any $T^*$ and those of  its ``opposite'' $\sigma T^*$ 
 differ by simultaneous conjugation. \\
 
 Let $(P,m)$ be the above special case and consider $M\in Branch_{P,m}$. Then $\overline{(\mathbb{P}^1,M(\mathbb{B}))}$ has 
 components $L_0,\dots ,L_g$ corresponding to the vertices $v_0,\dots ,v_g$ and nodes $p_i:=L_0\cap L_i, i=1,\dots ,g$. 
 All the $L_i$ are odd. For $i\neq 0$, the condition ``restricted'' is satisfied for $L_i$. For $L_0$, the condition ''restricted'' reads
 $\{p_1,\dots ,p_g\}\cap \{s_{L_0}(p_1),\dots ,s_{L_0}(p_g)\}=\emptyset$.  
 
   Thus $Fix_{P,m}$ is a proper subspace of $Branch_{P,m}$. \\

    {\it The analytic structure of $Branch_{P,m}$} for the special case can be given as follows. It consists of all tuples
   $(a_0,b_0,\dots ,a_g,b_g)$ with $a_0=0, b_0=\infty, a_1=1$ and 
   $(b_1,a_2,b_2,\dots ,a_g,b_g)\in K^{2g-1}$ satisfying the equalities and inequalities
   $|a_i|=1$ for $1\leq i \leq g$; $|a_i-a_j|=1$ for $i\neq j$ and $0\leq i,j\leq g$; $0<|a_i-b_i|<1$ for $1\leq i \leq g$.
   
  {\it The subspace $Fix_{P,m}$} of $Branch_{P,m}$ is given as rigid space by the additional equalities $|a_i+a_j|=1$
   for $i\neq j$ and $0\leq i,j\leq g$.          \hfill $\square$\\

\label{(2.3.3)} \noindent  {\bf (5.2.3) Third example}.\label{example5.2.3} {\it The   configuration $(P,\tilde{m})$  is defined as follows:}
  The set of vertices is $\{v_0,v_1,\dots ,v_n, w_0,w_1,\dots ,w_m\}$ and the set of edges is
 \[\{(v_0,v_i), i=1,\dots, n\}\cup \{(w_0,w_j), j=1,\dots ,m\} \cup \{(v_0,w_0)\}.\] 
The points $a_0,b_0$ are mapped by $m$ to $v_0, w_0$.
  The  points $a_j,b_j, 1\leq j \leq n+m$ are mapped pairwise to the vertices $v_1,\dots ,v_nw_1,\dots ,w_m$. 
 {\it The new feature} is that the pair $a_0,b_0$ is not mapped to the same vertex.
 
We will not study this case in detail because it resembles (5.2.2). In fact, if one identifies the vertices $v_0$ and $w_0$ and contracts the edge $(v_0,w_0)$, then we recover case (5.2.2).

Consider $M\in Branch_{P,\tilde{m}}$. Then $\overline{(\mathbb{P}^1,M(\mathbb{B}))}$ has irreducible  components \\$L_0,\dots ,L_n,M_0,\dots ,M_m$
corresponding to   $\{v_0,v_1,\dots ,v_n, w_0,w_1,\dots ,w_m\}$. All the components are odd. Restricted does not
impose restrictions for the components $L_i, i\neq 0$ and $M_i, i\neq 0$. Write $\{p_i\}=L_0\cap L_i$ for $i=1,\dots ,n$ and write
$s$ for the involution on $L_0$. Then the restriction for $L_0$ reads $\{p_1,\dots ,p_n\}\cap \{s(p_1),\dots, s(p_g)\}=\emptyset$.
The condition for $M_0$ is similar.
  
  One observes that $Fix_{P,\tilde{m}}$ is a proper subspace of $Branch_{P,\tilde{m}}$ except for the case $n=m=1$.

  The description of $Branch_{P,\tilde{m}}$ and $Fix_{P,\tilde{m}}$ as rigid space is similar
   to the two cases above.   \hfill $\square$ \\

 \end{examples}

 \subsection{\rm The Galois covering \texorpdfstring{$FB\colon Fix_{P,m}\rightarrow Branch_{P,m}$}{FB: Fix(P,m) -> Branch(P,m)}}
  
  \begin{proposition}\label{2.6} Let the Whittaker curve  $X/K$ have branch locus
  $M\colon \mathbb{B}\rightarrow \mathbb{P}^1$ with configuration $(P,m)$ and Whittaker group $\Gamma$.
   Let $F\subset \Omega$ be any fundamental domain.  The intersection of $F$ with set of all fixed points of $\Gamma$ in $\Omega$ 
   is a set of pairs  $\{a_i,b_i\}_{i=0}^g$. Let $s_i$ be the involution with fixed points $\{a_i,b_i\}$.
   
  Then $\Gamma = <s_0>*\cdots *<s_g>$. Moreover, the map  $\tilde{M}\colon \mathbb{B}\rightarrow \mathbb{P}^1$ defined by these fixed points
    is restricted of type $(P,m)$.  
   \end{proposition}

 \begin{proof} All statements are already shown in \ref{new4.1} and \ref{2.2}, except for the item concerning  ``restriction''. 
 The degree two map $\overline{\phi}\colon \overline{X}\rightarrow \overline{(\mathbb{P}^1,\mathbb{B})}$ induces a degree two covering
 $\psi\colon \tilde{L}\rightarrow L$  for any odd irreducible component $L$ of  $\overline{(\mathbb{P}^1,\mathbb{B})}$. 
 This $\psi$ is branched at two points which are distinct from the nodes on $L$. Hence above each node of $L$,
  there are two nodes of $\tilde{L}$. 
  
   Using the notation of Section~\ref{Section4}, one has that $F$ is defined by the choice of a special maximal tree $R\subset \overline{X}^d$. 
 Then $\tilde{L}\subset F_1\subset \overline{X}$ and $\tilde{L}$, as subspace of $F_1$, contains for every pair of conjugated
 nodes one element.  Hence, the nodes of $\tilde{L}$ which
 belong to $F_1$ have an empty intersection with their conjugates under the involution $s_{\tilde{L}}$.
 The same property holds for the preimage $\overline{F}\subset \overline{\Omega}$ of $F_1$. 
 Thus $\tilde{M}$ is restricted. \end{proof}

  \begin{theorem}\label{new5.4}  The fixed points, defined by any restricted map $M\colon \mathbb{B}\rightarrow \mathbb{P}^1$ of type $(P,m)$,
   are in good position. The branch locus of the associated Whittaker curve $X\rightarrow \mathbb{P}^1$ is also of type $(P,m)$. \end{theorem}
  {\noindent  The proof is postponed until Proposition~\ref{5.2}}.

  \begin{remarks}\label{new4.9} {\it The map $FB\colon Fix_{P,m}\rightarrow Branch_{P,m}$.}\\
 (1).  $FB$ denotes the map that associates to any set of restricted fixed points of type $(P,m)$,  the branch points of the
  corresponding Whittaker curve.
  
   That {\it $FB$ is a morphism of rigid space} will be apparent from the explicit form of $FB$
  in terms of theta functions, see Section~\ref{Section6}.

   By Proposition \ref{2.6}, the choice of $M\in Branch_{P,m}$ and fundamental domain $F$, determines   
  an element, say, $(M,F)\in Fix_{P,m}$ with $FB(M,F)=M$.  In particular, $FB$ is surjective.  We want to show that for any $N\in Fix_{P,m}$
  with image $M$, there is a fundamental domain $F$ with $N=(M,F)$. \\
  (2). For notational convenience we introduce the following notion.\\
  A {\it (pointed) $\mathbb{B}$-tree over $k$} is a tree $T$ of projective lines over the residue field $k$, intersecting in ordinary nodes.
   In addition, a map is given from $\mathbb{B}:=\{a_0,b_0,\dots ,a_g,b_g\}$ to $T(k)\setminus \{nodes \}$ such that the dual graph
    satisfies the properties of Definitions~\ref{2.4}.
   Every element $M\in Branch_{P,m}$ or $M\in Fix_{P,m}$ induces a pointed $\mathbb{B}$-tree, namely 
   $\overline{(\mathbb{P}^1,M(\mathbb{B}))}$.\\
   (3). Let $X$ be a Whittaker curve, having $(P,m)$ as configuration of the  branch locus. Let $\overline{X}$ denote, as before, the minimal 
   semi-stable reduction that separates the ramification points of $X\rightarrow \mathbb{P}^1$. A fundamental domain $F$ corresponds to a choice of a special
    maximal subtree $R$ of $\overline{X}^d$. By deleting the nodes and the components of $\overline{X}$ prescribed by $R$, one obtains
    a restricted pointed $\mathbb{B}$-tree of a type $(P,m)$.  One can verify: \\   
   \indent   {\it Every restricted (pointed) $\mathbb{B}$-trees of type $(P,m)$ is obtained in this way from a fundamental domain which
      is unique up to conjugation.       }\\  
    (4). The elements of  $Branch_{P,m}$ and $Fix_{M,p}$ can be represented by tuples $\{a_0,b_0,a_1,b_1\dots ,a_g,b_g\}$
  normalized, for instance, by $a_0=0, b_0=\infty, a_1=1$.
    \end{remarks}

 \begin{proposition} 
 {\rm  (1)}. Let $M_1,M_2\in Fix_{P,m}$ define the same Whittaker group $\Gamma$ and the same pointed $\mathbb{B}$-tree. Then 
$M_1=M_2$.\\
 {\rm (2)}.   Consider $M\in Fix_{P,m}$ with image $N\in Branch_{P,m}$.
 There exists a fundamental domain $F$, unique up to conjugation by $\sigma$, such that the fixed points  $(N,F)$ 
 coincides with $M$.   
  \end{proposition}
 \begin{proof} (1). Let $\{a_0,b_0,\dots, a_g,b_g\}$ be the fixed points for $M_1$, and  $\{A_0,B_0,\dots ,A_g,B_g\}$ those for $M_2$.
 By normalization $a_0=A_0=0$, $b_0=B_0=\infty$, $a_1=A_1=1$. Consider, for $i=0,\dots ,g$, the involutions $s_i$ and $S_i$ with fixed points 
 $a_i,b_i$ and $A_i,B_i$. We may suppose that $s_i$ and $S_i$ are conjugated. Assume that $\#\{a_i,b_i,A_i,B_i\}=4$. Then 
 $\overline{(\mathbb{P}^1, M_1(\mathbb{B}))}=\overline{(\mathbb{P}^1, M_2(\mathbb{B}))}$ implies that the reduction
 $\overline{(\mathbb{P}^1, \{a_i,b_i,A_i,B_i\})}$ has two components $L_1,L_2$ and the images of $a_i,A_i$ are on $L_1$ and 
 those of $b_i,B_i$ are on $L_2$.  This implies that the group $<s_i,S_i>$, generated by $s_i$ and $S_i$, is not discontinuous.
 This contradicts that $<s_i,S_i>$ is a subgroup of $\Gamma$.
 
 Also $A_i=a_i$ and $B_i\neq b_i$ is not possible since $s_iS_i$ has $a_i=A_i\in \Omega$ as fixed point and $s_iS_i$ lies
 in the subgroup $W\subset \Gamma$  of index two. Indeed, the  elements of $W\setminus \{1\}$ have their fixed points outside $\Omega$.
The same reasoning holds in the case $A_i\neq a_i$, $B_i=b_i$.
 We conclude that  $s_i=S_i$ for all $i$ and $M_1=M_2$.\\
 
 \noindent   (2). By Theorem~\ref{new5.4}, $M$ has good position and produces a branch locus $N$ of the same configuration $(P,m)$. 
 By Remarks~\ref{new4.9} part (3), there is a fundamental domain $F$, unique up to conjugation, such that the combination with $N$ produces 
 a $M_1\in Fix_{P,m}$ with the same pointed $\mathbb{B}$-tree as $M$. This includes the maps from 
 $\mathbb{B}$ to the trees. Indeed, the images of the points of $\mathbb{B}$ on the trees of $M$ and $M_1$ are prescribed by the images
 of the ramification points on $\overline{X}$. 
 
 Since $M$ and $M_1$ define the same Whittaker group $\Gamma$, we have $M=M_1$ according to (1).  \end{proof}
  
\begin{corollary} \label{5.7} For a given configuration $(P,m)$, the number of fundamental domains is $2^d$ for some $d\geq 1$.
Every fibre of $FB\colon Fix_{M,p}\rightarrow Branch _{P,m}$ has $2^{d-1}$ elements.
\end{corollary}
\begin{proof} The fundamental domains correspond to suitable subtrees $R\subset \overline{X}^d$. This $R$ is obtained
by deleting for each pair of conjugated edges and each pair of conjugated vertices, one of the two elements. Thus there are
$2^d$ possibilities for fundamental domains.  Moreover, fundamental domains $F_1, F_2$ are conjugated w.r.t. 
${\rm PGL}_2$ if and only if  the corresponding 
$R_1,R_2\subset \overline{X}^d$ satisfy $R_1=R_2$ or $R_1=\sigma R_2$.  Now Remarks~\ref{new4.9} part (3)
finishes the proof.   \end{proof}

\begin{theorem}\label{new5.8} The map $FB\colon Fix_{P,m}\rightarrow Branch_{P,m}$ is a rigid analytic \'etale, Galois covering with
Galois group $\{\pm 1\}^{d-1}$ with $d\geq 1$. Here $d$ is defined by the number of fundamental domains for $(P,m)$ is
$2^d$.
\end{theorem} 

\begin{proof} That $Branch_{P,m}$ and $Fix_{P,m}$ have a natural structure as connected rigid space is made explicit in
Examples \ref{2.3}.  The analyticity of $FB$ follows from the explicit formulas in terms of theta function, see Sections~\ref{Section6}, \ref{explicit1}, \ref{explicit2}.
According to Corollary~\ref{5.7}, any fibre of $FB$ consists of the possibilities for the special maximal trees $R\subset \overline{X}^d$, 
obtained by deleting for each pair of conjugated edges or conjugated vertices, one of the two elements. Two special maximal trees
$R_1,R_2$ are identified if $R_1=R_2$ or $R_1=\sigma R_2$. There is clearly a Galois group of the form $\{\pm 1\}^{d-1}$ acting 
on the set of special maximal trees. In fact, $d$ is the number of the conjugation pairs of edges and vertices.  \end{proof}

\subsection{\rm  Examples of configurations $(P,m)$, the moduli spaces 
 $Fix_{P,m}$,    $Branch_{P,m}$\\ and the morphism $FB$.}

\begin{figure}[ht]\label{picture2}
\centering

\tikzset{
    pt/.style={circle,fill,inner sep=1.5pt}
}

\begin{subfigure}{0.32\textwidth}
\centering
\begin{tikzpicture}

\draw (-2,2) -- (-2,-2);
\draw (0,2) -- (0,-2);
\draw (2,2) -- (2,-2);

\draw (-3,-1) -- (3,-1);

\node[pt] (a0) at (-2,1) {};
\node[pt] (b0) at (-2,0.2) {};

\node[pt] (a1) at (0,1) {};
\node[pt] (b1) at (0,0.2) {};

\node[pt] (a2) at (2,1) {};
\node[pt] (b2) at (2,0.2) {};

\node[left=2mm of a0] {$a_0$};
\node[left=2mm of b0] {$b_0$};

\node[left=2mm of a1] {$a_1$};
\node[left=2mm of b1] {$b_1$};

\node[left=2mm of a2] {$a_2$};
\node[left=2mm of b2] {$b_2$};

\node at (-2.5,-1.6) {(a)};

\end{tikzpicture}
\end{subfigure}
\hfill

\begin{subfigure}{0.32\textwidth}
\centering
\begin{tikzpicture}

\draw (-2,2) -- (-2,-2);
\draw (2,2) -- (2,-2);

\draw (-3,-1) -- (3,-1);

\node[pt] (a0) at (-2,1) {};
\node[pt] (b0) at (-2,0.2) {};

\node[pt] (a2) at (2,1) {};
\node[pt] (b2) at (2,0.2) {};

\node[pt] (a1) at (-0.5,-1) {};
\node[pt] (b1) at (0.5,-1) {};

\node[left=2mm of a0] {$a_0$};
\node[left=2mm of b0] {$b_0$};

\node[right=2mm of a2] {$a_2$};
\node[right=2mm of b2] {$b_2$};

\node[below=2mm of a1] {$a_1$};
\node[below=2mm of b1] {$b_1$};

\node at (-2.5,-1.6) {(b)};

\end{tikzpicture}
\end{subfigure}
\hfill
\begin{subfigure}{0.32\textwidth}
\centering
\begin{tikzpicture}

\draw (-2,2) -- (-2,-2);
\draw (2,2) -- (2,-2);

\node[pt] (a0) at (-2,1) {};
\node[pt] (b0) at (-2,0.2) {};

\node[pt] (a2) at (2,1) {};
\node[pt] (b2) at (2,0.3) {};

\node[pt] (b1) at (-1.15,-1.1) {};
\node[pt] (a1) at (1.4,-1.1) {};

\draw (-2.6,-0.2) -- (0.6,-2.2);
\draw (-0.3,-2.4) -- (2.6,-0.2);

\node at (-2.35,1.0) {$a_0$};
\node at (-2.35,0.2) {$b_0$};

\node at (2.35,1.0) {$a_2$};
\node at (2.35,0.3) {$b_2$};

\node at (-1,-0.75) {$b_1$};
\node at (1.3,-0.75) {$a_1$};

\node at (-2.5,-1.6) {(c)};

\end{tikzpicture}
\end{subfigure}

\caption{The three configurations for fixed points for Whittaker curves of genus~2, 
here represented as pointed $\mathbb{B}$-trees, and their generalizations of higher genus
 are studied in detail in the Subsections~\ref{5.3.1}--\ref{5.3.3}.}
\end{figure}
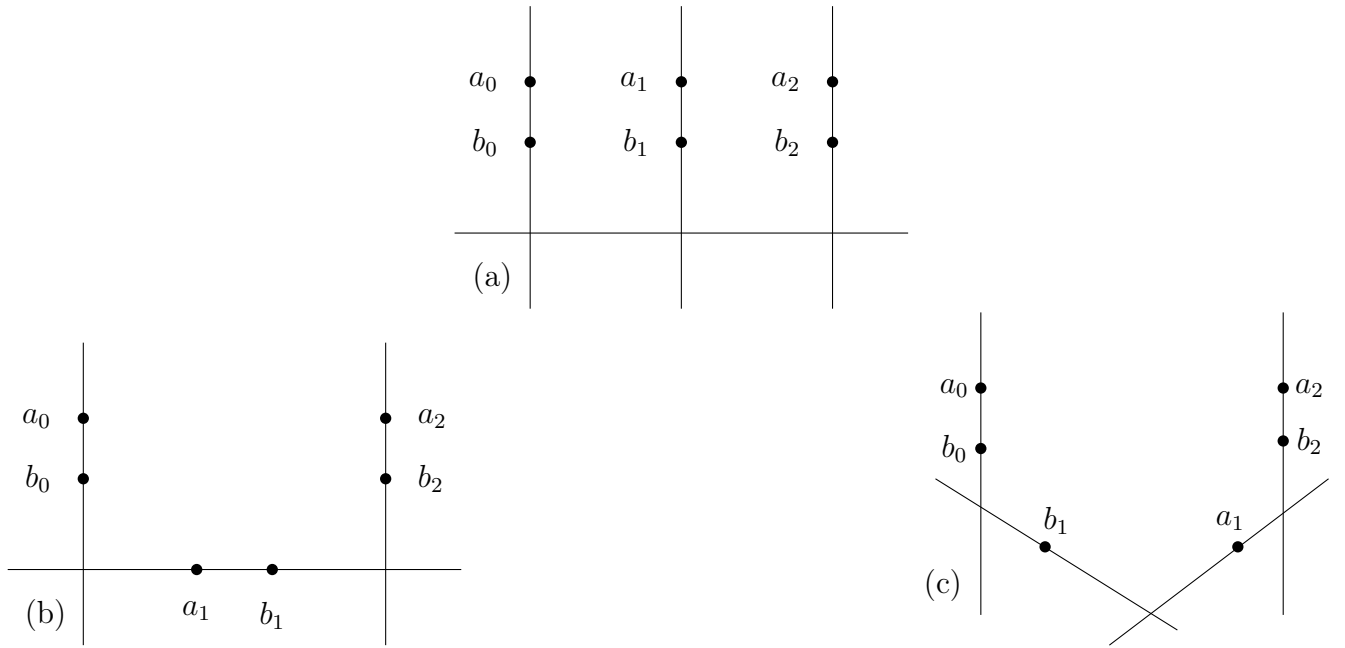

 \subsubsection{\rm Case (a)}\label{5.3.1}
    This case, already treated in Example~\ref{example5.2.1}, is  an example of the ``closed disk condition'' for $\Gamma$. The latter property  is defined by: \\
     $\Gamma$ is generated by involutions $s_0,\dots ,s_g$ with fixed points $\{a_0,b_0\},\dots , \{a_g,b_g\}$ and 
     there are disjoint  closed disks $D_0,\dots, D_g$ such that $a_i,b_i\in D_i$ for all $i$.  
     
     In terms of configuration $(P,m)$ this property translates into:\\
     the end vertices of the tree $P$ are $v_0,\dots ,v_g$ and $m(a_i)=m(b_i)=v_i$ for all $i$. \\
     
     In this  ``closed disk case'', $\overline{X}^d$ is obtained from $\overline{(\mathbb{P}^1,\mathbb{B})}^d$ by doubling every edge and every vertex
     which is not an end vertex.  There are only two possibilities for a fundamental domain, namely $F_0$ and $\sigma F_0$. 
     
     Here $F_0$ is obtained  by deleting the end vertices and end edges of $\overline{(\mathbb{P}^1,\mathbb{B})}^d$ and  taking
     the trivial degree two covering.  Next one omits one of the two connected components and glues  back the end vertices.  
     This produces a presentation of $\Gamma$  by independent elements $s_0,\dots ,s_g$ of order two. This presentation is unique up
     to simultaneous conjugation. 
     
     Finally,  the condition ``restricted'' pointed $\mathbb{B}$-tree is  no restriction at all.   \hfill $\square$

   \subsubsection{\rm Case (b)} A more general configuration is studied in detail in (5.2.2).
   We continue the special case of  (5.2.2), namely the  configuration $(P,m)$ with the set of vertices  $V=\{v_0,v_1,\dots ,v_g\}$, edges
  $E=\{ (v_0,v_i)\ |\ i=1,\dots ,g\}$ and $m (a_i), m (b_i)\in v_i$ for $i=0,\dots ,g$.
  The Whittaker curve $X$ has reduction $\overline{X}$ with  components
  $\tilde{L}_i, \ i=0,\dots ,g$. Here  $\tilde{L}_i\rightarrow L_i$ is the degree two covering with branch points 
  the images of $a_i,b_i$.
  
  Choose the parameter $z$ for $L_0\cong \mathbb{P}^1_k$ such that $a_0,b_0$ are mapped to $z=0,\infty$
  on $L_0$. The nodes $e_1=L_0\cap L_1,\dots , e_g=L_0\cap L_g$ on $L_0$ have coordinates $t_1,\dots ,t_g$.  For $\tilde{L}_0$ we take a parameter
  $u$ with $u^2=z$. The nodes on $\tilde{L}_0$ have coordinates $u_1, -u_1,\dots ,u_g,-u_g$ with $u_i^2=t_i$.
  Then $\tilde{L}_i\cap \tilde{L}_0=\{ u_i,-u_i\}$ for $i=1,\dots ,g$.  We  normalize further by taking $t_1=1$ and $u_1=1$.
  The $t_i$ are in $k$ but, in general, $u_i\not \in k$ for $i>1$. This illustrates Remarks~\ref{Rem3.2}.\\
     
     The  hyperelliptic involution $\sigma$ acts on $\overline{X}$ as follows. Each $\tilde{L}_i$ is invariant
     under $\sigma$. Further, the images of $a_i,b_i$ are the fixed points of the restriction of $\sigma$ to $\tilde{L}_i$.
     In particular, $\sigma$ permutes each pair $\{u_i,-u_i\}$. \\
     
     A {\it fundamental domain $F\subset \Omega$} is determined by its reduction $\overline{F}\subset  \overline{\Omega}$. The latter is derived from 
     $\overline{X}\setminus \{ n_1,\dots , n_g \}$, where, for each $i$, $n_i$ is  of the two nodes of $\overline{X}$ lying over the node $L_0\cap L_i$.
     The preimage under the universal covering $\overline{\Omega}\rightarrow \overline{X}$ of  $\overline{X}\setminus \{ n_1,\dots , n_g \}$ is a disjoint   union of copies of $\overline{F}$. \\
     
        Thus $\overline{F}\cong \tilde{L}_0\cup \tilde{L}_1\cup \dots \cup \tilde{L}_g \setminus \{n_1,\dots ,n_g\}$ and, due to the
     choice of the omitted $\{n_i\}$, it  is a restricted $\mathbb{B}$-tree with configuration $(P,m)$.\\

      Fix a choice of the $n_i$. This produces pairs of fixed points and corresponding order two elements 
      $s_0,s_1,\dots ,s_g$ with $\Gamma =<s_0>*<s_1>*\cdots *<s_g>$.  Consider another choice $\{\tilde{n}_1,\dots ,\tilde{n}_g\}$ of the
      omitted nodes. This yields a fundamental domain $\tilde{F}$, fixed points, order two elements $\tilde{s}_i$ such that 
       $\Gamma =<\tilde{s}_0>*<\tilde{s}_1>*\cdots *<\tilde{s}_g>$. A computation (using  pictures)  shows that  
     \[ \tilde{s}_0=s_0, \mbox{ and for i}\neq 0\ \ \tilde{s}_i=s_i \mbox{ if }n_i=\tilde{n}_i\mbox{ and }\tilde{s}_i=s_0s_is_0^{-1}
      \mbox{ if } n_i\neq \tilde{n}_i.\]

 \subsubsection{\rm Case (c)}\label{5.3.3} This continues the  third example (5.2.3) with a slightly different notation. 
 The description of the configuration $(P,m)$ uses the notations:\\
  $\mathbb{B}=\{x_0,y_0\}\cup \{a_i,b_i\}_{i=1}^n\cup \{c_i,d_i\}_{i=1}^m$. The components of 
  $\overline{(\mathbb{P}^1_K,\mathbb{B})}$ are $L_0,M_0$,  $A_1,\dots ,A_n,C_1,\dots ,C_m$.
  The nodes are $\{A_i\cap L_0\}_{i=1}^n, L_0\cap M_0, \{M_0\cap  C_i\}_{i=1}^m$ and the $m$-images
  of $\mathbb{B}$ are $a_i,b_i\mapsto A_i$, $x_0\mapsto L_0, y_0\mapsto M_0$, $c_i,d_i\mapsto C_i$.\\
  
  The components of $\overline{X}$ are $\tilde{L}_0,\tilde{M}_0$, $\tilde{A}_i$ and $\tilde{C}_i$. These
  are degree $2$ coverings of the components of $\overline{(\mathbb{P}^1_K,\mathbb{B})}$.  
  Above the nodes $A_i\cap L_0$ and $M_0\cap C_i$ there are two nodes of $\overline{X}$. In general, these nodes
  involve a quadratic extension of $k$. In particular, the position of these nodes on
  $\tilde{L}_0$ and $\tilde{M}_0$ are in at most quadratic extensions of $k$.
  The covering of the node $L_0\cap M_0$ is ramified. This illustrates Remarks~\ref{Rem3.2}.\\ 
  
   The hyperelliptic involution $\sigma$ induces an involution on  every component of $\overline{X}$. 
   The node $\tilde{L}_0\cap \tilde{M}_0$ is invariant under $\sigma$ and the other nodes of $\overline{X}$
   are permuted pairwise. After normalizing $x_0=0, y_0=\infty$ and using the obvious parameters on
   $L_0,M_0, \tilde{L}_0,\tilde{M}_0$ one has:
    the nodes on $L_0$ are $f_1,\dots , f_n $; the nodes on $M_0$ are $g_1,\dots ,g_m$;
    the nodes on $\tilde{L}_0$ are $f_1^{1/2},-f_1^{1/2},\dots, f_n^{1/2},-f_n^{1/2}$ and the nodes on 
    $\tilde{M}_0$ are $g_1^{1/2},-g_1^{1/2},\dots , g_m^{1/2}, -g_m^{1/2}$.
    In the sequel we normalize by $ f_1^{1/2}=1$. \\

     A {\it fundamental domain} $F\subset \Omega$ is obtained by omitting, for each pair of nodes on $\tilde{L}_0$ 
     and $\tilde{M}_0$, one of the two. Let $\overline{F}\subset \overline{\Omega}$ denote the reduction of $F$ induced by the reduction of
     $\Omega$. The dual graph $\overline{F}^d$ of $\overline{F}$ coincides with the dual graph of 
     $\overline{(\mathbb{P}^1_K,\mathbb{B})}$. The space $\overline{F}$ is similar (but not equal) to
      $\overline{(\mathbb{P}^1_K,\mathbb{B})}\setminus Y$, where $Y$ is the union of the nodes on $\tilde{L}_0$ and on $\tilde{M}_0$
      which are chosen to be omitted in the construction of $F$.  For the omitted nodes on $\tilde{L}_0$ one has for $i\neq j$
      the inequalities   $(\pm f_i^{1/2})^2\neq (\pm f_j^{1/2})^2$ and similar inequalities for the omitted nodes on $\tilde{M}_0$.
      The $\mathbb{B}$-tree $\overline{F}$ is restricted of type $(P,m)$.\\

      For a fixed choice of $Y$ one obtains a fundamental domain $F$ and order two elements 
      $s_1,\ldots ,s_n, s_0$, $t_1,\ldots , t_m$ for the pairs
      of fixed points corresponding to  $\{a_i,b_i\}_i$, $\{x_0,y_0\}$, $\{c_j,d_j\}_j$ such that 
      \[ \Gamma =<s_1>*\cdots  *<s_n>*<s_0>*<t_1>*\cdots *<t_m>.\] 
       Any other fundamental domain, defined by deleting suitable nodes $Y$, produces elements of order 
       two obtained by conjugating some of the elements $\{s_1,\dots, t_m\}$ by $s_0$. 
       This situation is almost identical to the case (b) of \ref{(2.3.1)}.\\
              The above example is a case where a pair of branch points $\{x_0,y_0\}$ has the property that the vertices $v_1=m(x_0)$ and 
         $v_2=m(y_0)$ are of type (b) (i.e. with $\#m^{-1}(\{v_1\})=\# m^{-1}(\{v_2\})=1$). In this case the distance of $v_1,v_2$ in
         $\overline{(\mathbb{P}^1,\mathbb{B})}^d$ is 1.\\

          The distance between $v_1=m(x_0)$ and $v_2=m(x_1)$ can be greater than 1.
          {\it Example}: The space  $\overline{(\mathbb{P}^1,\mathbb{B})}$ has components $L_0,\dots ,L_5$ and nodes\\
         $L_0\cap L_1, L_1\cap L_2,L_2\cap L_3,L_2\cap L_4, L_4\cap L_5$. Further
         $m$ sends the pairs $\{a_0,b_0\}, \{a_1,b_1\}, \{a_2,b_2\}, \{a_3,b_3\}$ to $L_0, \{L_1,L_4\}, L_3,L_5$.  \\

           If the distance between $v_1$ and $v_2$ is greater than 1, then
          the shortest path from $v_1$ to $v_2$ contains only vertices $v$ of type (a) (i.e., $\# m^{-1}(v)=0$). A similar description of the 
          fundamental domain and the presentation of $\Gamma$, generated by order two elements can be computed.    \\

    \begin{remarks} {\it Metric configurations and clusters}.\\
     The ``geometry'' of a finite subset $S$ of $\mathbb{P}^1(K)$ (with $\# S\geq 3$) is described by a reduction
     $\overline{(\mathbb{P}^1, S)}$ and by the tree $\overline{(\mathbb{P}^1, S)}^d$. 
     
     For special $S=\{\{a_i,b_i\} \}_{i=0}^g$, 
     the latter is a configuration $(P,m)$ and it captures an essential part of the above reduction. 
     
     We define a {\it metric configuration} $(P,m)^+$ as a configuration $(P,m)$ together with a map from the edges 
     to the positive integers. 
     
     For the case of a local field $K$, the reduction $\overline{(\mathbb{P}^1, S)}$ defines (as before)
     $(P,m)$ and this map.  The latter sends any edge $e$ to $n\geq 1$, such that the local equation at the node corresponding to
      $e$ has the form $xy=\pi ^n$.    \\

  We compare the above to the concept of {\it clusters}, see  \cite{B---N,D-D-M-M}. 
  Now $S$ is a finite subset of the affine line 
  $\mathbb{A}^1(K)$ with $\# S\geq 3$. The cluster of $S$ consists of the subsets $\underline{s}$ of $S$ of the form 
  $D\cap S$ where $D\subset \mathbb{A}^1$ is a closed disk. To every $\underline{s}$, which has at least two elements,
  one associates the integer $d$ such that the maximal distance between elements of $\underline{s}$ is $|\pi |^d$.    
  The cluster of $S$ is (up to a shift of the numbers $d$) invariant under the transformations $z\mapsto az+b$.\\
  
 From a metric configuration and a choice of the position of the point $\infty$ in the projective line over $K$, one
 deduces a cluster for $S$. This cluster depends on the choice of the point $\infty$, since the set of closed disks
 depends on the position of $\infty$.  It can be shown  that any cluster is induced by a metric configuration and a location for $\infty$.
   
   For a metric configuration $(P,m)^+$ one can also define rigid spaces $Fix_{(P,m)^+}$ and $Branch_{(P,m)^+}$
   and a morphism $FB$ between these for suitably chosen ``metrics''.   \end{remarks}
      
 \section{\rm  Parametrization of Whittaker curves}\label{Section6}

 Let a set of fixed points $S:=\{\{ a_i, b_i\}\}_{i=0,\dots ,g}$ in good position, be given. This defines the groups
 $W \subset \Gamma=\langle s_0,\dots ,s_g\rangle$ and the Whittaker curve $\Omega/W$.
 
{\it The aim of this section is to parametrize the equation of $\Omega/W$  by theta functions for the groups $W$ and $\Gamma$.
We start by assuming that $\infty \in \Omega$ and that the $\Gamma$-orbit of $\infty$ is disjunct with $S$.}
 %
%
%
 These data induce the following commutative diagram where the  arrows are the natural maps
 \[ \begin{array}{ccc} \Omega & & \\  \downarrow& \searrow&  \\  \Omega/W& \rightarrow & \Omega/\Gamma  
 \\ \downarrow&  &\downarrow \\  X &\rightarrow  & \mathbb{P}^1  \end{array} \]
 and where $X=\Omega/W$ is a  hyperelliptic curve with function field $K(x,y)$ and equation $y^2=$ a polynomial in $x$ of degree $2g+2$. 
 The two lower downarrows are isomorphisms and the lower rightarrow is the map $(x,y)\mapsto x$.\\

 As in the proof of Theorem \ref{1.5}, the
 points $a,b\in \Omega$ are such that $\Gamma$-orbits of $a,b,a_0,b_0,\dots a_g,b_g$ and $\infty$ are all disjoint. 
 We consider the theta function $F(z)=\prod _{\gamma \in \Gamma} \frac{z-\gamma a}{z-\gamma b}$ 
 for the group $\Gamma$. Theorem \ref{1.5} implies the following.
\begin{proposition} \label{3.1}There exists a constant $c\in K^*$ such that
 $X=\Omega/W$ has affine equation $y^2=c\cdot \prod_{i=0}^{2g} (x-F(a_i))(x-F(b_i))$.  
 \end{proposition}
  
Just for a moment we write $F_{a,b}$ for $F$ in order to indicate that $F$ depends on $a,b$. The function field of the hyperelliptic curve $\Omega/W$ is a degree two extension of $K(F_{a,b})$. Consider another choice $a',b'$. The genus zero subfield of a hyperelliptic function field defining a degree $2$ extension is unique and thus $K(F_{a,b})=K(F_{a',b'})$ holds.  Thus $F_{a',b'}$ is a fractional linear expression in $F_{a,b}$. The dependence on $a,b$ of the constant $c$ and of the $W$-invariant function $y=H(z)$ that Proposition~\ref{3.1} promises,
can be made explicit in a similar way. \\    
 
The above construction and Proposition \ref{3.1} are well known (\cite{G-vdP,S}).  We proceed by developing  {\it an explicit formula}  $H(z)$ for the element $y$ in the function field. An alternative, related approach to this is presented in \cite[\S~4.2]{KMMP}, where theta functions
are viewed in terms of a ``cross ratio pairing'' on the group of zero cycles of degree $0$ on $\mathbb{P}^1(K)$, following \cite{MX}.\\

\noindent
{We start the search for a formula of $y$} by constructing a product of theta functions on $\Omega$ having the required zeros and poles.
The divisor of $y^2$ as function on $\mathbb{P}^1$ is 
 $\sum [F(a_i)]+[F(b_i)]-(2g+2)[\infty]$. 
 This divisor, pulled back to $\Omega$ and divided by 2 is the $W$-orbit of 
$\sum [a_i]+\sum [b_i] -(g+1)[b]-(g+1)[sb]$ with, say, $s=s_0$  (note that a divisor like $\sum _{\gamma \in \Gamma} [\gamma a_i]$ equals
$\sum_{w\in W}2 [wa_i]$ because $\Gamma = W \sqcup Ws_i$).

A product of theta-functions for $W$ with this divisor is
\[G(z):=\prod_{i=0}^g\prod_{w\in W}(\frac{z-w(a_i)}{z-w(b)} \cdot \frac{z-w(b_i)}{z-w(sb)}).\]    
 
 \noindent
 Indeed, write $\Theta_W(a,b;z)$ for  $\prod_{w\in W}\frac{z-w(a)}{z-w(b)}$, the standard theta function for $W$, then
 \[G(z)=\prod_{i=0}^g \Theta_W(a_i,b;z)\cdot \Theta_W(b_i,sb;z).\]
In general $G(z)$ will not be $W$-invariant. We have to multiply $G(z)$ by an invertible function on $\Omega$ in order to obtain a 
$W$-invariant function on $\Omega$ and thus a function on $X=\Omega/W$.
 For any $\delta \in \Gamma $, there is a constant $C(\delta)\in K^*$ such that 
 \[C(\delta) G(\delta z)= \prod_{i=0}^g \Theta_W(\delta^{-1}a_i,\delta^{-1}b;z)\cdot 
 \Theta_W(\delta^{-1}b_i,\delta^{-1}sb;z).\]
 This follows from the observation that $\frac{\delta z-c}{\delta z-d}=\eta \cdot \frac{z-\delta^{-1}c}{z-\delta^{-1}d}$ for some $\eta \in K^*$.\\

 For $\delta \in W$ one has $G(\delta z)=C_1(\delta)\cdot G(z)$ with $C_1(\delta)\in K^*$, since $G(z)$ is a product of theta functions for 
 $W$. Now we try to compute $G(s_0z)$, up to constants.
 Up to constants, we may replace $G(s_0z)$ by the product 
 \[\prod _{i=0}^g \prod_{w\in W}\frac{z-s_0w(a_i)}{z-s_0w(b)}\cdot \frac{z-s_0w(b_i)}{z-s_0ws_0(b)}.\]
 {\it For the terms in this product with $i=0$}, one writes $s_0w=vs_0$ and one obtains
 $\prod _{v\in W} \frac{z-va_0}{z-vs_0b}\cdot \frac{z-vb_0}{z-vb}$. This is equal to the original product.\\
 
 \noindent {For the terms with $i=1$}, one writes $s_0w=vs_1$ and this gives
 \[ (\prod _{v\in W}\frac{z-va_1}{z-vs_0b}\cdot \frac{z-vb_1}{z-vb})\cdot (\prod_{v\in W}  \frac{z-vs_0b}{z-v\gamma_1s_0b})\cdot
 (\prod _{v\in W}\frac{z-vb}{z-v\gamma_1b})\mbox{ with } \gamma_1=s_1s_0.\] 
 For $\alpha \in W$ and any $\omega \in \Omega$, write $u_\alpha(z):=\Theta_W(\omega ,\alpha(\omega);z)$.
 Recall from \cite[Chapter 2]{G-vdP} that $u_\alpha$ is invertible on $\Omega$ and does not depend on the 
 choice of $\omega$. Moreover, $u_\alpha \cdot u_\beta=u_{\alpha \beta}$.\\  
 
 Thus we find for the terms with $i=1$  the extra factor $u_{\gamma_1}^2$ where $\gamma_1=s_1s_0$.  More generally, for any 
 $j\in \{1,\dots , g\}$, the terms in the product with $i=j$ produce the extra factor $u_{\gamma_j}^2$ (where $\gamma_j=s_js_0$). 
 Thus we obtain the formula 
 \[G(s_0z)=cst\cdot G(z)\cdot u_{\gamma}(z)^2, \mbox{ where } \gamma =\gamma_1\cdots \gamma _g\mbox{  and } cst \mbox{ is some constant}.\] 
 We want to multiply $G(z)$ by an invertible $u_\alpha(z)$ with $\alpha \in W$ such that $G(z) u_\alpha(z)$ is, up to a constant factor, invariant under $s_0$. One computes the formula  $u_\alpha(s_0z)=cst\cdot u_{s_0\alpha s_0^{-1}}(z)$ for some $cst\in K^*$.    Now
 \[G(s_0z)u_\alpha(s_0z)=cst\cdot G(z)u_\alpha(z)\cdot u_\gamma(z)^2\cdot u_{s_0\alpha s_0^{-1}\alpha^{-1}}(z) \mbox{ and thus we want }\]
 $\gamma^2\cdot s_0\alpha s_0^{-1}\alpha ^{-1}$ to be equal to 1 in $W_{ab}$. Observing  that
 $s_0\gamma_is_0^{-1}=\gamma_i^{-1}$, one finds that $\alpha =\gamma_1\cdots \gamma_g$ has the required property.\\
 
 \noindent {\it Conclusion}: The function $H(z)=G(z)u_\gamma(z)$ with $\gamma=\gamma_1\cdots \gamma_g$ is up to 
 constants invariant under $\Gamma$.   Hence, there  is a homomorphism 
 $c\colon \delta \in \Gamma \mapsto \{\pm 1\}$ such that $H(\delta z)c(\delta)=H(z)$ holds for every $\delta \in \Gamma$. \\

\begin{proposition} \label{3.2} \begin{itemize}
\item[{\rm (1).}] $H=H(z):=u_{\gamma_1\dots \gamma_g}(z)\cdot \prod _{i=0}^g\Theta_W(a_i,b;z)\cdot \Theta_W(b_i,s_0b;z)$
 is $W$-invariant. 
\item[{\rm (2).}]  Consider $H$ as element in the function field $K(x,y)$ of 
 $X=\Omega /W$. Then \[H^2=c\prod (x-F(a_i))(x-F(b_i))\] with 
 $c^{-1}=\prod (1-F(a_i))(1-F(b_i))\in K^*$. 
\item[{\rm (3).}] The affine equation 
\[y^2=\prod_{i=0}^g \frac{(x-F(a_i))}{(1-F(a_i))}\cdot \frac{(x-F(b_i))}{(1-F(b_i))} \]
for the curve $X=\Omega/W$ is parametrized by $z\in \Omega \mapsto (F(z),H(z))$.
\end{itemize}
\end{proposition}
\begin{proof} (1) and (2). The element $y$ in the function field of $X$ is induced by a unique $W$-invariant meromorphic function $\tilde{y}$ on 
$\Omega$. Since the involution of $X$ maps $y$ to $-y$, one has $\tilde{y}(s_iz)=-\tilde{y}(z)$ for $i=0,\dots ,g$. Consider now
$C(z):=\frac{H(z)}{\tilde{y}(z)}$.

 By construction, $C(z)$ has no poles and no zero's. Furthermore,  for every
$\gamma \in \Gamma$ there is a constant $cst(\gamma)$ such that $C(\gamma z) =C(z)\cdot cst(\gamma)$ holds. Since $\Gamma$ is generated
by elements of order~2, one has $cst(\gamma)\in \{\pm 1\}$. 

 Then $C(z)^2$ is $\Gamma$-invariant and has no zeros and no poles. Thus $C(z)^2$ is constant. The function $C(z)$ is holomorphic on $\Omega$
 and $\Omega$ is connected. Therefore $C(z)$ itself is constant. Because $\tilde{y}$ is $W$-invariant, also $H$ is $W$-invariant.  
 By construction the formula  for $H^2$ holds.
  Since $\infty \in \Omega$ we may put $z=\infty$ in (2) and obtain the value of $c$. Part (3) is obvious.
  \end{proof}
  
  \begin{observation}\label{g2geval} The terms $F$ and $H$ in the affine equation for $\Omega/W$ depend on the choice of points
  $a,b\in \Omega$. However,   the image $\bar{c}$ of $c$ in $K^*/(K^*)^2$ is independent of this choice.
For genus $2$ and for various genus $3$ configurations we verified using Proposition~\ref{3.2}(2), that
$\bar{c}=1$. We expect that this holds in general for Whittaker groups defined over $K$.
\end{observation}
   \color{black}
   
   \bigskip
   
 {\it For completeness we develop now formulas for the equation of $\Omega/W$  
 for the case that $\infty$ is a branch point}. \\

In general, one adopts the notational conventions $\frac{z-a}{z-b}=1$ if $a=b=\infty$, $\frac{z-a}{z-\infty}=z-a$ if $a\neq \infty$,
and $\frac{z-\infty}{z-b}=\frac{1}{z-b}$ if $b\neq \infty$.\\

As before the points $a_0,b_0,\ldots ,a_g,b_g $ are supposed to be in good position and the $\Gamma$-orbits of $a_0,b_0,a_1,b_1,\ldots ,a_g,b_g, \infty$ are distinct.  
 For the  $a,b$  we make the choice $a=a_0, b=b_0$. Then 
 $F(z)=\prod _{\gamma \in \Gamma} \frac{z-\gamma a_0}{z-\gamma b_0}$ is well defined, convergent and the same reasoning
 as before shows that it is a $\Gamma$-invariant function (in the proof one allows $a_0$ and $b_0$ to move and one uses connectedness).
 The branch points are again the $F(a_0), \dots , F(b_g)$ and now  $F(a_0)=0$ and $F(b_0)=\infty$. The equation for the hyperelliptic 
 curve can be written as $y^2=c\cdot x\cdot \prod_{j=1}^g (x-F(a_j))(x-F(b_j))$.   \\

 The next item is to produce a product formula for $y$ as function on $\Omega$.   The divisor of $y$ as function on $\Omega$
 is the $W$-orbit of
  \[ [a_0]+\sum_{j=1}^g [a_j]+\sum _{j=1}^g[b_j]-(2g+1)[b_0].\]
  A product of theta functions for $W$ with this divisor is
 \[ G(z):=\prod _{w\in W} \frac{z-w(a_0)}{z-w(b_0)}\cdot  \prod _{j=1}^g \prod_{w\in W} \frac{z-w(a_j)}{z-w(b_0)} 
  \cdot \frac{z-w(b_j)}{z-w(b_0)} .\]
 Thus for any $\tilde{w} \in W$ there is a constant $c(\tilde{w})$ with  $c(\tilde{w})G(\tilde{w} z)=G(z)$. For $\delta \in \Gamma$ there is a 
 constant $C(\delta)$ such that     
 \[C(\delta) G(\delta z):=\prod _{w\in W} \frac{z-\delta^{-1} w(a_0)}{z-\delta^{-1}w(b_0)}\cdot 
   \prod _{j=1}^g \prod_{w\in W} \frac{z-\delta^{-1}w(a_j)}{z-\delta^{-1}w(b_0)} 
  \cdot \frac{z-\delta^{-1}w(b_j)}{z-\delta^{-1}w(b_0)} .\]         
 For the evaluation of $G(\delta z)$ up to constants for any $\delta \in \Gamma$ it suffices to consider $\delta=s_0$ and the right hand side of the above
 formula.  
 
 For the first part $\prod \frac{z-s_0w(a_0)}{z-s_0w(b_0)}$, write $s_0w=\tilde{w}s_0$. Since $a_0,b_0$ are fixed points for $s_0$,
 this part of the formula remains unchanged.
 
 For the part of the formula involving $a_1$ one writes $s_0w=\tilde{w}s_1$ and observes that
 $\prod _{\tilde{w}}\frac{z-\tilde{w}s_1a_1}{z-\tilde{w}s_1b_0}$ is equal to  
  $\prod _{\tilde{w}}\frac{z-\tilde{w}a_1}{z-\tilde{w}b_0}$  multiplied by  $\prod _{\tilde{w}}\frac{z-\tilde{w}b_0}{z-\tilde{w}\gamma_1b_0}$, where
  $\gamma_1=s_1s_0$. The last term is the invertible function $u_{\gamma_1}$ on $\Omega$. 
  
  The $b_1$ terms of the formula also produce the extra term $u_{\gamma_1}$. In fact, for each $j\in \{1,\dots ,g\}$ the terms involving $a_j$ and $b_j$
  produce an extra term $u_{\gamma_j}$. Write $\gamma =\gamma_1\cdots \gamma_g$. Then we have shown that 
  $G(s_0z)$ is up to a constant equal to $G(z)\cdot u_{\gamma}^2$. As in the previous part of the manuscript one obtains:
  $H(z):=G(z)\cdot   u_{\gamma}$ is up to constants invariant under $\Gamma$ and finally   
 
 \begin{proposition} \label{3.3} \begin{itemize}
 \item[{\rm (1)}.] The meromorphic function $H=H(z)$ on $\Omega$ defined by the formula 
 \[u_{\gamma}(z)\cdot\Theta_W(a_0,b_0;z)\cdot\prod _{j=1}^g \Theta_W(a_j,b_0;z) \Theta_W(b_j,b_0;z)\]  is $W$-invariant.
\item[{\rm (2)}.] The relation $H^2=c \cdot F(z)\prod _{j=1}^g(F(z)-F(a_j))(F(z)-F(b_j))$ holds for certain $c\in K^*$.
 \item[{\rm (3)}.] The constant $c$ can be computed by considering a suitable value for $z$. The affine equation of $X=\Omega/W$ is uniformized by $z\in \Omega \mapsto (F(z),H(z))\in X$. 
\end{itemize}
 \end{proposition}
 
 \noindent
 {\it Comments}.\\
 Above we have not specified the good position of the $a_0,b_0,\ldots ,a_g,b_g$. But we have specified that
 $0,\infty$ are branch points. Formulas, similar to the ones for $F(z)$ and $H(z)$,  can  be obtained from the previous      
 formulas by  a fractional linear  transformation sending the previous $(F(a_0),F(b_0))$ to $(0,\infty)$.

   \begin{observations} \label{convergence} {\it Convergence and approximation of theta functions.}\\
   Fix a point $\infty$ in $\mathbb{P}^1$. This induces an absolute value $z\mapsto |z |$ on $\mathbb{P}^1\setminus \{\infty \}$.
   This absolute value is unique up to a transformation $z\mapsto |a\cdot z+b|$. For any three distinct points  
    $x_1,x_2,x_3, \in \mathbb{P}^1\setminus \{\infty \}$, the expression $\frac{|x_2-x_3|}{|x_1-x_2|}$ is independent of the choice
    of the absolute value. One easily verifies the following:\\
    
 \noindent   {\it  $\frac{|x_2-x_3|}{|x_1-x_2|}<1$ if and only if the reduction $\overline{(\mathbb{P}^1,\{x_1,x_2,x_3,\infty\})}$
has components $L_1,L_2$,  the images of $x_1,\infty$ are on $L_1$ and the images of $x_2,x_3$ are on $L_2$.\\

\noindent Suppose $\frac{|x_2-x_3|}{|x_1-x_2|}<1$.  Then the  formal local ring at the double point $L_1\cap L_2$ has the form
 $K^o[[x,y]]/(xy -\alpha)$ with $0<|\alpha |<1$ and  $\frac{|x_2-x_3|}{|x_1-x_2|}=|\alpha |$. }\\

  In the general case, {\it we will call $|\alpha |$ the  size of the double point in the reduction with formal local ring
  $K^o[[x,y]]/(xy-\alpha)$.}  In the case that the valuation of $K$ is discrete, we fix a uniformizing element  $\pi$ and
  write $|\alpha|=|\pi |^n$ with $n\geq 1$. Then the dual graph of a reduction can be made into a ``metric graph'' by attaching to each
  double point the above integer $n$.\\    

  The theta function for the group $\Gamma$ has the form $\prod_{\gamma \in \Gamma} \frac{z-\gamma (a)}{z-\gamma (b)}$.
  We assume, as in Theorem \ref{2.2}, that a fundamental domain $F$ is chosen. This yields $\Gamma =<s_0>*\cdots  *<s_{g}>$.
  Further we consider $a,b,\infty \in F$.  Let $\pi \in K, 0<|\pi |<1$ satisfy $|\pi |$ is greater than or equal to all the sizes of the double
  points of $\overline{\Omega}$.\\
  
  Consider $\frac{z-\gamma (a)}{z-\gamma (b)}-1=\frac{\gamma(b) -\gamma (a)}{z-\gamma (b)}$ with $z$ in the fundamental domain
  $F$. 
  \[ \mbox{ For $\gamma \in \Gamma$ with length $\geq 2$ one has } |\frac{z-\gamma (a)}{z-\gamma (b)}-1|\leq |\pi |^{-1+length(\gamma)}.\]
     
   \begin{proof} Write $x_1=z, x_2=\gamma(a), x_3=\gamma (b)$. Since $x_1,\infty \in F$, $x_2,x_3\in \gamma F$ and
   $F\cap \gamma F=\emptyset$ we conclude that the position of $x_1,\infty, x_2,x_3$ is such that 
    $\frac{|x_2-x_3|}{|x_1-x_2|}<1$. The reduction  \[\overline{(\mathbb{P}^1,\{x_1,x_2,x_3,\infty\})}=L_1\cup L_2\] is obtained from the
    reduction of $\overline{\Omega}$ by contracting all components  not involving $\{x_1,x_2,x_3,\infty\}$.  The size of the
    double point $L_1\cap L_2$ is equal to the product of the sizes of the double points in $\overline{\Omega}$ in a (minimal)
    path from to image of $x_1$ to the image of $x_2$ or $x_3$. According to  the comments on the proof of Theorem \ref{2.2},
    the length of this path is greater than or equal to $-1+length(t)$. This proves the above inequality. \end{proof} 
    
    A similar estimate for the terms in a theta function for the group $W$ holds. From these estimates the convergence of
    the theta functions follows and the approximations are made explicit.      \end{observations}

 \section{\rm Good position for a set $S$ of fixed points}\label{Section7} 
   
 First we discuss a necessary condition for ``good position''.  

\begin{lemma}\label{4.1}  Let $S=\{\{a_i,b_i\}\}_{i=0,\dots , g}$ denote a set of fixed points.\\ Write $R_S\colon \mathbb{P}^1_K\rightarrow \overline{(\mathbb{P}^1_K,S)}$ for the reduction. For any irreducible component $L$ of $\overline{(\mathbb{P}^1,S)}$ we put 
 $R_L\colon \mathbb{P}^1_K\stackrel{R_S}{\rightarrow} \overline{(\mathbb{P}^1_K,S)}\stackrel{pr}{\rightarrow}L$ . \\
 
Suppose that $S$ is in good position. Then any irreducible component $L$ in the reduction {\rm separates} at most 
one pair $\{a_i,b_i\}$. \\ 
\indent In other words, the divisor $\sum _im_i[p_i]:= R_L(S)$ on $L$, satisfies:\\
(i) all $m_i$ are even, {\rm or}\\
(ii)  $[R_L(a_i)]+[R_L(b_i)]+\sum n_j[p_j]$ for some $i\in \{0,\cdots ,g\}$ and all $n_j$ even.
\end{lemma}
\begin{proof} It suffices to consider the case $g=1$. For the tree of lines  $\overline{(\mathbb{P}^1_K,S)}$, there are the following possibilities:\\
(a) $L_1\cup L_2$ with $a_0,b_0$ mapped to $L_1$ and $a_1,b_1$ mapped to $L_2$.\\
(b) $L_1\cup L_2$ with $a_0,b_1$ mapped to $L_1$ and $a_1,b_0$ mapped to $L_2$.\\
(c) $L$ with the images of all four points.\\
A matrix computation shows that (a) is the only case of  good position.
\end{proof}

\noindent {\it Comments}.  Lemma \ref{4.1} states that a necessary condition for the position of the set of fixed point $S$ coincides with
the necessary and sufficient condition for the position of the branch locus of a potential Whittaker  curve. This follows from
Theorem~\ref{1.6} and Comments~\ref{newpairs}, stating that the branch locus has a canonical  division into pairs.\\

\noindent  {\it An example of bad position}. \label{bad position} The condition in Lemma \ref{4.1} is not sufficient for ``good position''. We consider the case  $g=2$, where 
there are three possibilities (a), (b) and (c) for $S$ satisfying 
the necessary condition of Lemma \ref{4.1} (see page~\pageref{picture2}).  Possibility (b) is given by the following (in)equalities.
\[ S=\{ \{0,b_0\}, \{ 1,b_1\}, \{b_2,\infty \}\}$ with $0<|b_0|<1,\ |b_1|=|b_1-1|=1,\ 1< |b_2|.\]
The reduction of $S$ consists of three lines $L_0,L_1,L_2$ and two nodes $L_0\cap L_1$ and $L_1\cap L_2$. Further $\{0, b_0\}$, $\{1,b_1\}$ and $\{b_2,\infty\}$ are mapped respectively on $L_0, L_1$ and $L_2$ (see page~\pageref{picture2}, Figure~1(b)).  For the values $b_1=-1, b_2=b_0^{-1}$ one finds $s_1s_2s_1=s_0$ and
for these values $S$ is not in good position.\\

{\it More generally}, also for values of $b_1$ close to $-1$, the pairs are not in good position. We note that
 ``restricted'' is equivalent to  $|b_1+1|=1$, see p.~\pageref{example5.2.2} (5.2.2) Second example. The condition $|b_1+1|=1$  is {\it sufficient }
  for good position, see Proposition~\ref{5.2} and Proposition~\ref{Prop7.3}.

 It seems impossible in this special case to give a ``closed formula'' for the $b_1$ such that $S$ has good position. However,  an algorithm for identifying good position in case (b)  can be constructed.\hfill $\square$\\

\subsection{\rm Examples of good position of fixed points}
 \subsubsection{\it The closed disk position}
  Suppose we fix a point $\infty\in \mathbb{P}^1_K$.  Let $S=\{ \{a_i,b_i\}\}_{i=0,\dots , g}$ be contained in $K$ and suppose that
 there are disjunct closed disks $B_0^+,\dots ,B_g^+$ such that $\{a_i,b_i\}\subset B_i^+$ for all $i$. This defines the closed disk condition. We omit the proof
 that it is a good position.
 
      We note that this condition (for a suitable choice of $\infty$) is equivalent to the statement: The endlines of the tree $\overline{(\mathbb{P}^1_K,S)}$ are 
      $L_0,\dots , L_g$ and for $i=0,\dots ,g$, the images of $\{a_i,b_i\}$ are on $L_i$. 
      
  \noindent   {\it Example}. For $g=2$, the case (a) (see page~\pageref{picture2}) is the closed disk position. It is defined by the following. The tree $\overline{(\mathbb{P}^1_K,S)}$ has lines
      $L_0,L_1,L_2,L$ and nodes $L_i\cap L$ for $i=0,1,2$. For $i=0,1,2$ the images of $a_i,b_i$ are on $L_i$. \\
      
\subsubsection{\it The thesis of G.~Van Steen}
In \cite[p.~76]{S} G.~Van~Steen presents
 complicated conditions $G_1-G_4$ that are sufficient and ``almost'' necessary. The thesis has two statements
 with detailed, complete proofs:\\
 (i). For any Whittaker group, there is a choice of generators $s_0,\dots ,s_g$ such that $G_1-G_4$ holds.\\
 (ii). Fixed points $\{\{a_i,b_i\}\}$ satisfying  $G_1-G_4$ are in good position.\\

 
 In $G_1-G_4$ the following method and notations are used. One makes a choice for $\infty \in \mathbb{P}_K^1$. 
 Let $s$ be the element of order 2 with fixed points $a,b\neq \infty$. Consider the projective line
 $red_{\{a,b,\infty\}}$ over the residue field. On it are the points $\bar{a},\bar{b},\bar{\infty} ,\bar{s(\infty)}$. The preimage
 in $\mathbb{P}^1_K$ of $\bar{s(\infty)}$ is an open disk and is denoted by $B_s$.  

 We note that $\mathbb{P}_K^1\setminus B_s$ can be interpreted as a fundamental domain for the group $<s>$. The position
 of this fundamental domain depends on the choice of $\infty$. Write $B_i=B_{s_i}$ for $i=0,\dots ,g$. \\
 
 \noindent 
 $G_1$ reads: for $i\neq j$ one has $B_i\cap B_j=\emptyset$.\\
 Let $B_i^+$ denote the closed disk obtained from $B_i$ and put $\partial B_i=B_i^+\setminus B_i$.
 If the closed disks are mutually disjoint, then $G_2-G_4$ are satisfied. This is again the {\it ``closed disk position''}. \\
 
 
 Now we copy from Van Steen's thesis:\\
 $G_2$: $\forall i,j,k \in \{0,\dots ,g\}$ such that $B_i^+\cup B_j^+\subset \partial B_k$
  one has  $s_k(B_i^+)\cap B_j^+=\emptyset$.\\
 $G_3$: $\forall i,j,k \in \{0,\dots ,g\}$ such that $B_i^+\cup B_j^+\subset \partial B_k$
one has $|s_i(\infty )-s_ks_j(\infty)|\geq |s_i(\infty )-s_ks_i(\infty)|$.\\  
 $G_4$: $\forall i,k\in \{0,\dots ,g\}$ such that $B_i^+ \subset \partial B_k$
one has $s_k(B_i^+)\cap   B_i^+=\emptyset$.\\
The last condition is the special case of $G_2$ with $i=j$ and expresses that $s_is_k$ is hyperbolic.\\

\noindent{\it Remarks}.\\
(i). Suppose that $G_1-G_4$ hold. Then $\mathbb{P}^1\setminus B_0\cup \dots \cup B_g$ can be interpreted   as  fundamental domain 
for $<s_0,\dots ,s_g>$. In general, this differs from the fundamental domain constructed in Section~\ref{Section4}.   \\
(ii). The conditions in Theorem~5.7 of  the thesis of Kadziela \cite{Kad1},  define in fact the closed disk situation.\\
(iii). The set of fixed points, the set of ramification points, their positions, $W$ and $\Gamma$ are considered up to 
the action of $PGL(2,K)$.  For, say, the set of fixed points $S$ , one takes representatives for this action by
identifying three  elements of $S$ with $0,1,\infty$. Different choices of identification lead to slightly different formulas.

In \cite{Kad1}, choosing representatives is done as follows: Let $S$ satisfy the necessary condition of Lemma \ref{4.1}. The tree
$\overline{(\mathbb{P}^1_K,S)}$ has at least two end lines. Each end line $L$ has three special points, namely the images of $a_i,b_i$ for a unique $i$ and one node. Now one choices two end lines $L_1,L_2$ with corresponding points $a_i,b_i$ and $a_j,b_j$. The normalization, used by Kadziela, is  defined by $a_i=0, a_j=1,b_j=\infty$.  This normalization of $S$ will be called the {\it Rosenhain position}.\\

\subsubsection{\it Realizing finite trees of finite groups}
 A finite tree of finite groups $(T,G)$ consists of a finite tree $T$, for every vertex $v$ a finite group $G_v$ and for every edge a finite
 group $G_e$. Further, a morphism $G_e\rightarrow G_v$ is given if $v$ is a vertex of the edge $e$. 
 
 A realization of $(T,G)$ in
 $\mathbb{P}^1$ is a morphism $\tau$ of $T$ into the (generalized) Bruhat-Tits tree of $\mathbb{P}_K^1$ and morphisms of
 $G_v$ and $G_e$ to the stabilizers of $\tau(v)$ and $\tau(e)$ in $PGL_{2,K}$. One requires that the image of $\tau$ generates a
 discontinuous group and that the only relations between the generators are those defined by the object $(T,G)$.\\
 
 From a configuration $(P,m)$ one obtains a rather special finite tree of finite groups $(T,G)$. Here $T$ is the tree of $(P,m)$. 
 For a vertex $v$ one defines  $G_v=\{\pm 1\}$ if $v$ is odd and $G_v=\{1\}$ if $v$ is even. 
 For an edge $e$ one defines $G_e=\{\pm 1\}$ if $e$ is odd  and $G_e=\{1\}$ is $e$ is even. Further, each morphism 
 $G_e\rightarrow G_v$ is injective.  A $\mathbb{B}$-tree with configuration $(P,m)$  can be seen as a realization $\tau$  of $(T,G)$.\\

  {\it  An explicit proof that  ``restricted'' implies ``good position'' } is possible by applying criteria for realization
  of finite trees of finite groups. These criteria have been developed by G.~Van Steen, F.~Herrlich, H.H.~Voskuil et al., see
  \cite{S,He,vdP-V}. Most of these results are rather technical and difficult to use for proving ``good position''.  However, we will apply 
  the following result of \cite[Theorem 3.3. part (2)]{vdP-V}:\\
 
 \noindent  (*) {\it
     Let a finite tree of groups $(T ,G)$ be given and let $e = \{v_1, v_2\}$ be an edge. Let 
     $(T_ 1,G_1), (T_ 2,G_2)$ denote the trees of groups obtained by deleting the
edge $e$ and suppose $v_1 \in T_1, v_2 \in T_ 2$. An embedding $\tau$ of $(T ,G)$ is a realization if :\\
(a). The restriction of $\tau$ to   $(T_1,G_1)$ and  $(T_2,G_2)$ is a realization.\\
Let $\Gamma_1, \Gamma_ 2$ denote the resulting discontinuous subgroups of ${\rm PGL}_2(K)$.\\
(b). A lattice class $V \in [\tau(v_1), \tau (v_2)]$ exists with $V\neq  \tau(v_1), \tau (v_2)$, such that for each
$g_i \in \Gamma_i \setminus  \tau(G_e)$, $i = 1, 2$, one has that $V\neq g_1V,g_2V$ and $V \in [g_1V,g_2V ]$.\\ }

Condition \textit{(b)} can be difficult to verify. We note that this condition is satisfied in the case that the double point $p$ of $\tau (T)$, corresponding to the edge $e$,  is not a fixed point of any element of $\Gamma _1\cup \Gamma _2 \setminus \{1\}$
(here we assume $G_e=\{1\}$).
Indeed, a lattice class $V$, as in \textit{(b)}, corresponds to a blow up of $\tau (T)$ at $p$, which replaces $p$ by a projective line.
Now $V$ is not  invariant under some element of $\Gamma _1\cup \Gamma _2 \setminus \{1\}$ since this holds for $p$.

The realization $\tau$ of a finite tree of groups $(T,G)$ produces a reduction 
$red\colon \mathbb{P}^1\rightarrow \overline{(\mathbb{P}^1, \tau)}$. A point
$q\in  \overline{(\mathbb{P}^1, \tau)}(k)$ of this reduction is called {\it exceptional} if it is the image of a fixed point
of an element $g\neq 1$ in the group generated by $(T,G)$. \\

\noindent 
We apply  the above theorem (*) with \textit{(b)}  replaced by the stronger condition \\
 \textit{(c)}. The double point $p$ of $\tau(T)$, corresponding to $e$, is not exceptional for the realizations of
 $(T_1,G_1)$ and $(T_2,G_2)$.\\

\noindent
{\it Example}. Define $(T,G)$ by $T$ has two vertices $v_1,v_2$ and one edge $e=\{v_1,v_2\}$. Further, assume that
$G_{v_1}\cong G_{v_2}\cong \{\pm 1\}$ and $G_{e}=\{1\}$. Consider a realization $\tau$ where $\tau (G_{v_i})=<s_i>$
and $s_i\in {\rm PGL}_2(K)$ is the involution with fixed points $a_i,b_i$ for $i=1,2$. The reduction 
$red\colon \mathbb{P}^1\rightarrow \overline{(\mathbb{P}^1,S)}$ with $S=\{a_1,b_1,a_2,b_2\}$ is supposed to consist of two lines $L_1,L_2$, 
intersecting normally at $p$. Further $\bar{a}_i=red(a_i), \bar{b}_i=red(b_i)\in L_i\setminus \{p\}$ for $i=1,2$.  
The action of $s_i$ on $L_i$ is defined by
its fixed points $\bar{a}_i,\bar{b}_i$ for $i=1,2$.  An easy computation shows that the set
of exceptional points is $\{\bar{a}_1,\bar{b}_1, \bar{a}_2,\bar{b}_2, s_1(p),s_2(p)\}$.  \\

\begin{proposition}\label{5.2} A restricted map $M\colon \mathbb{B}\rightarrow \mathbb{P}^1$ has good position. \end{proposition}

We recall that ``$M$ restricted'' means that the (pointed) $\mathbb{B}$-tree $\overline{(\mathbb{P}^1, M(\mathbb{B}) ) }$ is restricted.
{\it We sketch a proof,} by induction on $g$, of the following statement:\\

\noindent 
(**) {\it A restricted map $M\colon \mathbb{B}\rightarrow \mathbb{P}^1$ of genus $g$ has good position and its set of exceptional points on the reduction   $\overline{(\mathbb{P}^1, M(\mathbb{B}))}$ is contained 
in the union of the set of images of $\mathbb{B}$, the set of all double points and the set of all points $s_L(p)$, where $L$ is an odd component of the reduction, $s_L$ is the involution on $L$ and $p$ is a double point on $L$. \\ }

The case $g=1$ is the example above. Consider a restricted $M$ with configuration $(P,m)$ and genus $g$. 
Let $(T,G)$ be its associated finite tree of finite groups. We consider a suitable end vertex $v$ and and edge $e=(v,w)$. We may suppose that $a_g,b_g$ are mapped to the vertex $v$.

 Put $\mathbb{B}^*=\{a_0,b_0,\dots , a_{g-1},b_{g-1}\}$.
By deleting $e$ we obtain a tree $(T_1,G_1)$ consisting of the vertex $v$ and group $\{\pm 1\}$ and the tree $(T_2,G_2)$, obtained by deleting $v$. The first tree is clearly a realization and only the images of $a_g,b_g$ are exceptional points.
 The second tree corresponds to a restricted $M^*\colon \mathbb{B}^*\rightarrow \mathbb{P}^1$ and has, by the induction hypothesis,  good position and moreover
 the double point $p$ corresponding to the edge $e$ is not an exceptional point since $M$ is restricted.
 Thus the conditions (a) and (c) hold for $M$ and one concludes by the theorem (*) that $M$ has good position. 
  Let $s_g$ denote the involution with fixed points $a_g,b_g$. 
Write $\Gamma _2, \Gamma$ for the groups generated by $(T_2,G_2)$ and $(T,G)$. Then $\Gamma$ is the free product 
$\Gamma_2*<s_g>$.

   Let $L$ denote the component of the reduction of $M$ corresponding to $v$ and let $s_L$ denote the involution on $L$. The exceptional points 
 of $M$ are contained in the union of the following sets:  $E$, the exceptional points  of $M^*$ (known by the induction hypothesis);
 the fixed points of $s_L$ and its conjugates by elements in the group generated by $\Gamma_2$ and the fixed points of the
 elements $s_Lss_L^{-1}$ with $s$ an involution in $\Gamma_2$. An inspection of the tree generated by $(T,G)$ shows that the set of exceptional points is contained  $E\cup \{s_L(p), \bar{a}_g,\bar{b}_g\}$. This verifies the second part of (**).  \hfill $\square$\\

\bigskip

 \subsection{\rm The genus two cases}\label{subs7.2}

The necessary condition in Lemma \ref{4.1} for the set of branch points $S:=\{\{a_0,b_0\}, \{a_1,b_1\}, \{a_2,b_2\}\}$
leads to the three possibilities a), b) and c) of Figure~1 on page~\pageref{picture2} (see also \cite[ p.~281]{G-vdP}).

\bigskip
\noindent {\it Case a)}. This is the closed disks position. One can normalize this position in 48 ways to $a_0=0, a_1=1, a_2=\infty$. 
Normalization to Rosenhain position $a_0=0, a_2=1,b_2=\infty$ can be done in 24 ways.\\
\noindent  {\it Case b)}.  One can normalize this position to  $0<|b_0|<1=|b_1|=|b_1-1|<|b_2|$ and $a_0=0$, $a_1=1$, $a_2=\infty$,
in 24 ways. For the Rosenhain position $a_0=0, a_2=1, b_2=\infty$ there are 8 possibilities.\\ 
\noindent {\it Case c)}. We normalize this position to  $a_0=0, a_1=1, a_2=\infty$ and  $0<|b_0|<|b_1|<1<|b_2|$. There are 8 possibilities.
For the Rosenhain position $a_0=0, a_2=1,b_2=\infty$ there are also 8 possibilities.
 \begin{proposition}\label{Prop7.3} The three cases are normalized by  $a_0=0,a_1=1,a_2=\infty$.
  Good position holds for the cases a) and c) and
   for the case b) under the assumption $|b_1+1|=1$.
  In case b) with $|b_1+1|<1$, the position is in general not good.
 \end{proposition}
\begin{proof}  The first part follows from Proposition~\ref{5.2}. Another proof consists of verifying the conditions of
Herrlich's criterion for amalgames, see \cite{He}.

 A proof, using \cite{S}, goes as follows. The conditions $G_1- G_4$ of \cite{S} are formulated under the assumption that $\infty$ is not a fixed point. In order to arrive at that situation, one applies an automorphism $\phi$ of $\mathbb{P}^1_K$ of the form $\phi (0)=0,\phi(1)=1,\phi(\infty)= \lambda b_2$ with $|\lambda |=|\lambda -1|=1$ for suitable $\lambda$. One of the choices $\lambda =2,3,4$ (depending on the residue characteristic) is suitable and
 one can verify $G_1- G_4$.  \\

Finally, the case $b_1=-1$ is already studied in ``An example of bad position'' on page \pageref{bad position}. \end{proof}

\noindent{\it The examples in Kadziela's work \cite{Kad1,Kad2}.}\\
 These examples are Mumford curves over $\mathbb{Q}_5$ of genus 2 and have the property that the {\it branch points} 
 are in a position of type b) and moreover such that $|b_1+1|<1$.
The computations in \cite{Kad1} produce approximations for generators  $g_1,g_2$ of $W$. 
As before, $\Gamma =\langle s_0,s_1,s_\infty\rangle$ with $s_0=g_2s_\infty$, $s_1=g_1s_\infty$ and $s_\infty(z)=2-z$.  

The first example is $y^2=x(x-1)(x-5)(x-195)(x-125)$ in \cite{Kad1}. The ramification locus can be normalized to  
$\{0,5^2\}, \{1,39\}, \{1/5, \infty \}$ and this is case b) with $|b_1+1|<1$. The fixed points are approximately
$\{ \{5^4\cdot 2,5^3\},\{5\cdot 3,5\},\{1,\infty \}\}$. Division by $5$ puts this in the normalised form above and we conclude that
{\it the fixed points } are in position b) with a new $b_1$ satisfying $|b_1+1|=1$. This case of good position 
guarantees that the product formulas for the  corresponding theta functions are convergent.

  \section{\rm Theta functions for genus two curves.}\label{explicit1}
  
  Recall that
  the morphism $FB\colon Fix_{P,m}\rightarrow Branch_{P,m}$ is given by 
  \[(a_0,b_0,\dots ,a_g,b_g)\mapsto (F(a_0),F(b_0),\dots ,F(a_g),F(b_g))\] where $F$ is the standard theta function
  $F=\prod _{\gamma \in \Gamma}\frac{z-\gamma(a)}{z-\gamma (b)}$ and $\Gamma =<s_0,\dots ,s_g>$. Further $s_i$
  is, for $i=0,\dots ,g$, the reflection with fixed points $a_i,b_i$.\\
  
  For $g=2$, the elements $(a_0,\dots ,b_2)$ of  $Fix_{P,m}$ and of $Branch_{P,m}$ are normalized by $a_0=0, a_1=1$, $a_2=\infty$. 
  Moreover, we normalize $FB$ by 
   \[(b_0,b_1,b_2)\mapsto (\frac{F(b_0)}{F(b_0)-1}, \frac{F(b_1)}{F(b_1)-1}, \frac{F(b_2)}{F(b_2)-1}),\]
  and $F=\prod _{\gamma \in \Gamma}\frac{z-\gamma(0)}{z-\gamma (1)}$. 
  The theta function $F$ is approximated by $F^*$, which is the product taken over all $\gamma \in \Gamma$ with 
  lenght $\leq 2$, compare Observations~\ref{convergence}. For the three configurations $(a), (b), (c)$ with genus two, see Figure~1 on page~\pageref{picture2}, we compute an approximation of $FB$.

   \subsection{\rm Configuration (a)}
   Here  $Branch_{(a)}:=\{(B_0,B_1,B_2)\in  K^3|\ 0<|B_0|,|B_1|,|B_2|<1\}$ as $K$-analytic space and  
   \[a_0=0, b_0=B_0,a_1=1,b_1=1+B_1,a_2=\infty, b_2=B_2^{-1},0<|B_0|,|B_1|,|B_2|<1.\]
  $Fix_{(a)}$ is the same space since ``restriction'' is automatically satisfied.  The approximation of $FB$ is
 \[(B_0,B_1,B_2)\mapsto (2B_0(1+\epsilon),2B_1(1+\epsilon),2B_2(1+\epsilon)),\]
 where each $\epsilon$ denotes a function which is, in absolute value, strictly less than~1. Note that this
confirms a special case of Theorem~\ref{new5.8}:

  \begin{proposition}\label{new8.1}
$FB\colon Fix_{(a)}\rightarrow Branch_{(a)}$ is an analytic isomorphism.
\end{proposition}
 
\begin{remarks}  \mbox{ }\\
(1). The configuration (a) has been studied in detail by J.~Teitelbaum \cite[\S 3.2]{T}. A main result, Lemma~35,
states that the map 
\[\lambda \colon\mathcal{M}=\{ (p_1,p_2,p_3) \ | \ 0<|p_1|,|p_2|, |p_3|<1\} \rightarrow Branch_{(a)} \]
from the space of half periods, to the space of branch points is a rigid analytic isomorphism.  For $\lambda$ and its inverse explicit approximations are produced. These formulas are applied to the 
genus two modular curves $X_0(p)$ over $\mathbb{Q}_p$ for $p=23,29,31$. 

We note that $p_1,p_2,p_3$ are  certain theta functions evaluated at fixed points. Therefore there is an explicit morphism
$Fix_{(a)}\rightarrow \mathcal{M}$  such that $Fix_{(a)}\rightarrow \mathcal{M}\stackrel{\lambda}{\rightarrow} Branch_{(a)}$ 
is the isomorphism $FB$ of Proposition \ref{new8.1}. \\

\noindent (2). For completeness we also consider {\it the Rosenhain normalization } for configuration (a).  
Now $Fix_{(a)}$  is  the space of pairs of fixed points are $\{0,b_0\}$, $\{ a_1,b_1\}$, $\{ 1,\infty\}$ with $0<|b_0|<|b_1|=|a_1|<1$ and $a_1=b_1(1+T)$ and $0<|T|<1$.   Write $b_0=B_0B_1, b_1=B_1,a_1=B_1(1+T)$, then one has 
$ Fix_{(a)}=\{ (B_0,B_1,T)\ |\  0<|B_0|,|B_1|,|T| <1\}$. The space $Branch_{(a)}$ has the same description.

  An approximation of $F:=\prod \frac{z-\gamma (0)}{z-\gamma (1)}$ yields the formulas
 \[\frac{F(b_0)}{F(b_0)-1}=4b_0(1+\epsilon),\ \frac{F(b_1)}{F(b_1)-1}=2b_1(1+\epsilon),\ \frac{F(a_1)}{F(a_1)-1}=2a_1(1+\epsilon).\]
 In terms of the variables $B_0,B_1,T$, the map $FB$ reads
  \[(B_0,B_1,T)\mapsto (2B_0(1+\epsilon), 2B_1(1+\epsilon), T(1+\epsilon)).\]
 This is again an analytic isomorphism.  
 The two normalizations of the  moduli spaces $Fix_{(a)}$ and $Branch_{(a)}$ are not quite the same. 
 The remark in \S \ref{subs7.2} on normalization  explains the factor 2 by which the formulas differ.  \\ 

\noindent (3). {\it Example: the modular curve $X_0(37)$ over $\mathbb{Q}_{37}$. }\\
  The Magma computational algebra system \cite{MAG}, using the command {\tt SmallModularCurve(37)}, presents the
equation $\eta^2-\xi^3\eta=2\xi^5-5\xi^4+7\xi^3-6\xi^2+3\xi-1$ for $X_0(37)$. With $x=\xi-1, y=2\eta-\xi^3$
this is transformed into
\[y^2=g(x):=x^6+14x^5+35x^4+48x^3+35x^2+14x+1.\]
 Here $g(x)$ is, modulo $37$, equal to $(x-1)^2(x^2+8x+1)^2$. This implies that the branch points lie in the field
 $\mathbb{Q}_{37}(\sqrt{15})$  and have the form 
 \[1+\epsilon_1, 1+\epsilon_2, -4+\sqrt{15} +\epsilon_3,
   -4+\sqrt{15} +\epsilon_4, -4-\sqrt{15} +\epsilon_5, -4-\sqrt{15} +\epsilon_6,\]
 where all $|\epsilon_* |<1$. This is configuration (a). After a M\"{o}bius transformation, sending, say,  
 $1+\epsilon_1,-4+\sqrt{15}+\epsilon_3, -4-\sqrt{15}+\epsilon_5$ to  $0,1,\infty$, approximating the
 theta function $F(z)=\prod \frac{z-\gamma (0)}{z-\gamma (1)} $ by restricting to $\gamma$'s of length $\leq 1$ results in
fixed points approximated as \[0,1,\infty, \frac{\sqrt{15}}{10}(\epsilon_1-\epsilon_2), \frac{15+4\sqrt{15}}{60}(\epsilon_3-\epsilon_4), 
-\frac{60+16\sqrt{15}}{\epsilon_5-\epsilon_6}.\]

   \color{black}
   \end{remarks}
   
   \subsection{\rm Configuration (b)} 
   
In this case ``restricted'' (pointed) $\mathbb{B}$-tree is really a restriction and the analytic space of fixed points $Fix_{(b)}$ is given, in normalized form, by 
\[a_0=0, a_1=1, a_2=\infty, 0<|b_0|<1, |b_1|=|b_1-1|=|b_1+1|=1,1<|b_2|. \] 
The space $Branch_{(b)}$ of branch points is given by
\[a_0=0, a_1=1, a_2=\infty, 0<|b_0|<1, |b_1|=|b_1-1|=1<|b_2|.\] 
 The map $FB\colon Fix_{(b)}\rightarrow Branch_{(b)}$ from fixed points to branch points is  
 \[(b_0,b_1,b_2) \rightarrow  (\frac{F(b_0)}{F(b_0)-1}, \frac{F(b_1)}{F(b_1)-1}, \frac{F(b_2)}{F(b_2)-1}).\]
Approximation and the equality $|b_1+1|=1$, yield the following formulas:  
$\frac{F(b_0)}{F(b_0)-1}=\frac{4b_1}{b_1+1}b_0(1+\epsilon) $, $\frac{F(b_1)}{F(b_1)-1}=b_1^2 (1+\epsilon)$
and $\frac{F(b_2)}{F(b_2)-1}=\frac{b_1+1}{4}b_2(1+\epsilon)$.

\begin{proposition}\label{Prop8.3} The analytic map $Fix_{(b)}\rightarrow Branch_{(b)}$ is an unramified rigid analytic covering of degree two.
\end{proposition}
\begin{proof} Let be given the pairs of branch points $0,\tilde{b}_0$, $1,\tilde{b}_1$, $\infty, \tilde{b}_2$ in position (b), i.e.,
$|\tilde{b}_1-1|=1$ and $0<|\tilde{b}_0|<1=|\tilde{b}_1|<|\tilde{b}_2|$   (we note that this is the general situation). \\
 
 The initial equations for $b_0,b_1,b_2$ are $b_1^2=\tilde{b}_1,\ b_0=\frac{b_1+1}{4b_1}\tilde{b}_0 ,\ b_2=\frac{4}{b_1+1}\tilde{b}_2$.
Choose a solution $b_1$ of the first equation. Then $|b_1-1||b_1+1|=|\tilde{b}_1-1|=1$ and thus 
$|b_1-1|=|b_1+1|=1$. Further $|b_0|=|\tilde{b}_0|$ and $|b_2|=|\tilde{b}_2|$. Thus we find two initial solutions in space $Fix_{(b)}$. 
The stepwise refining of the initial solution is a unique and convergent process. These lead to actual solutions, since good 
position and convergence of the theta function $F$ is guaranteed by $|b_1+1|=|b_1-1|=|b_1|=1$.
\end{proof}

 \begin{remarks} $\ $ 

\vspace{\baselineskip}\noindent
(1). {\it The Rosenhain  normalization for case (b).}\\
 The fixed points are $\{0,b_0\}, \{a_1,b_1\}, \{1,\infty\}$ with $0<|b_0|<|b_1|=|a_1|<1$ and $a_1=Tb_1$ with 
 $|T|=|T-1|=|T+1|=1$.  This is the ``restricted'' case.
 The approximation of $F$ produces the formulas for the branch points \begin{footnotesize}
 \[\frac{F(b_0)}{F(b_0)-1}=\frac{16T}{(T+1)^2}b_0(1+\epsilon),\ \frac{F(a_1)}{F(a_1)-1}=\frac{4T}{T+1}a_1(1+\epsilon),
  \  \frac{F(b_1)}{F(b_1)-1}=\frac{4}{T+1}b_1(1+\epsilon).   \] \end{footnotesize}
 The branch points are  in position (b) and for the new
  $\tilde{T}=\frac{F(a_1)}{F(a_1)-1}(\frac{F(b_1)}{F(b_1)-1})^{-1}$ one has
 again $|\tilde{T}|= |\tilde{T}-1|= |\tilde{T}+1|=1$.  
 Write $b_0=B_0B_1, b_1=B_1,a_1=B_1T$. Then \[(B_0,B_1,T)\mapsto (\frac{16T}{(1+T)^2}B_0,\frac{4T}{1+T}B_1,T^2)\]
 is an approximation of $FB$.\\
 Thus the Rosenhain position produces, up to constants, the same formulas. \\
 
 \noindent (2). {\it Kadziela's example revisited, \cite{Kad1}}. \\
  The equation of this curve over $\mathbb{Q}_5$  is $y^2=x(x-1)(x-5)(x-195)(x-125)$. The branch points can be
  normalized to $0,25,1,39,1/5,1/25$ and this yields that $\tilde{b}_0=5^2(1+O(5)), \tilde{b}_1=4+O(5)$, and $\tilde{b}_2=1/5\cdot(1+O(5))$. 
  
  For the fixed points of the two generators $g_1,g_2$ of $\Gamma$, 
  there are two initial solutions, namely  \[b_0= 5^2\cdot(1+O(5)), \; b_1=2+O(5),\; b_2=1/5\cdot(3+O(5))\] 
and 
  \[b_0= 5^2\cdot(2+O(5)), \; b_1=3+O(5),\; b_2=1/5\cdot(1+O(5)).\] 
  The second choice for the initial solutions fits the data for the matrices $g_1,g_2$, computed in \cite{Kad1}.\\
  \end{remarks}

\subsection{\rm Configuration (c)}  Here the space of branch points $Branch_{(c)}$ is normalized to  $\{0,b_0\}, \{b_1,1\}$, $\{b_2,\infty\}$ with the inequalities $0<|b_0|<|b_1|<1<|b_2|<\infty $. Write $b_0=B_0B_1$, $  b_1=B_1$, $b_2=B_2^{-1}$.

 Then $ Branch_{(c)}=  \{   (B_0,B_1,B_2)\  |\  0<|B_j|<1\mbox{ for } j=0,1,2  \}$. All $\mathbb{B}$-trees of type (c) are ``restricted'' and thus we have $Fix_{(c)}=Branch_{(c)}$.  The morphism $FB\colon Fix_c\rightarrow Branch_c$ is approximated by the formula

$(B_0,B_1,B_2)\mapsto (4B_0(1+\epsilon),B_1^2(1+\epsilon),4B_2(1+\epsilon))$. This implies another special
case of Theorem~\ref{new5.8}:
\begin{proposition} \label{5.4} For configuration (c), one has $Fix_{(c)}=Branch _{(c)}$ and 
 $FB\colon Fix_{(c)}\rightarrow Branch_{(c)}$ is an unramified covering of degree two.
 \end{proposition}

\begin{observations} $\ $ \\
(1). {\it Rosenhain position for configuration (c)}. \\
Now  $a_0=0, a_2=1,b_2=\infty$ and $0<|b_0|<|a_1|<|b_1|<1<\infty$. One writes $b_0=B_0TB_1$,
 $a_1=TB_1$, $b_1=B_1$ with $0<|B_0|, |B_1|, |T|<1$. The space $Fix_{(c)}$ is identified with
 \[\{ (B_0,T,B_1)\  | \  0<|B_0|, |T|, |B_1|<1\}.\] The space $Branch_{(c)}$ has the same definition. 
 The morphism $FB\colon Fix_{(c)}\rightarrow Branch_{(c)}$ is approximated by    
\[\frac{F_2(b_0)}{F_2(b_0)-1}=16b_0T(1+\epsilon), \ \frac{F_2(a_1)}{F_2(a_1)-1}=4a_1T(1+\epsilon),
\ \frac{F_2(b_1)}{F_2(b_1)-1}=4b_1(1+\epsilon),\]
or, equivalently, by    $(B_0,T,B_1)\mapsto (4B_0,T^2,4B_1)$. 
As expected, this is similar to the first case (c) considered above.  \\

\noindent (2). {\it An example for configuration  (c) using the Rosenhain position.}\\
For a local field $K$ a set of branch points, represented by a $K$-valued tuple $(\tilde{B_0}, \tilde{T},\tilde{B_1})\in Branch_{(c)}$, 
is not the image of a $K$-valued tuple in $Fix_{(c)}$ if $\tilde{T}$ is not a square in $K$. The ramified extension 
$K(\sqrt{\tilde{T}})\supset K$ is needed for the two corresponding sets of fixed points.\\ 

\noindent {\it The curve $X$ over $\mathbb{Q}_3$ defined by $y^2=x(x-3^3)(x-3^2)(x-3)(x-1)$}.\\ 
Its branch points are in Rosenhain position and correspond to the point 
\[(\tilde{B_0}, \tilde{T},\tilde{B_1})=(3,3,3)\in Branch_{(c)}.\] A preimage in $Fix_{(c)}$ is approximated by 
$(3/4, \pm \sqrt{3}, 3/4)$. Thus the fixed points are defined over the ramified extension $\mathbb{Q}_3(\sqrt{3})$ of $\mathbb{Q}_3$.

The reduction of $\mathbb{P}^1_K$ with respect to the branch set $\{0,3^3,3^2,3,1,\infty \}$ 
consists of the four lines $L_0,L_1,L_2,L_3$ with nodes $L_0\cap L_1,L_1\cap L_2, L_2\cap L_3$. The images of  $\{0,3^3\},\{3^2\},\{3\},\{1,\infty\}$ are (in this order) on $L_0,L_1,L_2,L_3$ (see page~\pageref{picture2}, Figure~1(c)).

 In the degree 2 covering of the reduction (given by these branch points), one sees that the covering of the node $L_1\cap L_2$ 
is ramified.

 The node corresponds to the analytic subspace $\{z \ |\  3^{-2}<|z|<3^{-1}\}$ of $\mathbb{P}^1_K$.  The covering is given by the 
 local equation $V^2-3z=0$ and produces the space  $\{ V\ |\  3^{-3/2}<|V|<3^{-1}\}$. 
 This cannot be the subspace corresponding to a node in a semi-stable
(formal) model of $X$. It follows that for a semi-stable model of $X$ the extension $\mathbb{Q}_3(\sqrt{3})\supset \mathbb{Q}_3 $ is needed. \\
 
\noindent (3).  Let a branch locus $\xi\in Branch_{(NN)}$  for a genus two curve be given over the field $K$. Let 
$\eta \in Fix{(NN)}$ be a preimage of $\xi$ (i.e., a corresponding locus of fixed points).   For $(NN)=(a)$,
the  preimage $\eta$ is also defined over $K$. For $(NN)=(b)$, a preimage is, in general, defined over an unramified  quadratic extension 
$L\supset K$. Indeed, using the formulas of \S 8.2, one needs a square root of some $b_1$ with $|b_1|=1$.
In the last case $(NN)=(c)$, a preimage  $\eta$ is defined over a ramified quadratic extension of $K$.  Indeed, according to \S 8.3,
one needs a square root of some $B_1$ with $0<|B_1|<1$.\\

 Let $X\rightarrow \mathbb{P}^1$ denote the hyperelliptic curve corresponding to a branch locus $\xi$ defined over $K$.
 Let $\overline{X}$ denote the minimal semi-stable reduction of $X$ which separates the ramification points.
 Then $\overline{X}$ is defined over $K$ if $(NN)=(a)$. For $(NN)=(b)$ or $(NN)=(c)$, in general, a quadratic extension
 of $K$ is needed, which is unramified for $(NN)=(b)$ and ramified for $(NN)=(c)$.  This illustrates Comments \ref{newpairs} part (2).
\end{observations}

\section{ \rm  Genus three cases of Whittaker curves. }\label{explicit2}

   In this section details regarding the map $FB\colon Fix \rightarrow Branch $ are presented for some of the possible genus $3$ configurations.

\begin{figure}[h]\label{picture}
\centering

\tikzset{
    v/.style={circle, draw, minimum size=4.5mm, inner sep=0pt},
    num/.style={circle, draw, minimum size=4mm, inner sep=0pt}
}

\begin{tikzpicture}[remember picture]

\node (row1base) at (0,-1.6) {};

\node at (0,0) {
\begin{tikzpicture}
    \node[v] (a1) at (0,0) {};
    \node[v] (a2) at (1.1,0) {};
    \node[v] (a3) at (2.2,0) {};
    \node[v] (a4) at (3.3,0) {};
    \node[v] (a5) at (4.4,0) {};
    \node[v] (a6) at (5.5,0) {};

    \draw (a1) -- (a2) -- (a3) -- (a4) -- (a5) -- (a6);

    \node[below=1mm of a1] {2};
    \node[below=1mm of a2] {1};
    \node[below=1mm of a3] {1};
    \node[below=1mm of a4] {1};
    \node[below=1mm of a5] {1};
    \node[below=1mm of a6] {2};

    \node[num] at ($(a6 |- row1base) + (0.5,1.2)$) {1};
\end{tikzpicture}
};

\node at (7,0) {
\begin{tikzpicture}
    \node[v] (a1) at (0,0) {};
    \node[v] (a2) at (1.1,0) {};
    \node[v] (a3) at (2.2,0) {};
    \node[v] (a4) at (3.3,0) {};
    \node[v] (a5) at (4.4,0) {};

    \draw (a1) -- (a2) -- (a3) -- (a4) -- (a5);

    \node[below=1mm of a1] {2};
    \node[below=1mm of a2] {2};
    \node[below=1mm of a3] {1};
    \node[below=1mm of a4] {1};
    \node[below=1mm of a5] {2};

    \node[num] at ($(a5 |- row1base) + (0.5,1.3)$) {2};
\end{tikzpicture}
};
s
\node at (13,0) {
\begin{tikzpicture}
    \node[v] (a1) at (0,0) {};
    \node[v] (a2) at (1.1,0) {};
    \node[v] (a3) at (2.2,0) {};
    \node[v] (a4) at (3.3,0) {};

    \draw (a1) -- (a2) -- (a3) -- (a4);

    \node[below=1mm of a1] {2};
    \node[below=1mm of a2] {2};
    \node[below=1mm of a3] {2};
    \node[below=1mm of a4] {2};

    \node[num] at ($(a4 |- row1base) + (0.5,1.4)$) {3};
\end{tikzpicture}
};

\end{tikzpicture}

\vspace{1cm}

\begin{tikzpicture}

\node (row2base) at (0,-1.8) {};

\node at (0,0) {
\begin{tikzpicture}
    \node[v] (c) at (0,0) {};
    \node[v] (l1) at (-0.8,1.0) {};
    \node[v] (l2) at (-1.6,2.0) {};
    \node[v] (r1) at (0.8,1.0) {};
    \node[v] (r2) at (1.6,2.0) {};
    \node[v] (b) at (0,-1.2) {};

    \draw (c) -- (l1) -- (l2);
    \draw (c) -- (r1) -- (r2);
    \draw (c) -- (b);

    \node[left=1mm of l2] {2};
    \node[left=1mm of l1] {1};
    \node[right=1mm of r1] {1};
    \node[right=1mm of r2] {2};
    \node[right=1mm of b] {2};

    \node[num] at ($(b |- row2base) + (0.5,0)$) {4};
\end{tikzpicture}
};

\node at (4.5,0) {
\begin{tikzpicture}
    \node[v] (c) at (0,0) {};
    \node[v] (l1) at (-0.6,0.8) {};
    \node[v] (l2) at (-1.2,1.6) {};
    \node[v] (l3) at (-1.8,2.4) {};
    \node[v] (r) at (1.0,1.6) {};
    \node[v] (b) at (0,-1.2) {};

    \draw (c) -- (l1) -- (l2) -- (l3);
    \draw (c) -- (r);
    \draw (c) -- (b);

    \node[left=1mm of l3] {2};
    \node[left=1mm of l2] {1};
    \node[left=1mm of l1] {1};
    \node[right=1mm of r] {2};
    \node[left=1mm of b] {2};

    \node[num] at ($(b |- row2base) + (0.5,0)$) {5};
\end{tikzpicture}
};
\node at (9,0) {
\begin{tikzpicture}
    \node[v] (c)  at (0,0) {};
    \node[v] (l1) at (-0.6,0.8) {};
    \node[v] (l2) at (-1.2,1.6) {};
    \node[v] (r)  at (1.0,1.6) {};
    \node[v] (b)  at (0,-1.2) {};

    \draw (c) -- (l1) -- (l2);
    \draw (c) -- (r);
    \draw (c) -- (b);

    \node[left=1mm of l2] {2};
    \node[left=1mm of l1] {2};
    \node[right=1mm of r] {2};
    \node[right=1mm of b] {2};

    \node[num] at ($(b |- row2base) + (0.5,0)$) {6};
\end{tikzpicture}
};
\node at (13.5,0) {
\begin{tikzpicture}
    \node[v] (c)  at (0,0) {};
    \node[v] (l1) at (-0.5,1.0) {};
    \node[v] (l2) at (-1.0,2.0) {};
    \node[v] (r)  at (1.0,1.6) {};
    \node[v] (b)  at (0,-1.2) {};

    \draw (c) -- (l1) -- (l2);
    \draw (c) -- (r);
    \draw (c) -- (b);

    \node[left=1mm of l2] {2};   
    \node[left=1mm of l1] {1};   
    \node[right=1mm of r] {2};   
    \node[right=1mm of c] {1};   
    \node[left=1mm of b] {2};    

    \node[num] at ($(b |- row2base) + (0.5,0)$) {7};
\end{tikzpicture}
};

\end{tikzpicture}

\vspace{1cm}

\begin{tikzpicture}

\node (row3base) at (0,-1.8) {};

\node at (0,0) {
\begin{tikzpicture}
    \node[v] (c) at (0,0) {};
    \node[v] (l) at (-1.0,1.3) {};
    \node[v] (r) at (1.0,1.3) {};
    \node[v] (b) at (0,-1.2) {};

    \draw (c) -- (l);
    \draw (c) -- (r);
    \draw (c) -- (b);

    \node[left=1mm of l] {2};
    \node[right=1mm of r] {2};
    \node[right=1mm of c] {2};
    \node[left=1mm of b] {2};

    \node[num] at ($(b |- row3base) + (0.5,0)$) {8};
\end{tikzpicture}
};

\node at (6,0) {
\begin{tikzpicture}
    \node[v] (c1) at (0,0) {};
    \node[v] (c2) at (1.3,0) {};
    \node[v] (l1) at (-1.0,1.6) {};
    \node[v] (l2) at (-0.8,-1.4) {};
    \node[v] (r1) at (2.2,1.6) {};
    \node[v] (r2) at (2.2,-1.4) {};

    \draw (c1) -- (c2);
    \draw (c1) -- (l1);
    \draw (c1) -- (l2);
    \draw (c2) -- (r1);
    \draw (c2) -- (r2);

    \node[left=1mm of l1] {2};
    \node[left=1mm of l2] {2};
    \node[right=1mm of r1] {2};
    \node[right=1mm of r2] {2};

    \node[num] at ($(r2 |- row3base) + (0.5,0)$) {9};
\end{tikzpicture}
};

\node at (12,0) {
\begin{tikzpicture}
    \node[v] (c)  at (0,0) {};
    \node[v] (ul) at (-1.3,1.3) {};
    \node[v] (ur) at (1.3,1.3) {};
    \node[v] (ll) at (-1.2,-1.2) {};
    \node[v] (lr) at (1.2,-1.2) {};

    \draw (c) -- (ul);
    \draw (c) -- (ur);
    \draw (c) -- (ll);
    \draw (c) -- (lr);

    \node[left=1mm of ul] {2};
    \node[right=1mm of ur] {2};
    \node[left=1mm of ll] {2};
    \node[right=1mm of lr] {2};

    \node[num] at ($(lr |- row3base) + (0.5,0)$) {10};
\end{tikzpicture}
};

\end{tikzpicture}

\caption{All configurations of genus three. The number of branch points that are mapped to a vertex is indicated by
$\emptyset$, 1 or 2.} 
\end{figure}
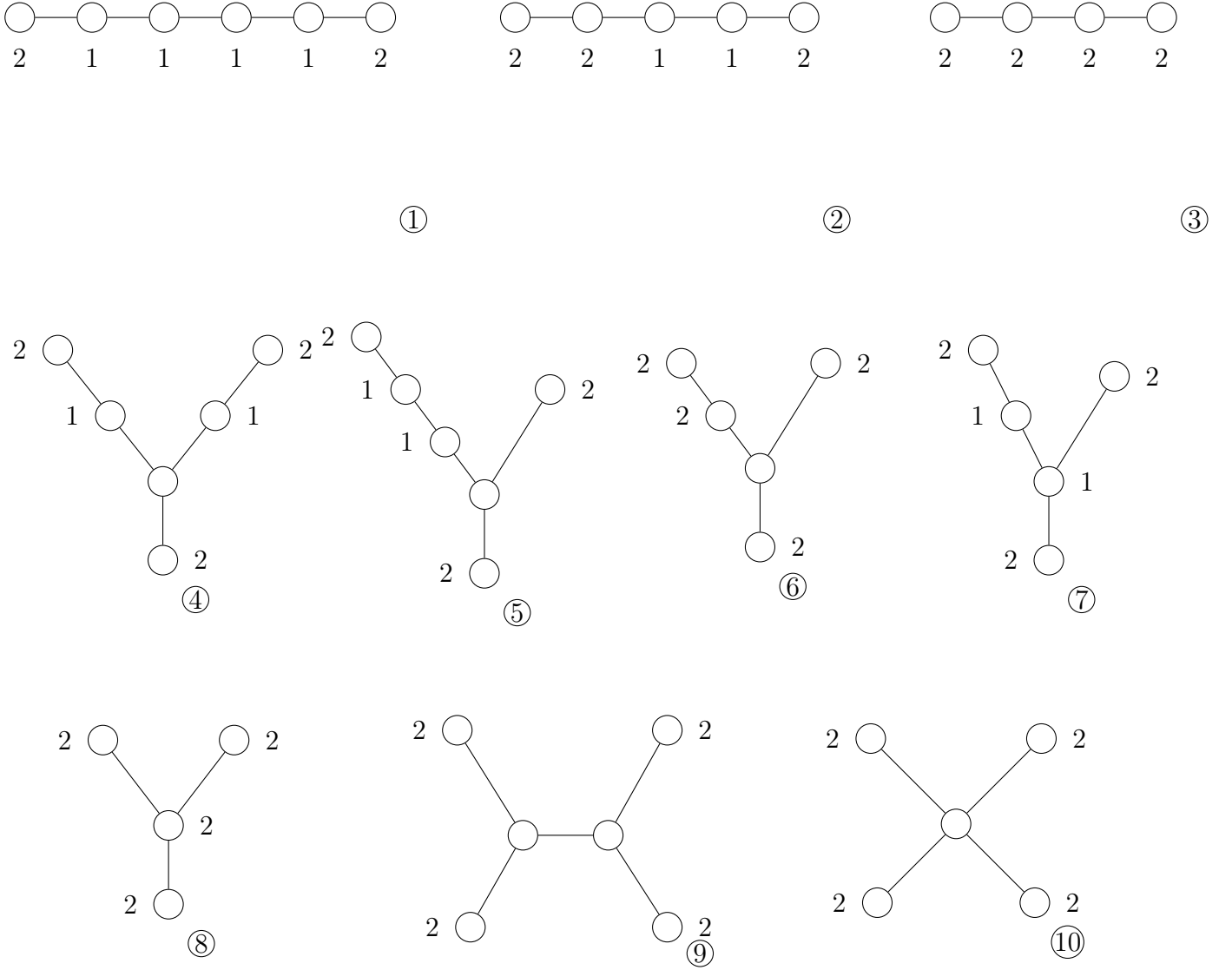

\subsection{\rm The curve $X_0(39)$ over $\mathbb{Q}_3$ and closed disks configurations}

 The modular curve $X_0(39)$ is hyperelliptic. An explicit equation is obtained
using the MAGMA command {\tt SmallModularCurve(39);} or from \cite[p.~22]{Kenku}. Over the field $\mathbb{Q}_3(\sqrt{-1},\sqrt{3})$ it can
be put in the form
 \[y^2=(x-a_0)(x-b_0)\cdot (x-a_1)(x-b_1)\cdot  (x-a_2)(x-b_2)\cdot (x-a_3)(x-b_3)\]  where eight branch points have absolute
 value 1 and are mapped pairwise to the points $1,2,\lambda +1$, $2\lambda +2\in \mathbb{F}_9$ where 
 $\lambda$ denotes a root of $X^2+2X+2=0$.
 
 From this position of the branch points  it follows that the stable reduction of $X_0(39)$ for $p=3$ (over the field 
 $\mathbb{Q}_3(\sqrt{-1}, \sqrt{3})$) consists of two projective lines over $\mathbb{F}_9$, intersecting normally in 4~points. \\
 
 The configuration of this example, call it $(a,g=3)$, is the one of \S\ref{5.3.1} with
  $g=3$ (see Figure~1(a), page~\pageref{picture2}) and number 10 in Figure~2 above.  The space $Branch_{(a, g=3)}$ 
can~be~described~as~follows.\\  
 
  The branch data consists of tuples $\{a_0,b_0\},\cdots ,\{a_3,b_3\}$ modulo the simultaneous action of ${\rm PGL}_2$. The tuples
  are normalized by imposing the conditions $a_0=0, a_1=1, a_3=\infty$. The rigid space 
  $Branch_{(a,g=3)}$ consists of the tuples $(B_0,T_1,T_2,A,B_3)\in K^5$ given by the (in)equalities
 \[0< |B_0|, |T_1|, |T_2|, |B_3|<1$ and $|A|=|A-1|=1 \mbox{ and the branch points are }\] 
 \[{a}_0=0, {b}_0=B_0 ,  {a}_1=1, {b}_1=1+T_1 , {a}_2=A, {b}_2=A+T_2 , {a}_3=\infty, {b}_3=B_3^{-1}. \]
 In this case $Fix_{(a,g=3)}=Branch_{(a,g=3)}$. 
 The morphism $FB\colon Fix_{(a, g=3)}\rightarrow Branch_{(a, g=3)}$ has the form: $x \mapsto \frac{F(x)}{F(x)-1}$ for all 
 $x\in \{a_0,\dots ,b_3\}$, where $F$ is the standard theta function   
 $\prod_{\gamma \in \Gamma}\frac{z-\gamma(0)}{z-\gamma(1)}$. We approximate $F$ by taking the product over all $\gamma$
 of length $\leq 2$ and translate this in terms of the coordinates of $Branch_{(a,g=3)}$. One obtains
   \[ FB (B_0,T_1,T_2,A,B_3)= (2B_0(1+\epsilon), T_1(1+\epsilon), T_2(1+\epsilon), A(1+\epsilon), 2B_3(1+\epsilon)) \]
where  the ``error term'' $(1+\epsilon)$ satisfies  $|\epsilon|\leq \max (|B_0|, |T_1|,|T_2|, |B_3|).$

  This implies that $FB\colon Fix_{(a, g=3)}\rightarrow Branch_{(a, g=3)}$ is bijective.
   For the above example $X_0(39)$ a unique set of fixed points
  for the Whittaker group can be calculated by approximating $F$ by suitable finite products. 
We refer to \cite[Example~8.4]{KMMP} for a further discussion of this. \\ \color{black}

  \begin{proposition}\label{9.1} Let $(P,m)$ denote a configuration of the branch locus consisting of pairs of points $\{\{a_i,b_i\}\}_{i=0}^g$
  in closed disk position, i.e., there are disjoint closed disk $\{D_i\}_{i=0}^g$ with $a_i,b_i\in D_i$ for all $i$.
  
    Then $Fix_{(P,m)}=Branch_{(P,m)}$ and $FB\colon Fix_{(P,m)}\rightarrow Branch_{(P,m)}$ is an isomorphism.
    \end{proposition}
 The curve $X_0(39)$ over $\mathbb{Q}_{3}(\sqrt{-1},\sqrt{3})$ is a case where the fixed points are in ``closed disk position''. 
 Proposition~\ref{9.1} is an immediate consequence of Theorem~\ref{new5.8}.  We note that \cite{Kad1} also provides an indication of a proof
of Proposition~\ref{9.1}.
 
 \subsection{\rm Configuration $(b, g=3)$}
 This configuration of branch points, denoted by $(b,g=3)$,  is given by a tree with vertices $v_0, v_1,v_2,v_3$ and the edges
 $[v_0, v_1], [v_0,v_2], [v_0, v_3]$.  The branch points  $a_i,b_i,\ i=0,\dots , 3$ are mapped to the vertex $v_i$. 
 This is the situation (5.2.2) on page~\pageref{example5.2.2}, with $g=3$; also given as number 8 in Figure~2 on page~\pageref{picture}.
  The space $Branch_{(b,g=3)}$ is obtained by the normalization
 \[ \begin{array}{l}a_0=0, b_0=\infty, a_1=1, b_1=1+B_1, a_2=T_2,b_2=T_2+B_2, a_3=T_3, b_3=T_3+B_3\\
 \mbox{with } 0<|B_1|, |B_2|, |B_3| <1 \mbox{ and } |T_2|=|T_3|=1 \mbox{ and the images of  $1, T_2,T_3$ }\\
 \textrm{in the residue field $k$ are distinct (i.e., } |(T_2-1)(T_3-1)(T_2-T_3)|=1 ).  \end{array}\]
  The subspace $Fix_{(b,g=3)}$ of $Branch_{(b,g=3)}$ is given by the {\it restrictions}\\
  $|(T_2^2-1)(T_3^2-1)(T_2^2-T_3^2)|=1$. The map $FB\colon Fix_{(b,g=3)} \rightarrow Branch_{(b,g=3)}$ 
  is obtained  by applying $FF$ to the fixed points. Here $FF(z):=\frac{F(z)}{F(z)-1}$ and $F(z)=\prod _{\gamma \in \Gamma} \frac{z-\gamma(0)}{z-\gamma (1)}$.

 Let $F^*$ denote the 
product above, taken over all $\gamma$ of length $\leq 2$. The normalized eight fixed points
$(0,1,\infty, 1+B_1,T_2,T_2+B_2,T_3,T_3+B_3)$ are mapped by $F^*$ to the set of points
$(0,\infty, 1,F^*(1+B_1),\dots ,F^*(T_3+B_3))$. Now normalizing by applying the map $t\mapsto \frac{t}{t-1}$
 results in a formula, approximating $FB$, namely
 \[ (B_1,B_2,B_3,T_2,T_3) \mapsto (4B_1,4T_2B_2, 4T_3B_3, T_2^2,T_3^2).  \] 
This formula confirms in an explicit way a special case of Theorem~\ref{new5.8}:
  \begin{proposition} $FB\colon Fix_{(b,g=3)}\rightarrow Branch_{(b,g=3)}$ is an unramified  rigid analytic Galois covering with
 Galois group $(\mathbb{Z}/2\mathbb{Z})^2$.
 \end{proposition}

\subsection{\rm Another configuration with genus three}

Consider the configuration of the eight branch points $\{a_0,b_0\ldots , a_3,b_3\}$ such that $a_0=0$, $a_3=1$, $b_3=\infty$ and it holds that
 the absolute values of $0,b_0,a_1,b_1,a_2,b_2,1$ are strictly increasing. This is number (1) in Figure~2 on page~\pageref{picture}.
 Call this configuration $(cc,g=3)$. 

 The space $Branch_{(cc,g=3)}$  is the set of tuples  $(B_0,A_1,B_1,A_2,B_2)\in K^5$ satisfying the conditions  $0<|B_0|, |A_1|, |B_1| ,|A_2| ,|B_2| <1$.
The other branch points are defined by
 \[ b_0=B_0A_1B_1A_2B_2; a_1=A_1B_1A_2B_2; b_1=B_1A_2B_2;a_2=A_2B_2; b_2=B_2.\]
 Now $Fix_{(cc,g=3)}=Branch_{(cc, g=3)}$ and  $FB\colon Fix_{(cc,g=3)}\rightarrow Branch _{(cc,g=3)}$,
   approximated as before, produces the map
 \[(B_0,A_1,B_1,A_2,B_2)\mapsto (4B_0(1+\epsilon) ,A_1^2(1+\epsilon),4B_1(1+\epsilon),A_2^2(1+\epsilon), 4B_2(1+\epsilon)),  \]
  One concludes the following.
  \begin{proposition} $FB\colon Fix_{(cc,g=3)}\rightarrow Branch_{(cc, g=3)}$ is a rigid analyic  unramified  Galois covering with
   group $(\mathbb{Z}/2\mathbb{Z})^2$.
   \end{proposition}

 \subsection{\rm Again a configuration with genus three}

 The configuration number (4) of Figure~2 on page~\pageref{picture}, called $X$ here, has branch space $Branch_X$, defined  in normalized
  form by the sequence of fixed points
 $0,b_0,a_1,a_2,b_2,b_1,1,\infty$ with \[0<|b_0|<|a_1|<|a_2|=|b_2|<|b_1|<1<\infty \mbox{ and is parametrized}\]
  by the variables $B_0,B_1,B_2,A,T$, all satisfying $0<| *|<1$ and  with the following equalities 
 \[ b_0=B_0B_1B_2A;\ a_1=B_1B_2A;\ a_2=B_1B_2(1+T);\ b_2=B_2B_1; b_1=B_1.\]
 The space of fixed points $Fix_X$ coincides with $Branch_X$ since ``restriction'' is no extra condition for this configuration. 
 The theta function $F$ is approximated by $F^*$, the product over the five terms with $\gamma$ of length $\leq 1$.
 The fixed points $0,1,\infty, b_0,a_1,a_2,b_2,b_1$ are send by $F^*$ to, approximately, $0,\infty,1$ and
 \begin{small}
 \[\frac{-64(1+T)}{(2+T)^2}B_0B_1B_2^2A^2,\frac{-16(1+T)}{(T+2)^2}B_1B_2^2A^2, \frac{8(1+T)^3}{2+T}B_1B_2^2,
 \frac{8}{2+T}B_1B_2^2, -4B_1.\] \end{small}
 One applies the transformation $t\mapsto \frac{t}{t-1}$ to these data in order to obtain the same normalization of $Branch_X$.
 Then the map $FB\colon Fix_X\rightarrow Branch_X$ is approximated by the following
 \[( B_0,B_1,B_2,A,T)\mapsto (4B_0,-4B_1, \frac{-2}{2+T}B_2^2, -2(1+T)A^2, (1+T)^3-1-T).\]
 This leads to the result
 \begin{proposition}
 The morphism $FB\colon Fix_X\rightarrow Branch_X$ is a rigid analytic unramified Galois covering with Galois group
 $(\mathbb{Z}/2\mathbb{Z})^2$. 
  \end{proposition}
 
 \subsection*{Acknowledgement} We thank Enis Kaya, Mark Masdeu and Steffen M\"{u}ller for their interest in this work and for various useful suggestions,
and Arnold Meijster who, aided by ChatGPT, transformed our original drawings into the figures presented on the pages~\pageref{picture2} and \pageref{picture}. No generative AI was used for anything else in the paper.

 \begin{small}

\end{small}
 \end{document}